\documentclass[12pt]{amsart}  

\usepackage{bm}
\usepackage{bbold}
\usepackage{amscd} 
\usepackage{amsmath,amssymb}
\usepackage{mathtools}
\usepackage{mathrsfs}
\newtheorem{theorem}{Theorem}[section]
\newtheorem{lemma}[theorem]{Lemma}
\newtheorem{corollary}[theorem]{Corollary} 
\theoremstyle{definition}

\newtheorem{remark}[theorem]{Remark} 
\numberwithin{equation}{section}

\def\B{\mathscr B}

\def\L{\mathsf L}

\def\fb{\mathfrak b}

\def\C{\mathcal C}
\def\dom{\text{\rm dom}}
\def\ran{\text{\rm ran}}

\def\RE{\mathbb R}
\def\CO{{\mathbb C}}

\def\ph*{\phi_\star}

\def\be{\begin{equation}}
\def\ee{\end{equation}}
\def\min{{\rm min}}
\def\max{{\rm max}}

\def\-{{\rm in}}
\def\+{{\rm ex}}

\def\B {{\mathscr B}}

\def\SL{S\!L}
\def\DL{D\!L}
\def\DN{D\!N}
\def\ve{\varepsilon}
\def\sb{{\mathsf B}}

\title[Resolvent, spectrum and resonances for the acoustic operator]{Resolvent, spectrum and resonances for the acoustic operator with piecewise constant coefficients}
\author{A. Mantile, A. Posilicano}
\address{Laboratoire de Math\'{e}matiques de Reims, 
Universit\'{e} de Reims Champagne-Ardenne, Moulin de la Housse BP 1039, 51687
Reims, France}
\address{DiSAT, Sezione di Matematica, Universit\`a dell'Insubria, via Valleggio 11, I-22100
Como, Italy}
\email{andrea.mantile@univ-reims.fr}
\email{andrea.posilicano@unisubria.it}

\begin{document}

\begin{abstract} We study the acoustic operator $A_{v,\rho }:=v^{2}\rho\nabla\!\cdot\rho^{-1}\nabla$  with transmission conditions at the  boundary of $\Omega=\Omega_{1}\cup\dots\cup\Omega_{n}$, where the $\Omega_{\ell}$'s are connected disjoint open bounded Lipschitz domains, the positive functions $v$ and $\rho$ are constant on each connected component of $\Omega$ and $v=\rho=1$ on $\RE^{3}\backslash\overline\Omega$.
Through a formula for the resolvents difference $(-A_{v,\rho }+z)^{-1}-(-\Delta+z)^{-1}$, we provide a Limiting Absorption Principle, determine the spectrum, which turns out to be purely absolutely continuous, and, in the case the connected components of $\Omega$ are of class $\C^{1,\alpha}$, characterize the resonance set. The second part of the paper is devoted to the case where $\Omega=\Omega(\ve)$ is connected with a small size $\varepsilon$ and the $\varepsilon$-analytic functions $v=v(\varepsilon)$ and/or $\rho=\rho(\varepsilon)$ converge to $0_{+}$ inside $\Omega(\varepsilon)$ as $\varepsilon\downarrow 0$; there, we provide the analytic $\varepsilon$-expansions of the resonances of $A_{v,\rho }$ according to different choices of the rate of convergence towards zero of the material parameters.   
\end{abstract}

\maketitle

\section{Introduction}
We consider acoustic operators in  $L^{2}(\RE^{3})$%
\begin{equation}\label{A_full_def}
A_{v,\rho}:=v^{2}\!\rho\nabla\!\!\cdot\!\rho^{-1}\nabla
\end{equation}
whose material coefficients $v$ and $\rho$, defining the speed of propagation of the waves and the density of the medium respectively, are piecewise constants and positive-valued
functions. More precisely, given $\Omega_{1},\dots,\Omega_{n}$ open, connected and bounded Lipschitz domains in $\RE^{3}$ such that 
$$
\overline\Omega_{i}\cap\overline\Omega_{j}=\varnothing\,,\quad i\not=j\,,
$$ 
and given the strictly positive constants $\rho_{1},\dots,\rho_{n}$ and $v_{1},\dots, v_{n}$, the material coefficients are 
\be\label{mat}
v:= \chi_{\Omega_{\+}}+\sum_{\ell=1}^{n} v_{\ell}\,\chi_{\Omega_{\ell}}\,,\qquad
\rho:=\chi_{\Omega_{\+}}+ \sum_{\ell=1}^{n} \rho_{\ell}\,\chi_{\Omega_{\ell}}\,,
\ee
where $\Omega_{\+}$ is the connected domain $\Omega_{\+}:=\mathbb{R}^{3}\backslash\overline{\Omega}$ with $\Omega\equiv\Omega_{\-}:=\Omega_{1}\cup\dots\cup \Omega_{n}$; $\chi_{S}$ denotes the characteristic function of a set $S$.
\par
At the interfaces $\Gamma_{\ell}$'s given by the boundaries of the $\Omega_{\ell}$'s, we require the transmission conditions 
\begin{equation}
\gamma_{0,\ell}^{\mathrm{ex}}u=\gamma_{0,\ell}^{\mathrm{in}}u\,,\qquad
\rho_{\ell}\,\gamma_{1,\ell}^{\mathrm{ex}}u=\gamma_{1,\ell}^{\mathrm{in}%
}u\,,\qquad\ell=1,...n\,,\label{bc}%
\end{equation}
where $\gamma_{0,\ell}^{\-/\+} $ and $\gamma_{1,\ell}^{\-/\+}$ denote the lateral Dirichlet and Neumann traces (we refer to Sections 2 and 3 for the more detailed definitions in the appropriate functional spaces).  Such boundary conditions ensure the symmetry of $A_{v,\rho}$ in the weighted Hilbert space (equivalent, in the Banach space sense, to $L^{2}(\RE^{3})$) given by 
$L^{2}(\mathbb{R}^{3},b^{-1})  $, where $b:=v^{2}\!\rho=\chi_{\Omega_{\+}}%
+\sum_{\ell=1}^{n}v_{\ell}^{2}\rho_{\ell}\,\chi_{\Omega_{\ell}}$. Further, see Theorem \ref{acoustic}, $A_{v,\rho}$ turn out to be, on an appropriate domain, a not positive self-adjoint operator in $L^{2}(\mathbb{R}^{3},b^{-1})$. This latter property entails the existence and uniqueness of the solutions of the Cauchy problem for the wave equation  
\be\label{wave}
\frac{\partial^{2} u }{\partial{t}^{2}}=A_{v,\rho}u
\ee
which models the dynamics of the acoustic waves in a medium with bulk modulus $b$. Indeed, by, e.g., \cite[Proposition 3.14.4, Corollary 3.14.8 and Example 3.14.16]{Arendt}, the unique (mild) solution of the Cauchy problem for \eqref{wave} with initial data $u_{0}$ and $\dot u_{0}$ in $L^{2}(\RE)$ is given by 
$$
u(t)=\cos(t(-A_{v,\rho})^{1/2})\, u_{0}+(-A_{v,\rho})^{-1/2}\sin(t(-A_{v,\rho})^{1/2})\,\dot u_{0}\,,
$$ 
where the sine and cosine operator  functions are defined through the functional calculus for the self-adjoint $A_{v,\rho}$. Moreover, by the inversion of the Laplace-Stieltjes transform for operator valued functions (see, e.g., \cite[Theorem 2.3.4]{Arendt}), the above solution can be expressed in terms of the resolvent $(-A_{v,\rho}-\kappa^{2})^{-1}$. Such a resolvent, or better, a formula for the resolvent difference $(-A_{v,\rho}-\kappa^{2})^{-1}-(-\Delta-\kappa^{2})^{-1}$ turns out to be our main tool in the study of the spectrum and resonances. This difference formula is obtained in Section 4 proceeding as follows: the heuristic
re-writing
\[
A_{v,\rho}=v^{2}\rho\rho^{-1}\Delta u+v^{2}\rho\nabla\rho^{-1}\!\cdot\!\nabla
u=\Delta u+({v^{2}}-1)\chi_{\Omega}\Delta+v^{2}\rho\nabla\rho^{-1}%
\!\cdot\!\nabla u
\]
suggests that $A_{v,\rho}$ correspond to the free Laplacian plus a
perturbation made of a regular component, supported on $\Omega$, and a
singular one (due to the jumps of $\rho^{-1}$), supported on its boundary $\Gamma$. This
picture emerges more precisely in a distributional framework; using the Green identities
and the boundary conditions, one gets 
\begin{equation}
A_{v,\rho}u={\Delta}u+(  v^{2}-1)  1_{\Omega}^{\ast}\Delta_{\Omega}%
^{\text{\textrm{max}}}1_{\Omega}u+2\sum_{\ell=1}^{n}
\frac{\rho_{\ell}-1}{\rho_{\ell}+1}\,\gamma_{0,\ell}^{\ast}\,\gamma_{1,\ell}u  \,,
\label{A_weak}%
\end{equation}
where the $\gamma_{0,\ell}^{*}$'s denote the distributions supported on $\Gamma_{\ell}$ corresponding to the duals of the full Dirichlet trace operators (we refer to Section 2 for more notation and definitions). By \eqref{A_weak},  $A_{v,\rho}$ belongs to the class of operators of the kind 
\be\label{kind}
\Delta+B_{1}+(\gamma_{0}^{*}\oplus\gamma_{1}^{*})B_{2}(\gamma_{0}\oplus\gamma_{1})
\ee
where $\gamma_{0}$ and $\gamma_{1}$ denote the full Dirichlet and Neumann trace operators for the boundary $\Gamma=\Gamma_{1}\cup\dots\cup \Gamma_{n}$ and $B_{1}$ and $B_{2}$ are bounded operator in suitable function spaces. Operators of the kind \eqref{kind}, in the case $B_{1}$ and $B_{2}$ are symmetric (this is not the case for $A_{v,\rho}$), were built and studied in depth in \cite{JST}  on the basis of the abstract trace maps method for singular perturbations introduced in \cite{JFA} and developed in \cite{MaPo JMPA 19}; this method produces self-adjoint operators $A_{\mathsf B}$ which are mixed (i.e., regular plus singular) perturbations of a pivot self-adjoint
operator $A_{0}$ (see \cite[Section 2]{JST}. Let us notice that in the present case, contrarily to the ones studied in \cite{JST}, $A_{v,\rho}$ and the free Laplacian $\Delta$ are self-adjoint in different Hilbert spaces. In \cite{JST}, operators $A_{\mathsf B}$ are built through a resolvent formula of the kind
$$
(-A_{\sb}-\kappa^{2})^{-1}=(-A_{0}-\kappa^{2})^{-1}
+G_{\kappa}Q_{\kappa}^{-1}(B_{1}\oplus B_{2})
G_{-\bar\kappa}^{*}\,,
$$
where 
$$
Q_{\kappa}:=1-(B_{1}\oplus B_{2})\tau G_{\kappa}\,.
$$
The $Q$-function $\kappa\mapsto Q_{\kappa}$ has a key role
in \cite{JST}: it encodes all the spectral properties of the
perturbation, it is naturally involved in the derivation of the Limiting
Absorption Principle (LAP for short) and in the proof of the asymptotic completeness of the
wave operators for the scattering couple $(A_{\sb},A_{0})$. 
In particular, an explicit formula for the scattering
amplitude (see \cite[Theorem 3.10]{JST}) allows to
identify the resonances of $A_{\sb}$ with the poles in unphysical sheet Im$(\kappa)<0$ of the meromorphic map $\kappa\mapsto Q_{\kappa}^{-1}$.
\par
In the first part of our work, we adapt the construction in \cite{JST}
to the case of the acoustic operator $A_{v,\rho}$ treated as a mixed perturbation of the free Laplacian $\Delta$. The main
outcome is the definition of the acoustic $Q$-function
\be\label{acQ}
Q_{\kappa}=\begin{bmatrix}v^{2}+({v^{2}}-1)\kappa^{2}N_{\kappa}&({v^{2}}-1)\kappa^{2}1_{\Omega}\SL_{\kappa}\\-\widetilde\rho\,\gamma^{\-}_{1}N_{\kappa}&1-\widetilde\rho\,\gamma_{1}\SL_{\kappa}
\end{bmatrix}\,,\qquad \widetilde{\rho}:=2
\sum_{\ell=1}^{n}
\frac{\rho_{\ell}-1}{\rho_{\ell}+1}\,\chi_{\Gamma_{\ell}}
\ee
which is used to build a formula for the
resolvents difference $(-A_{v,\rho}-\kappa^{2})^{-1}-(-\Delta-\kappa^{2})^{-1}$, 
holding for any $\operatorname{Im}\kappa>0$ (see Theorems \ref{equal} and
\ref{RD}). This provides a
direct framework for the study of the resonances. By the mapping properties of traces and by the abstract results from \cite{Ren1}, a Limiting Absorption Principle for the acoustic resolvent is
derived in Theorem \ref{Theorem_LAP}. The continuity of the map $(v,\rho)\mapsto A_{v,\rho}$  is discussed in Section 6.\par
For the convenience of the reader, we summarize the main results of the first part of the paper in the following Theorem, referring to Sections 3 to 6 for more details and the proofs and to Section 2 for the definitions of the involved functional spaces and operators.
\begin{theorem}
\label{Part I}$\left.  {}\right.  \medskip$\newline$i)$ The self-adjoint representation in $L^{2}(\mathbb{R}^{3},{b}^{-1})$ of the 
acoustic operator is given by 
$$A_{v,\rho}=(\Delta_{\Omega_{\+}}^{\mathrm{max}}\oplus v^{2}\Delta_{\Omega }^{\mathrm{max}})|{\mathscr D}_{\rho}
\,,
$$
where 
\[
{\mathscr D}_{\rho}
:=\left\{  u\in H^{1}(\mathbb{R}^{3})\cap L^{2}_{\Delta}(\mathbb{R}^{3}\backslash\Gamma)\,:\, \rho_{\ell}\,\gamma
_{1,\ell}^{\mathrm{ex}}u=\gamma_{1,\ell}^{\mathrm{in}}u\,,\ \ell
=1,\dots,n\right\}\,.
\]
Its spectrum is 
$$
\sigma(  A_{v,\rho})  =\sigma_{\mathrm{ac}}(A_{v,\rho})  =(-\infty,0]\,.
$$\newline$ii)$ For any
$\kappa\in{\mathbb{C}}_{+}$ there holds%
\begin{equation}
(-A_{v,\rho}-\kappa^{2})^{-1}-(-\Delta-\kappa^{2})^{-1}=%
\begin{bmatrix}
R_{\kappa}1_{\Omega}^{\ast} & S\!L_{\kappa}%
\end{bmatrix}Q_{\kappa}^{-1}%
\begin{bmatrix}
({v^{2}}-1)\Delta_{\Omega}^{\mathrm{max}}1_{\Omega}R_{\kappa}\\
\widetilde{\rho}\,\gamma_{1}R_{\kappa}%
\end{bmatrix},
\, \label{A_Krein}%
\end{equation}
where $Q_{\kappa}$ defined in \eqref{acQ}.
\medskip\newline
$iii)$ (LAP) The limits%
\[
\lim_{\delta\searrow0}(-A_{v,\rho}+\lambda\pm i\delta)^{-1}\,,
\]
exist in ${\B}(L_{\alpha}^{2}(\mathbb{R}^{3}),L_{-\alpha}^{2}%
(\mathbb{R}^{3}))$ for any $\alpha>1/2$ and any $\lambda\in\mathbb{R}%
\backslash\{0\}$; the same hold true in the case $\lambda=0$ whenever
$\alpha>3/2$. The  two ${\B}(L_{\alpha}^{2}(\mathbb{R}^{3}),L_{-\alpha}^{2}(\mathbb{R}^{3}))$-valued extended resolvents defined through such limits are continuous  maps  on $\CO_{+}\cup\RE$ and on $\CO_{-}\cup\RE$ respectively.
\medskip\newline
$iv)$ If, for any $1\le \ell\le n$, the sequence $\{v_{\ell}(j),\rho_{\ell}(j)\}_{j=1}^{+\infty} \subset (0,+\infty)\times(0,+\infty)$ converges to $(v_{\ell},\rho_{\ell})\in
(0,+\infty)\times(0,+\infty)$ as $j\to+\infty$, then $\{A_{v(j),\rho(j)}\}_{j=1}^{+\infty}$ converges in norm resolvent sense to  $A_{v,\rho}$. If $\rho_{\ell}\to+\infty$, then the norm resolvent limit still exits as a well defined closed operator $A_{v,+\infty}$.
\end{theorem}
The second part of the paper is devoted to resonances. The resonance problem for acoustic operators in  discontinuous materials has
been addressed in previous works where the scattering resonances are defined
in terms of a transmission problem whose solutions are the quasi-modes of the
system, see \cite{MoSpan M3AS 19} and references therein. The common approach
relies on the microlocal analysis of the transmission problem allowing to
construct high-energy quasimodes and the asymptotic characterization the
corresponding resonances, which is achieved by the quasi-modes to resonances
result, see e.g. in \cite[Proposition 2.1]{PoVo CMP 99}. When $\Omega$ is
strictly convex with a smooth boundary, the acoustic operators
\eqref{A_full_def} fit into the scheme of such works. In particular, when
$v^{2}\rho\in\left(  0,1\right)  $ \cite[Theorem 1.1]{PoVo CMP 99} predicts
the existence of an infinite sequence of resonances $\kappa_{j}^{2}$ tending rapidly to the real axis. 
In the opposite case, $v^{2}\rho>1$, there exists a resonance-free region of the form
$\{  \kappa^{2}:|\operatorname{Re}(\kappa)|
>C_{1}\,,\ \operatorname{Im}(\kappa)\leq-C_{2}/| \kappa|\}  $, see \cite[Theorem 1.1]{PoVo AA 99}. Such results have been
extended to the case of Lipschitz star-shaped obstacles, \cite{MoSpan M3AS
19}.
\par
We elaborate a direct approach to resonances  which
relies on the resolvent difference formula \eqref{A_Krein}. The framework is provided
in Section 7, where, under the hypothesis that $\Omega$ is of class $\C^{1,\alpha}$, it is shown that $\kappa\mapsto Q_{\kappa
}^{-1}$ is an operator-valued meromorphic map with poles of finite rank 
$\kappa_{\circ}\in{\mathbb{C}%
}\backslash{\mathbb{C}}_{+}$ fulfilling the condition $\ker(Q_{\kappa_{\circ}%
})\not =\{0\}$. Then, in view of the results of \cite{JST}, where the
inverse $Q$-function determines the scattering amplitude, we are lead to identify the
resonances as those $\kappa_{\circ}^{2}$ such that $\operatorname{Im}%
\kappa_{\circ}<0$ and $\ker(Q_{\kappa_{\circ}})\not =\{0\}$. This definition turns out to be equivalent to the ones given by the blackbox formalism, see \cite{SjoZwo}, \cite[Chapter
4]{DZ}. Indeed, from (\ref{A_Krein}) there follows that
$\kappa\mapsto(  -A_{v,\rho}-\kappa^{2})  ^{-1}$, outside the poles of 
$\kappa\mapsto Q_{\kappa}^{-1}$, is analytic as a map from
$L_{comp}^{2}(  \mathbb{R}^{3})$ to $L_{loc}^{2}(\mathbb{R}^{3})  $. The link between these two equivalent definitions of resonances has been exploited in \cite{Galk} as regards $\delta$ and $\delta'$
purely singular perturbation supported on a boundary. Furthermore, by Lemma \ref{res}, the poles of $\kappa\mapsto Q_{\kappa}^{-1}$ can not lie on the real line and so our definition of the resonance set simplifies to
\be\label{res-def}
\text{$\kappa_{\circ}^{2}\in\CO$ is a resonance of $A_{v,\rho}$}\quad\Leftrightarrow\quad  \ker
(Q_{\kappa_{\circ}})\not =\{0\}\,.
\ee
A scale-dependent material design is introduced in Sections 8 where it is assumed
\begin{equation}
\Omega_{\ell}(\ve)  =y_{\ell}+\ve\Omega_{\circ}  \,,
\qquad\ve\Omega_{\circ}:=\{\varepsilon
x,\ x\in\Omega_{\circ}\}\,,\qquad\Omega_{i}(\varepsilon)\cap\Omega
_{j}(\varepsilon)=\varnothing\text{ if }i\not =j\,,\label{microres}%
\end{equation}
$\Omega_{\circ}$ is open, bounded, connected and of class $\C^{1,\alpha}$, $y_{1},\dots y_{n}$
a collection of distinct points and the material coefficients $v=v(  \varepsilon)  $ and
$\rho=\rho(\ve)$  depend on $\varepsilon\in(
0,1]  $. By dilating and translating, each $\Omega_{\ell
}(\ve)  $ is transformed into the pivot domain $\Omega_{\circ}$ and one gets a recipe for calculating the resonances involving the volume
and surface operators of $\Omega_{\circ}$  (see 
\eqref{LF}). This prepares the framework for the asymptotic regimes
considered in Section 9. Here, we consider the simpler medium containing a single
inhomogeneity $\Omega(\ve)=y+\ve\Omega$ fulfilling the scaling (\ref{microres}), while the
material parameters belong to one of the following four cases:
\vskip5pt\noindent
case 1) \be\label{caso1}
v^{2}(\ve)=\chi_{\Omega_{\+}(\ve)}+\ve^{2}\big(v^{2}+v_{1}^{2}\ve+O(\ve^{2})\big)\chi_{\Omega(\ve)}\,,\quad\rho(\ve)=\chi_{\Omega_{\+}(\ve)}+\left(1+\rho_{1}\ve  +O(\ve^{2})\right)\chi_{\Omega(\ve)}\,;
\ee
\vskip5pt\noindent
case 2) 
\be\label{caso2}
v^{2}(\ve)=\chi_{\Omega_{\+}(\ve)}+\left(1+v_{1}^{2}\ve+O(\ve^{2})\right)\chi_{\Omega(\ve)}\,,\qquad
\rho(\ve)=\chi_{\Omega_{\+}(\ve)}+\ve^{2}\big(\rho+O(\ve)\big)\chi_{\Omega(\ve)}\,;
\ee
\vskip5pt\noindent
case 3) 
\be\label{caso3}
v^{2}(\ve)=\chi_{\Omega_{\+}(\ve)}+\ve\big(v^{2}+v_{1}^{2}\ve+O(\ve^{2})\big)\chi_{\Omega(\ve)}\,,\qquad
\rho(\ve)=\chi_{\Omega_{\+}(\ve)}+\ve\big(\rho+\rho_{1}\ve+O(\ve^{2})\big)\chi_{\Omega(\ve)}\,;
\ee
\vskip5pt\noindent
case 4) 
\be\label{caso4}
v^{2}(\ve)=\chi_{\Omega_{\+}(\ve)}+\ve^{2}\big(v^{2}+v_{1}^{2}\ve+O(\ve^{2})\big)\chi_{\Omega(\ve)}\,,\quad\rho(\ve)=\chi_{\Omega_{\+}(\ve)}+\ve\big(\rho+\rho_{1}\,\ve+O(\ve^{2})\big)\chi_{\Omega(\ve)}\,.
\ee
\vskip5pt\noindent
We denote by $A(\ve)$ the $\ve$-dependent operator corresponding to  such cases. 
The results of our analysis are
resumed in the next theorem.
\begin{theorem}
\label{Part II}For $\varepsilon$ small enough, there exist resonances
$\kappa_{\pm}^{2}(\varepsilon)$ of $A(\ve)$ such that the $\CO_{-}$-valued functions 
$\varepsilon\mapsto\kappa_{\pm}(\varepsilon)=
\pm \kappa_{0}+\kappa_{\pm}^{(1)}\varepsilon+O(\ve^{2})$ are analytic, where: 
\vskip5pt\noindent
case 1) 
$$
\kappa_{0}=
\frac{v}{\lambda^{1/2}}\,,\qquad \kappa_{\pm}^{(1)}=\pm
\left(\frac{v_{1}^{2}}{2v\lambda ^{1/2}}-
\frac12\,\rho_{1} v{\lambda ^{1/2}}\,\langle\gamma_{0}^{\-}e_{\lambda },\gamma_{1}^{\-}e_{\lambda }\rangle_{L^{2}(\Gamma)}\right)
- i\, \frac{v^{2}}{8\pi\lambda ^{2}}\,\langle 1,e_{\lambda }\rangle^{2}_{L^{2}(\Omega)} 
\,,
$$
$\lambda>0$ is a simple eigenvalue of the Newton potential operator $N_{0}$ in $L^{2}(\Omega)$ and $e_{\lambda}$ is the corresponding eigenfunction; 
\vskip5pt\noindent
case 2)
$$
\kappa_{0}=\rho^{1/2}\,\omega_{M}\,,\qquad \kappa_{\pm}^{(1)}=
\pm\frac{v_{1}^{2}}{2}\,\rho^{1/2}\omega_{M}-i\,\frac{|\Omega|}{8\pi} \,\rho\,\omega_{M}^{4}\,,
$$
$\omega_{M}>0$ the ''Minnaert resonance'' defined by 
$$
\omega_{M}^{2}:=\frac1{|\Omega|}\,\langle
S_{0}^{-1}1,1\rangle_{-\frac12,\frac12}\,;
$$
case 3)  \par\noindent
$\kappa_{0}=\rho^{1/2}\,\omega_{M}$, the expression for $\kappa_{\pm}^{(1)}$ can be  computed but for the sake of brevity is not written explicitly; 
\vskip5pt\noindent
case 4)
$$
\kappa_{0}=v\,\nu^{1/2}\,,\qquad \kappa_{\pm}^{(1)}=\pm\,\frac{1}{2}\,{\nu}^{1/2}\left(  \rho \,v\,\langle
\gamma_{0}u_{\nu},\phi_{\nu}\rangle_{\frac{1}{2},-\frac{1}{2}}-v_{1}^{2}%
v^{-1}\right)  ,
$$
$\nu>0$ is a simple eigenvalue of the positive Neumann Laplacian $-\Delta_{\Omega}^{N}$ in $L^{2}(\Omega)$, $u_{\nu}$ the corresponding normalized eigenfunction and%
\[
\phi_{\nu}:=S_{0}^{-1}\gamma_{0}({\nu}^{-1}-N_{0})u_{\nu}\,.
\]
Furthermore, there exist resonances $\kappa^{2}_{\pm}(\varepsilon^{1/2})$ of $A(\ve)$,  emerging from the zero eigenvalue of $-\Delta_{\Omega}^{N}$, such that the $\CO_{-}$-valued function 
$$\varepsilon^{1/2} \mapsto\kappa_{\pm}(\varepsilon^{1/2})=
\pm\rho^{1/2}\omega_{M}(  1+\kappa_{1}
\varepsilon)  \varepsilon^{1/2}+O(  \varepsilon^{2})  
$$
is analytic; the constant $\kappa_{1}$ is explicitly computed.
\end{theorem}
\par
Some of the acoustic resonances detected in Theorem \ref{Part II} have
counterparts in the corresponding scattering problems at a fixed frequency. We refer to
those localization phenomena of scattering of waves with a fixed frequency close to small inhomogeneities which are observed in different systems where the
material parameters at the inhomogeneity sites are highly contrasted with
respect to the background, see \cite{DeLiu}. Such behaviors, roughly
corresponding to quasi-modes, occur at specific incoming frequencies referred
to as 'resonances' of these micro-resonators. In acoustics, specific asymptotic
assumptions can excite localized quasi-modes with support concentrated on the
volume or on the surface of the resonator. The stationary acoustic wave
equation in the regime (\ref{caso1}) has been considered, among others, in
\cite{DaGhaSi IPI 21} (see also \cite{Amm et al dielctric} for an equivalent
dielectric problem), while the regime (\ref{caso2}) was considered in \cite{Amm
et al Minnaert}, \cite{DaGhaSi IPI 21} (with applications to the acoustic
imaging) and \cite{Minn} (where the localization effect was
interpreted in terms of the behavior of a frequency-dependent resolvent). 
In the first case, volume
quasi-modes are observed at frequencies $\lambda^{-1/2}$ corresponding to the
eigenvalues $\lambda$ of the Newton potential operator $N_{0}$. In the
second case, a surface quasi-mode is detected at the Minnaert frequency
$\omega_{M}$. The results for cases 1 and 2 in Theorem \ref{Part II}
confirm the common wisdom that localization phenomena in
micro-resonator systems correspond to resonances of the
generator of the dynamics: at the best of our knowledge, Theorem \ref{Part II}
provides a first proof of this correspondence.
\par
Several works on the stationary analysis of micro-resonator quasi-modes follow
the Gohberg-Sigal scheme based on a generalized Rouch\'{e} theorem for
meromorphic operator-valued functions. In Section 9 we adopt a different
strategy: after recasting the resonance problem in terms of a suitable
functional equation, the localization of resonances rests on implicit function
arguments and, in the case \eqref{caso4}, to detect the resonances emerging from $0\in\RE$, on a further Lyapunov-Schmidt reduction. This approach presents some novelties. At first,  it provides  solutions which are analytic functions with respect to the small parameter $\varepsilon$ (or $\sqrt{\varepsilon}$ in
the \eqref{caso4} case with zero energy); moreover, it provides a recursive algorithm for computing
the coefficients of the analytic expansion of the solution at any order. It is worth
noticing that although the computation of the higher-order corrections may be a difficult task, such a difficulty is only due to the length of the computations and no technical difficulty beyond the ones already overcame at the first-order level appears.
\par 
Our results, in cases \eqref{caso1} and \eqref{caso4}, require a simple spectrum for the Newton operator $N_{0}$ and for the Neumann Laplacian $\Delta_{\Omega}^{N}$ respectively. It is known that this is a generically true hypothesis as regards $\Delta_{\Omega}^{N}$ and we conjecture that the same is true as regards $N_{0}$. We expect that the same kind of results as in Theorem \ref{Part II} hold even without such an assumption on the spectra, by using a Lyapunov-Schmidt reduction and proceeding along the lines of our proof of case 4 for the zero energy-case (also see the approach to the branching of eigenvalues and eigenvectors of Fredholm operators in \cite[Section 32]{VT}). 
\par
A specific feature of micro-resonator regimes is the emergency of finite-energy
resonances asymptotically close to the continuous spectrum. It is worth
noticing that this does not exclude the existence of other resonances. In
particular, as recalled above, \cite{PoVo CMP 99} predicts the existence of a
family of acoustic resonances $\kappa_{j}^{2}$ such that 
$\operatorname{Re}(\kappa_{j}^{2})\to+\infty$ and $|\text{Im}(\kappa_{j})|
\le C_{N}|\kappa_{j}|^{-N}$ for any $N\gg 1$.  Such result applies to each of the cases
considered in Theorem \ref{Part II}. However, due to the scaling, the
energy-range of such resonances shifts at $+\infty$ as $\ve\to 0$. Due to this fact, we
argue that only the micro-resonator finite-energy resonances are expected to
have relevant contributions to the dynamics in the asymptotic limit.
\par
The spectral and scattering theory for operators of the kind $A_{v,1}$
were considered in \cite[Sections II.5 and III.6]{Yafa}, while a corresponding
micro-resonator setting with $v^{2}(\ve)=  \chi_{\mathbb{R}^{3}%
\backslash\Omega(\ve)  }+\varepsilon^{2}\chi_{\Omega(
\varepsilon)  }$ and $\rho=1$ was investigated in \cite{ZAMP} (there, no hypothesis on the simplicity of the spectrum of $N_{0}$ was assumed).  
The configurations \eqref{caso1} and \eqref{caso2} generalize the
ones considered in several previous works on acoustic quasi-modes, while we
underline that neither resonances nor scattering-resonances in the settings
\eqref{caso3} and \eqref{caso4} have been considered before. In particular, the
emergence of a system of resonances converging to $\sigma(
-\Delta_{\Omega}^{N})  $ when $v(\ve)  $ and
$\rho(\ve)  $ converge to zero with the same rate was not
observed so far. We are not aware of other contributions to 
resonance problems in micro-resonator settings beside the preprint
\cite{LiSi} dealing with a simplified
version of the Minnaert regime (\ref{caso2}); comparing \eqref{acQ} with the operator in \cite[Section 5]{LiSi}, the scattering resonances defined there coincide with the resonances defined by \eqref{res-def}.
\section{\label{Sec_Notation}Notation and definitions}{\ }
Here we introduce some notation and definitions, referring to \cite{McLe} as regards Sobolev spaces, trace maps and layer operators.
\vskip7pt
\noindent$\bullet$ $R_{\kappa}\equiv(-\Delta-\kappa^{2})^{-1}$ denotes the free resolvent with kernel
$\frac{e^{i\kappa|x-y|}}{4\pi|x-y|}$.
\vskip7pt
\noindent$\bullet$ $\Omega\equiv\Omega_{\-}=\Omega_{1}\cup\dots\cup\Omega_{n}$, the $\Omega_{\ell}$'s are connected disjoint open bounded Lipschitz domains and  $\Omega_{\+}:=\RE^{3}\backslash\overline\Omega$. 
The boundary of $\Omega$ is denoted by $\Gamma$ and $\Gamma_{\ell}$ denotes the boundary of $\Omega_{\ell}$.
\vskip7pt
\noindent$\bullet$ $\gamma_{0}^{\mathrm{in/ex}}$ and $\gamma_{1}^{\mathrm{in/ex}}$ denotes the lateral Dirichlet and Neumann trace maps at the boundary of $\Omega_{\-}$ and $\Omega_{\+}$respectively; $\gamma_{0}:=\frac{1}{2}(  \gamma_{0}^{\+}+\gamma_{0}^{\-})$,  $\gamma_{1}:=\frac{1}{2}(  \gamma_{1}^{\+
}+\gamma_{1}^{\-})$ denote the respective global traces.  
\vskip7pt
\noindent$\bullet$ $[  \gamma_{0}]  :=\gamma_{0}^{\+}-\gamma_{0}^{\-}$ and $[  \gamma_{1}]  :=\gamma_{1}^{\+}-\gamma_{1}^{\-}$ denote the jump at the boundary of the Dirichlet and Neumann trace maps.
\vskip7pt
\noindent$\bullet$ $\SL_{\kappa}$ and $\DL_{\kappa}$ denote the single- and double-layer operators.
\vskip7pt
\noindent$\bullet$
$S_{0}:=\gamma_{0}\SL_{0}$, $K_{0}:=\gamma_{0}\DL_{0}$, $K_{0}^{*}=\gamma_{1}\SL_{0}$.
\vskip7pt
\noindent$\bullet$
$\DN_{0}=(\frac12+\gamma_{1}\SL_{0})S_{0}^{-1}$ denotes the Dirichlet-to-Neumann operator.
\vskip7pt
\noindent$\bullet$ $1_{\Omega}:L^{2}(\mathbb{R}^{3})\rightarrow L^{2}(\Omega)$
denotes the restriction to $\Omega$; its adjoint $1_{\Omega}^{\ast}%
:L^{2}(\Omega)\rightarrow L^{2}(\mathbb{R}^{3})$ represents the extension by zero.
\vskip7pt
\noindent$\bullet$ $N_{\kappa}:=1_{\Omega}R_{k}1_{\Omega}^{\ast}$; in particular, $N_{0}:=1_{\Omega}R_{0}1_{\Omega}^{\ast}$ is the Newton potential operator in $L^{2}(\Omega)$.
\vskip7pt
\noindent$\bullet$ $\Delta_{\mathcal O}$ denotes the distributional Laplacian on the open set $\mathcal O$; whenever $\mathcal O=\RE^{3}$ we use the shorthand notation $\Delta$.
\vskip7pt
\noindent$\bullet$ $H^{s}(\mathcal O)$ denotes the scale of Sobolev spaces on the open set $\mathcal O$.
\vskip7pt
\noindent$\bullet$ $H^{s}(\Gamma)$, $|s|\le 1$, denotes the scale of Sobolev spaces on the boundary $\Gamma$ of the Lipschitz domain $\Omega$. 
\vskip7pt
\noindent$\bullet$ $\langle\cdot,\cdot\rangle_{-s,s}$ denotes the dual pairing, induced from the $L^{2}(\Gamma)$-scalar product, anti-linear with respect to the first variable, between $H^{-s}(\Gamma)$ and $H^{s}(\Gamma)$.
\vskip7pt
\noindent$\bullet$ $H_{\Delta}^{s}(\mathcal O):=\{u\in H^{s}(\mathcal O):\Delta_{\mathcal O}u\in L^{2}(\mathcal O)\}$; $L_{\Delta}^{2}(\mathcal O):=H_{\Delta}^{0}(\mathcal O)$.
\vskip7pt
\noindent$\bullet$ $H_{\Delta}^{s}(  \mathbb{R}^{3}\backslash\Gamma):=
H_{\Delta}^{s}(  \Omega)  \oplus H_{\Delta}^{s}(
\mathbb{R}^{3}\backslash\overline{\Omega})$; $L_{\Delta}^{2}(  \mathbb{R}^{3}\backslash\Gamma):=H_{\Delta}^{0}(  \mathbb{R}^{3}\backslash\Gamma)$.
\vskip7pt
\noindent$\bullet$ $\Delta^{\min}_{\mathcal O}:=\overline{\Delta_{\mathcal O}|{\mathcal C}^{\infty}_{comp}(\mathcal O)}$ denotes the minimal Laplacian in $L^{2}(\mathcal O)$.
\vskip7pt
\noindent$\bullet$ $\Delta^{\max}_{\mathcal O}:=\Delta_{\mathcal O}|L_{\Delta}^{2}(\mathcal O)$ denotes the maximal Laplacian in $L^{2}(\mathcal O)$.
\vskip7pt
\noindent$\bullet$ Given two Banach spaces $(X_{1},\|\cdot\|_{1})$ and $(X_{2},\|\cdot\|_{1})$, $X_{1}\cong X_{2}$ means that $X_{1}=X_{2}$  and that there exists $c_{1}>0$ and $c_{2}>0$ such that $c_{1}\|\cdot\|_{1}\le \|\cdot\|_{2}\le c_{2}\|\cdot\|_{1}$.
\vskip7pt
\noindent$\bullet$ $\B(X,Y)$ denotes the Banach space of bounded linear operators on the Banach space $X$ to the Banach space $Y$; $\B(X)\equiv\B(X,X)$. $\|B\|_{X,Y}$ denotes the operator norm of $B\in \B(X,Y)$; when no confusion arises, we use the abbreviated notation $\|B\|$. 
\vskip7pt
\noindent$\bullet$ ${\mathfrak{S}}_{\infty }(X,Y)\subseteq \B(X,Y)$ denotes the set of compact operators.
\vskip7pt
\noindent $\bullet$ $\varrho(A)$ and $\sigma(A)=\CO\backslash\varrho(A)$ denote the resolvent set and the spectrum of the self-adjoint operator $A$; $\sigma_{p}(A)$, $\sigma_{pp}(A)$,  $\sigma_{ac}(A)$,  $\sigma_{sc}(A)$, $\sigma_{ess}(A)$, $\sigma_{discr}(A)$, denote the point, pure point, absolutely continuous, singular continuous, essential and discrete spectra.
\vskip7pt
\noindent$\bullet$ $\chi_{S}$ denotes the characteristic function of a set $S$.
\vskip7pt
\noindent$\bullet$ $a\lesssim b$  means $a\le c\,b$ for some constant $c$.
\vskip7pt
\noindent$\bullet$ Given a map $0<\ve\mapsto u(\ve)\in X$, $X$ a Banach space, $u(\ve)=O(\ve^{\lambda})$ means $\|u(\ve)\|_{X}\lesssim \ve^{\lambda}$.
\vskip7pt
\noindent$\bullet$ $\CO_{\pm}:=\{z\in\CO:\pm\text{\rm Im}(z)>0\}$.
\section{The domain of self-adjointness} 
Let $\Omega_{1},\dots,\Omega_{n}$ be open, connected and bounded Lipschitz domains in $\RE^{3}$ as in the introduction. Notice that 
\be\label{os1}
H^{s}(\Omega)=H^{s}(\Omega_{1})\oplus\dots\oplus H^{s}(\Omega_{n})\,,
\ee
\be\label{os2}
H_{\Delta}^{s}(\Omega)=H_{\Delta}^{s}(\Omega_{1})\oplus\dots\oplus H_{\Delta}^{s}(\Omega_{n})\,,
\ee
and
\be\label{os3}
H^{s}(\Gamma)=H^{s}(\Gamma_{1})\oplus\dots\oplus H^{s}(\Gamma_{n})\,.
\ee
Furthermore,
$$
\gamma^{\-/\+}_{k}u=\gamma^{\-/\+}_{k,1}u\oplus\dots\oplus\gamma^{\-/\+}_{k,n}u\,,\quad k=0,1\,,
$$
where $\gamma_{k}^{\mathrm{in/ex}}$ and $\gamma_{k,\ell}^{\mathrm{in/ex}}$ denotes the lateral (Dirichlet if $k=0$ and Neumann if $k=1$) trace maps at the boundary of $\Omega$ and $\Omega_{\ell}$ respectively.\par
Given the strictly positive constants $\rho_{1},\dots,\rho_{n}$ and $v_{1},\dots, v_{n}$, we define the   piecewise-constant functions $v$ and $\rho$ as in \eqref{mat}.\par
Our first aim is to provide the self-adjoint realization of the acoustic operator  $A_{v,\rho}$ in \eqref{A_full_def}
with the transmission boundary conditions across $\Gamma_{\ell}$, $\ell=1,\dots,n$, given in \eqref{bc}.
\par In the following, we use the notation $L^{2}(\RE^{3},{b}^{-1})$ to denote the Hilbert space of square integrable functions with scalar product 
$$
\langle f,g\rangle_{L^{2}(\RE^{3}\!,\,{b}^{-1})}:=\int_{\RE^{3}}\overline f(x)g(x)\,b(x)^{-1}dx\,,\qquad b:=v^{2}\!\rho=\chi_{\Omega_{\+}}+\sum_{\ell=1}^{n} v^{2}_{\ell}\rho_{\ell}\,\chi_{\Omega_{\ell}}\,.
$$
Notice that $L^{2}(\RE^{3},{b}^{-1})\cong L^{2}(\RE^{3})$ as Banach spaces.
\begin{theorem}\label{acoustic} The acoustic operator $A_{v,\rho }$ 
with domain 
\begin{align*}
{\mathscr D}_{\rho}
:=&\big\{  u\in
H_{\Delta}^{1}(  \mathbb{R}^{3}\backslash\Gamma)  \,:\,\gamma^{\+}_{0}u=\gamma^{\-}_{0}u\,,\ \rho_{\ell}\,\gamma_{1,\ell}^{\mathrm{ex}}u=\gamma_{1,\ell}^{\mathrm{in}}u\,,\ \ell=1,\dots,n\big\}  \\
\equiv&\big\{  u\in H^{1}(\RE^{3})\cap
L_{\Delta}^{2}(  \mathbb{R}^{3}\backslash\Gamma)  \,:\, \rho_{\ell}\,\gamma_{1,\ell}^{\mathrm{ex}}u=\gamma_{1,\ell}^{\mathrm{in}}u\,,\ \ell=1,\dots,n\big\}
\end{align*}
is a nonpositive self-adjoint operator in
$L^{2}(  \mathbb{R}^{3},{b}^{-1})  $. 
\end{theorem}
\begin{proof} We introduce the sesquilinear form 
$$
{\mathscr Q}:H^{1}(  \mathbb{R}^{3})  \times
H^{1}(  \mathbb{R}^{3})
\subseteq L^{2}(  \mathbb{R}^{3},{b}^{-1}) \times L^{2}(  \mathbb{R}^{3},{b}^{-1}) \to\CO \,,\quad
{\mathscr Q}(  u_{1},u_{2}):=\langle \nabla u_{1},{\rho}^{-1}\nabla u_{2}\rangle _{L^{2}(\mathbb{R}^{3})  }\,.
$$
Since
\be\label{sesqui}
\|\nabla
u\|_{L^{2}(  \mathbb{R}^{3})  }^{2}\lesssim {\mathscr Q}(  u,u)\lesssim \|\nabla
u\|_{L^{2}(  \mathbb{R}^{3})  }^{2}\,,
\ee
the quadratic form $u\mapsto {\mathscr Q}(  u,u)  $ is closed and
nonnegative. The corresponding nonnegative self-adjoint operator $A$ in 
$L^{2}(\mathbb{R}^{3},{b}^{-1})  $ is  
defined by
$$
\begin{cases}
\text{dom}(  A)  :=\big\{  u_{1}\in H^{1}(\mathbb{R}^{3})  :\exists\,w\in L^{2}(  \mathbb{R}^{3})  \text{ s.t. }{\mathscr Q}(  u_{1},u_{2})  =\langle w,u_{2}\rangle _{L^{2}(\mathbb{R}^{3}\!,\,{b}^{-1})  }\ \forall\,u_{2}\in H^{1}(\mathbb{R}^{3})  \big\}, 
\\
Au_{1}:=w\,.
\end{cases}
$$
By the first Green identity (see, e.g., \cite[Theorem 4.4]{McLe}) one has, for any $u_{1}\in\text{dom}(  A_{v,\rho   }) $ and $u_{2}\in
H^{1}(  \mathbb{R}^{3})$,
\begin{align*}
&\langle  -A_{v,\rho }u_{1},u_{2}\rangle_{L^{2}(\mathbb{R}^{3}\!,\,{b}^{-1})  }\\
=&\langle   -\Delta^{\max}_{\Omega_{\+}}1_{\Omega_{\+}}\,u_{1},1_{\Omega_{\+}}\,u_{2}\rangle_{L^{2}(\Omega_{\+})  }+\sum_{\ell=1}^{n}\rho_{\ell}^{-1}
\langle  -\Delta^{\max}_{\Omega_{\ell} } 1_{\Omega_{\ell} }u_{1}, 1_{\Omega_{\ell} }u_{2}\rangle_{L^{2}(\Omega_{\ell} )  }
\\
=&\langle  \nabla1_{\Omega_{\+}}\,u_{1}, \nabla1_{\Omega_{\+}}\,u_{2}\rangle_{L^{2}(\Omega_{\+})  }+\sum_{\ell=1}^{n}\rho_{\ell}^{-1}
\langle  \nabla 1_{\Omega_{\ell} }u_{1}, \nabla 1_{\Omega_{\ell} }u_{2}\rangle_{L^{2}(\Omega )  }
\\
+&\langle 
\gamma_{1}^{\+}u_{1},\gamma_{0}^{\+}u_{2}\rangle_{-\frac12,\frac12}
-\sum_{\ell=1}^{n}\rho_{\ell}^{-1}\langle  \gamma_{1,\ell}^{\-}u_{1},\gamma_{0,\ell}^{\-}u_{2}\rangle_{-\frac12,\frac12}
\\
=&{\mathscr Q}(  u_{1},u_{2})  +\sum_{\ell=1}^{n}\langle 
(\gamma_{1,\ell}^{\+}-\rho_{\ell}^{-1} \gamma_{1,\ell}^{\-})u_{1},\gamma_{0,\ell}u_{2}\rangle
_{-\frac12,\frac12}
\\
=&{\mathscr Q}(  u_{1},u_{2})
\end{align*}
and so $-A_{v,\rho }\subseteq A$. Now, we prove the reverse inclusion. 
By
$$
\Delta_{\Omega_{\+}}^{\mathrm{min}}\oplus v_{1}^{2}\Delta_{\Omega_{1} }^{\mathrm{min}}\oplus \dots \oplus v_{n}^{2}\Delta_{\Omega_{n} }^{\mathrm{min}}  
\subseteq A_{v,\rho }
\,,
$$
one gets\footnote{
Since the function ${b}$ is constant on any $\Omega_{\ell}$ and on  
$\Omega_{\+}$, the adjoint in $L^{2}(\mathbb{R}^{3},{b}^{-1})$ of the direct sum $\Delta_{\Omega_{\+}}^{\mathrm{min}}\oplus v_{1}^{2}\Delta_{\Omega_{1} }^{\mathrm{min}} \oplus \dots \oplus v_{n}^{2}\Delta_{\Omega_{n} }^{\mathrm{min}}$ coincides with the one in $L^{2}(\mathbb{R}^{3})$.
} 
\begin{align}\label{op-incl}
A=A^{*}\subseteq -A^{*}_{v,\rho   }\subseteq&(-\Delta_{\Omega_{\+}}^{\mathrm{min}})^{*}\oplus(-v_{1}^{2}(\Delta_{\Omega_{1} }^{\mathrm{min}})^{*})\oplus\dots\oplus(-v_{n}^{2}(\Delta_{\Omega_{n} }^{\mathrm{min}})^{*})\\
=&(-\Delta_{\Omega_{\+}}^{\mathrm{max}})\oplus(-v_{1}^{2}\Delta_{\Omega_{1} }^{\mathrm{max}})\oplus\dots\oplus(-v_{n}^{2}\Delta_{\Omega_{n} }^{\mathrm{max}})
\end{align}
and so 
\begin{equation}
\text{dom}(  A)  \subseteq H^{1}(  \mathbb{R}^{3})  \cap
L_{\Delta}^{2}(  \mathbb{R}^{3}\backslash\Gamma)  
\label{dom_A_incl}%
\end{equation}
and 
$$
A\subseteq -A_{v,\rho }\quad\Leftrightarrow\quad \dom(A)\subseteq 
{\mathscr D}_{\rho}\,.
$$
The inclusion \eqref{dom_A_incl} allows to use, for any $u_{1}\in \dom(  A) $ and $u_{2}\in H^{1}(  \mathbb{R}^{3}) $, the first Green  identity. Thus, by the inclusion \eqref{op-incl},
\begin{align*}
&0=\langle  Au_{1},u_{2}\rangle  _{L^{2}(  \mathbb{R}^{3}\!,\,{b}^{-1})  }-{\mathscr Q}(u_{1},u_{2})\\
=&\langle -\Delta_{\Omega_{\+}}^{\mathrm{max}}1_{\Omega_{\+}}u_{1},1_{\Omega_{\+}}u_{2}\rangle
_{L^{2}(\Omega_{\+})}+\sum_{\ell=1}^{n}\rho_{\ell}^{-1}\langle  -\Delta_{\Omega_{\ell} }^{\mathrm{max}}1_{\Omega_{\ell} }u_{1},1_{\Omega_{\ell} }u_{2}\rangle  _{L^{2}(\Omega_{\ell} )}\\
-&
\langle  \nabla1_{\Omega_{\+}}\,u_{1}, \nabla1_{\Omega_{\+}}\,u_{2}\rangle_{L^{2}(\Omega_{\+})  }-\sum_{\ell=1}^{n}\rho_{\ell}^{-1}
\langle  \nabla 1_{\Omega_{\ell} }u_{1}, \nabla 1_{\Omega_{\ell} } u_{2}\rangle_{L^{2}(\Omega_{\ell} )  }
\\
=&-\langle 
\gamma_{1}^{\+}u_{1},\gamma_{0}^{\+}u_{2}\rangle_{-\frac12,\frac12}+\sum_{\ell=1}^{n}\rho_{\ell}^{-1}\langle  \gamma_{1,\ell}^{\-}u_{1},\gamma_{0,\ell}^{\-}u_{2}\rangle_{-\frac12,\frac12} \,.
\end{align*}
By $u_{2}\in
H^{1}(  \mathbb{R}^{3})  $, one has $\gamma_{0,\ell}^{\-}u_{2}=\gamma^{\+}_{0,\ell}u_{2}=\gamma_{0}u_{2}|\Gamma_{\ell}
$  and so, whenever $u_{1}\in\dom(A)$,
\begin{align*}
0=&\sum_{\ell=1}^{n}\rho_{\ell}^{-1}\langle  \gamma_{1,\ell}^{\-}u_{1},\gamma_{0,\ell}^{\-}u_{2}\rangle_{-\frac12,\frac12} -\langle 
\gamma_{1}^{\+}u_{1},\gamma_{0}^{\+}u_{2}\rangle_{-\frac12,\frac12}\\
=&
\sum_{\ell=1}^{n}\left(\rho_{\ell}^{-1}\langle  \gamma_{1,\ell}^{\-}u_{1},\gamma_{0,\ell}^{\-}u_{2}\rangle_{-\frac12,\frac12} -\langle 
\gamma_{1,\ell}^{\+}u_{1},\gamma_{0,\ell}^{\+}u_{2}\rangle_{-\frac12,\frac12}\right)\\
=&
\left\langle(\rho_{1}^{-1}  \gamma_{1,1}^{\-}-\gamma^{\+}_{1,1})u_{1}\oplus\dots\oplus(\rho_{n}^{-1}  \gamma_{1,n}^{\-}-\gamma^{\+}_{1,n})u_{1},\gamma_{0}u_{2}\right\rangle_{-\frac12,\frac12}\,.
\end{align*}
By the surjectivity of $\gamma_{0}\in{\B}(  H^{1}(\mathbb{R}^{3})  ,H^{1/2}(  \Gamma) )  $, this implies
$ \gamma_{1,\ell}^{\+}u_{1}=\rho_{\ell}^{-1}\gamma_{1,\ell}^{\-}u_{1}$ for any index $\ell$. Since $\dom(  A)  \subseteq H^{1}(\mathbb{R}^{3})  $, for any $u\in \dom(  A)$ one has $\gamma_{0}^{\-}u=\gamma_{0}^{\+}u$ and in conclusion one obtains $\dom(  A)  \subseteq\dom(  A_{v,\rho })$.\end{proof}
\begin{remark} Let us notice the operatorial inclusions
\be\label{op-incl1}
\Delta_{\Omega_{\+}}^{\mathrm{min}}\oplus v^{2}\Delta_{\Omega }^{\mathrm{min}}\subseteq A_{v,\rho }\subseteq\Delta_{\Omega_{\+}}^{\mathrm{max}}\oplus v^{2}\Delta_{\Omega }
\ee
and the operatorial identities
$$
v^{2}\Delta_{\Omega}^{\mathrm{min}}=v_{1}^{2}\Delta_{\Omega_{1} }^{\mathrm{min}}\oplus\dots\oplus v_{n}^{2}\Delta_{\Omega_{n} }^{\mathrm{min}}\,,\qquad v^{2}\Delta_{\Omega }^{\mathrm{max}}=v_{1}^{2}\Delta_{\Omega_{1} }^{\mathrm{max}}\oplus\dots\oplus v_{n}^{2}\Delta_{\Omega_{n} }^{\mathrm{max}}\,.
$$
By the previous Theorem, the self-adjoint realization of the acoustic operator is given by the restriction 
$$A_{v,\rho}=(\Delta_{\Omega_{\+}}^{\mathrm{max}}\oplus v^{2}\Delta_{\Omega }^{\mathrm{max}})|{\mathscr D}_{\rho}
\,.
$$ 
Obviously, $A_{1,1}$ coincides with the free Laplacian in $L^{2}(\RE^{3})$ with self-adjointness domain $H^{2}(\RE^{3})$.
\end{remark}
\section{Resolvents difference}
Now, we look for for the resolvent of $A_{v,\rho }$; in particular, we are interested in a formula for the difference  $(-A_{v,\rho }-\kappa^{2})^{-1}-R_{\kappa}$.\par
Following the scheme in \cite{JST}, we introduce the bounded operators
$$
\tau:H^{1}_{\Delta}(\RE^{3}\backslash\Gamma)\to H^{1}_{\Delta}(\RE^{3}\backslash\Gamma)\oplus H^{1/2}(\Gamma)\oplus H^{-1/2}(\Gamma)\,,
$$
$$
\tau u:=\tau_{1}u\oplus\tau_{2}u\,,\qquad \tau_{1}u:=u\,,\qquad\tau_{2}u:=\gamma_{0}u\oplus \gamma_{1}u\,,
$$
$$
G_{\kappa}: L^{2}(\RE^{3})\oplus H^{-1/2}(\Gamma)\oplus H^{1/2}(\Gamma)\to H^{1}_{\Delta}(\RE^{3}\backslash\Gamma)\,,
$$
$$
G_{\kappa}(u\oplus\varphi\oplus\phi):=R_{\kappa}u+\SL_{\kappa}\varphi+\DL_{\kappa}\phi\,,
$$
$$
\tau G_{\kappa}:L^{2}(\RE^{3})\oplus H^{-1/2}(\Gamma)\oplus H^{1/2}(\Gamma)
\to H^{1}_{\Delta}(\RE^{3}\backslash\Gamma)\oplus H^{1/2}(\Gamma)\oplus H^{-1/2}(\Gamma)\,,
$$
$$
\tau G_{\kappa}:=\begin{bmatrix}R_{\kappa}&\SL_{\kappa}&\DL_{\kappa}\\
\gamma_{0}R_{\kappa}&\gamma_{0}\SL_{\kappa}&\gamma_{0}\DL_{\kappa}\\
\gamma_{1}R_{\kappa}&\gamma_{1}\SL_{\kappa}&\gamma_{1}\DL_{\kappa}\end{bmatrix}.
$$
Then, given two bounded operators
$$
B_{1}: H^{1}_{\Delta}(\RE^{3}\backslash\Gamma)\to L^{2}(\RE^{3})\,,
$$
$$
B_{2}: H^{1/2}(\Gamma)\oplus H^{-1/2}(\Gamma)\to H^{-1/2}(\Gamma)\oplus H^{1/2}(\Gamma)\,,
$$
we define the set
$$
Z_{\sb}:=\{\kappa\in\CO_{+}:(1-(B_{1}\oplus B_{2})\tau G_{\kappa})^{-1}\in\B(L^{2}(\RE^{3})\oplus H^{-1/2}(\Gamma)\oplus H^{1/2}(\Gamma))\}\,.
$$
\begin{remark} Notice that the (possibly empty) set $Z_{\sb}$ is open by the continuity of the map $\kappa\mapsto \tau G_{\kappa}$.
\end{remark}
With reference to \cite{JST}, here we have $\sb=(1,B_{1},B_{2})$. Even if the operators $B_{0}=1$, $B_{1}$ and $B_{2}$ do not satisfy the relations  (2.4) in \cite{JST}, and the spaces $H^{1}_{\Delta}(\RE^{3}\backslash\Gamma)$ and $L^{2}(\RE^{3})$ do not constitute a dual couple (these should correspond to the dual couple $(\fb_{1},\fb_{1}^{*})$ in \cite{JST}), most of the results in \cite[Section 2]{JST}, in particular, \cite[Theorems 2.1 and 2.6, Lemmata 2.4 and 2.5]{JST},  apply also to our case (aside from the self-adjointness issue in \cite[Theorem 2.1]{JST}).
\begin{theorem}\label{teo-res} Suppose $Z_{\sb}\not=\varnothing$; then there exists a closed operator 
$$
A_{\sb}:\dom(A_{\sb})\subseteq L^{2}(\RE^{3})\to L^{2}(\RE^{3})
$$
such that 
\be\label{resB}
(-A_{\sb}-\kappa^{2})^{-1}=R_{\kappa}+G_{\kappa}
\big(1-(B_{1}\oplus B_{2})\tau G_{\kappa}\big)^{-1}(B_{1}\oplus B_{2})G_{-\bar\kappa}^{*}\,,\quad \kappa\in Z_{\sb}\,.
\ee
\end{theorem}
\begin{proof} One follows the same arguments as in the proof of \cite[Theorem 2.1]{JST} (which build on \cite[Theorem 2.1]{JFA}). Noticing that the additional conditions \cite[hypothesis (2.4)]{JST} serve only to show that the resolvent in \cite[Theorem 2.1]{JST} is the resolvent of a self-adjoint operator, in our case it is enough to prove that the operator in \eqref{resB} is an injective pseudo-resolvent (hence, is the resolvent of a closed operator). The proof of this statement is exactly the same as the one provided in \cite[Theorems 2.1]{JST} and is not reproduced here\footnote{For a similar proof of injectivity, see Lemma \ref{inj} below.}.
\end{proof}
Now, we determine $B_{1}$ and $B_{2}$ such that $A_{\sb}=A_{v,\rho }$. 
\begin{theorem}\label{equal}
Let
\be\label{B1}
B_{1}=B_{v}:=({v^{2}}-1)1_{\Omega}^{*}\Delta^{\max}_{\Omega }1_{\Omega}\,,
\ee
\be\label{B2}
B_{2}=B_{\rho}:=\begin{bmatrix}0&\widetilde\rho \ \\
0&0
\end{bmatrix},\qquad
\widetilde\rho:=2\sum_{\ell=1}^{n}\frac{\rho_{\ell}-1}{\rho_{\ell}+1}\,\chi_{\Gamma_{\ell}}
\,.
\ee
If $Z_{\sb}\not=\varnothing$ for such a choice, then $$
A_{\sb}=A_{v,\rho }\,.
$$
\end{theorem}
\begin{proof} In the following, we take $\kappa\in Z_{\sb}$; let us denote by $\widehat R_{\kappa}$ the resolvent in \eqref{resB}. Then,
\begin{align*}
&\dom(A_{\sb})=\widehat R_{\kappa} (L^{2}(\RE^{3}))\\
=&\big\{u\in H_{\Delta}^{1}(\RE^{3}\backslash\Gamma):u=u_{\kappa}+G_{\kappa}\big(1-(B_{1}\oplus B_{2})\tau G_{\kappa}\big)^{-1}(B_{1}\oplus B_{2})\tau u_{\kappa}\,,\ u_{\kappa}\in H^{2}(\RE^{3})\big\}\,.
\end{align*}
Furthermore, by
$$
\widehat R_{\kappa} (-\Delta-\kappa^{2})u_{\kappa}=u_{\kappa}+G_{\kappa}\big(1-(B_{1}\oplus B_{2})\tau G_{\kappa}\big)^{-1}(B_{1}\oplus B_{2})\tau u_{\kappa}=u\,,
$$
one gets
\be\label{act}
(\widehat R_{\kappa} )^{-1}u=(-A_{\sb}-\kappa^{2})u=(-\Delta-\kappa^{2})u_{\kappa}\,.
\ee
According to \cite[Lemma 2.4]{JST}, the linear map\footnote{Notice the change of notation with 
respect to \cite{JST}: the map there denoted by ${\rho}_{\sb}$ is here denoted by $q_{\sb}$}
$$
q_{\sb}:\dom(A_{\sb})\to L^{2}(\RE^{3})\oplus H^{-1/2}(\Gamma)\oplus H^{1/2}(\Gamma)\,,
$$ 
$$
q_{\sb}(u):=\big(1-(B_{1}\oplus B_{2})\tau G_{\kappa}\big)^{-1}(B_{1}\oplus B_{2})\tau u_{\kappa}\equiv q_{\circ}(u)\oplus q_{-}(u)\oplus q_{+}(u)\,,
$$
is well-defined and so
$$
\dom(A_{\sb})=\big\{u\in H_{\Delta}^{1}(\RE^{3}\backslash\Gamma):u=u_{\kappa}+R_{\kappa}{q}_{\circ}(u)+
\SL_{\kappa}{q}_{-}(u)+\DL_{\kappa}{q}_{+}(u)\,,\ u_{\kappa}\in H^{2}(\RE^{3})\big\}\,.
$$
Furthermore, by \cite[Lemma 2.5]{JST}, 
\be\label{bc-B}
u\in\dom(A_{\sb})\quad\Rightarrow\quad
\begin{cases}
&B_{1}u={q}_{\circ}(u)\\
&B_{2}(\gamma_{0}u\oplus\gamma_{1}u)=
{q}_{-}(u)\oplus{q}_{+}(u)\,.
\end{cases}
\ee
By $\ran(B_{1})\subseteq L^{2}(\RE^{3})$, one gets
$$
\ran(R_{\kappa}B_{1})\subseteq H^{2}(\RE^{3})\subseteq\ker([\gamma_{0}])\cap\ker([\gamma_{1}])\,.
$$
Hence, by the mapping properties and by the jump relations for the single- and double-layer operators  and by \eqref{bc-B},
\begin{align*}
\dom(A_{\sb})=&\big\{u\in H_{\Delta}^{1}(\RE^{3}\backslash\Gamma):u_{\kappa}:=u-R_{\kappa}B_{1}u+
\SL_{\kappa}[\gamma_{1}]u-\DL_{\kappa}[\gamma_{0}]u\in H^{2}(\RE^{3})\big\}\\
\equiv&\big\{u\in H_{\Delta}^{1}(\RE^{3}\backslash\Gamma):\widetilde u_{\kappa}:=u+
\SL_{\kappa}[\gamma_{1}]u-\DL_{\kappa}[\gamma_{0}]u\in H^{2}(\RE^{3})\big\}\,,
\end{align*}
\be\label{bc-B-1}
u\in\dom(A_{\sb})\quad\Rightarrow\quad
 B_{2}(\gamma_{0}u\oplus\gamma_{1}u)=(-[\gamma_{1}]u)\oplus[\gamma_{0}](u)\,.
\ee
Since
$$
\gamma_{1}^{\+} u=\gamma_{1}u+\frac12\,[\gamma_{1}]u\,,\qquad 
\gamma_{1}^{\-} u=\gamma_{1}u-\frac12\,[\gamma_{1}]u\,,
$$
the boundary conditions \eqref{bc} are equivalent to 
\be\label{bc-equi}
[\gamma_{0}]u=0\,,\qquad [\gamma_{1,\ell}]u=2\,\frac{1-\rho_{\ell}}{1+\rho_{\ell}}\,\gamma_{1,\ell}u\equiv -\widetilde\rho_{\ell}\,\gamma_{1,\ell}u\,,\qquad \ell=1,\dots,n\,.
\ee
Therefore, by \eqref{bc-B-1}, taking $B_{2}=B_{\rho}$ one obtains
\be\label{dAB}
\dom(A_{\sb})=\big\{u\in H^{1}(\RE^{3})\cap L_{\Delta}^{2}(  \mathbb{R}^{3}\backslash\Gamma)
: \widetilde u_{\kappa}:=u
-\widetilde\rho \,\SL_{\kappa}\gamma_{1}u\in H^{2}(\RE^{3})\big\}
\ee
and 
\be\label{d-in}
\dom(A_{\sb})\subseteq {\mathscr D}_{\rho}
\,.
\ee
Furthermore, by \eqref{act}, one has
\begin{align*}
&(A_{\sb}+\kappa^{2})u=(\Delta+\kappa^{2})(u-R_{\kappa}B_{1}u-\widetilde\rho \,
\SL_{\kappa}\gamma_{1}u)\\
=&(\Delta^{\max}_{\Omega }\oplus\Delta^{\max}_{\Omega_{\+}}+\kappa^{2})u-(\Delta+\kappa^{2})R_{\kappa}B_{1}u-\widetilde\rho \,(\Delta^{\max}_{\Omega }\oplus\Delta^{\max}_{\Omega_{\+}}+\kappa^{2})\SL_{\kappa}\gamma_{1}u\\
=&(\Delta^{\max}_{\Omega }\oplus\Delta^{\max}_{\Omega_{\+}}+\kappa^{2})u+B_{1}u\,.
\end{align*}
Therefore, by taking $ B_{1}=B_{\rho}\equiv0\oplus({v^{2}}-1)\Delta^{\max}_{\Omega }$,
one gets
$$
A_{\sb}\subseteq\Delta^{\max}_{\Omega_{\+}}\oplus(v^{2}\Delta^{\max}_{\Omega })\,.
$$
Thus, by \eqref{d-in} and \eqref{op-incl1},
$$
A_{\sb}\subseteq A_{v,\rho }
$$ 
and, to obtain the equality $A_{\sb}=A_{v,\rho }$, it remains to prove the inclusion opposite to 
\eqref{d-in}.\par
For any $u\in L^{2}_{\Delta}(\RE^{3}\backslash\Gamma)\equiv\dom(\Delta_{\Omega_{\+}}^{\max}\oplus \Delta_{\Omega}^{\max})$, one has (see \cite[Lemma 4.2]{JDE 16})
$$u+\SL_{\kappa}[\gamma_{1}]u-\DL_{\kappa}[\gamma_{0}]u\in H^{2}(\RE^{3})\,.
$$
Hence, for any $u\in {\mathscr D}_{\rho}
\subseteq H^{1}(\RE^{3})\cap L^{2}_{\Delta}(\RE^{3}\backslash\Gamma)$ 
one has, by \eqref{bc-equi}, 
$$
u-\widetilde\rho \,\SL_{\kappa}\gamma_{1}u\in H^{2}(\RE^{3})\,.
$$
By \eqref{dAB}, this gives 
$$
{\mathscr D}_{\rho}
\subseteq\dom(A_{\sb})\,.
$$
\end{proof}
\begin{remark}\label{alt} By the proof of the previous Theorem, see \eqref{bc-equi}, one gets an alternative definition of the self-adjointness domain of $A_{v,\rho}$:
$$
{\mathscr D}_{\rho}
=\left\{  u\in H^{1}(\RE^{3})\cap
L_{\Delta}^{2}(  \mathbb{R}^{3}\backslash\Gamma)  \,:\, \frac12\,
[\gamma_{1,\ell}]u=\frac{1-\rho_{\ell}}{1+\rho_{\ell}}\,\gamma_{1,\ell}u\,,\ \ell=1,\dots,n\right\}\,.
$$  
\end{remark}
Now, we show that $Z_{\sb}$ in Theorem \ref{equal} is not empty.
\begin{lemma}\label{ZB} Let $B_{1}$ and $B_{2}$ as in \eqref{B1} and \eqref{B2}. Then $Z_{\sb}=\CO_{+}
$.
\end{lemma}
\begin{proof} By
\begin{align}\label{Qf}
&\big(1-(B_{1}\oplus B_{2})\tau G_{\kappa}\big)^{-1}\\\nonumber
=&
\begin{bmatrix}1-({v^{2}}-1)1_{\Omega}^{*}\Delta^{\max}_{\Omega }1_{\Omega}R_{\kappa}&(1-{v^{2}})1_{\Omega}^{*}\Delta^{\max}_{\Omega }1_{\Omega}\SL_{\kappa}&(1-{v^{2}})1_{\Omega}^{*}\Delta^{\max}_{\Omega }1_{\Omega}\DL_{\kappa}\\
-\widetilde\rho \,\gamma_{1}R_{\kappa}&1-\widetilde\rho \,\gamma_{1}\SL_{\kappa}&-\widetilde\rho \,\gamma_{1}\DL_{\kappa}\\
0&0&1\end{bmatrix}^{-1}\\\nonumber
=&
\begin{bmatrix}\big(1-\widetilde M_{\kappa}\,\big)^{-1}&\big(1-\widetilde M_{\kappa}\,\big)^{-1}V_{\kappa}\\
0&1
\end{bmatrix},
\end{align}
where
$$
\widetilde M_{\kappa}:=\begin{bmatrix}({v^{2}}-1)1_{\Omega}^{*}\Delta^{\max}_{\Omega }1_{\Omega}R_{\kappa}&({v^{2}}-1)1_{\Omega}^{*}\Delta^{\max}_{\Omega }1_{\Omega}\SL_{\kappa}\\
\widetilde\rho \,\gamma_{1}R_{\kappa}&\widetilde\rho \,\gamma_{1}\SL_{\kappa}
\end{bmatrix}
$$
and
$$
V_{\kappa}:=\begin{bmatrix}({v^{2}}-1)1_{\Omega}^{*}\Delta^{\max}_{\Omega }1_{\Omega}\DL_{\kappa}\\
\widetilde\rho \,\gamma_{1}\DL_{\kappa}
\end{bmatrix},
$$
one gets
$$
Z_{\sb}\supseteq\{\kappa\in\CO_{+}: \big(1-\widetilde M_{\kappa}\,\big)^{-1}\in\B(L^{2}(\RE^{3})\oplus H^{-1/2}(\Gamma))\}\,.
$$ 
Using the identification $L^{2}(\RE^{3})\equiv L^{2}(\Omega_{\+})\oplus L^{2}(\Omega)$ given by the unitary map $u\mapsto 1_{\Omega_{\+}}\,u\oplus 1_{\Omega}u$ with inverse $u_{\+}\oplus u_{\-}\mapsto 
1_{\Omega_{\+}}^{*}\,u_{\+}+1^{*}_{\Omega}u_{\-}$, one gets, setting
$$
N_{\kappa}:=1_{\Omega}R_{\kappa}1_{\Omega}^{*}:L^{2}(\Omega)\to L^{2}(\Omega) \,,
$$
and noticing that $\gamma_{1}1_{\Omega}u=\gamma_{1}^{\-}u=\gamma_{1}u$ whenever $u\in H^{2}(\RE^{3})$,
\begin{align}\label{1-M}
&(1-\widetilde M_{\kappa})^{-1}\\\nonumber
=&\begin{bmatrix}1&0&0\\
(1-{v^{2}})\Delta^{\max}_{\Omega }1_{\Omega}R_{\kappa}1_{\Omega_{\+}}^{*}&1+(1-{v^{2}})\Delta^{\max}_{\Omega }N_{\kappa}
&(1-{v^{2}})\Delta^{\max}_{\Omega }1_{\Omega}\SL_{\kappa}\\
-\widetilde\rho \,\gamma_{1}R_{\kappa}1_{\Omega_{\+}}^{*}&-\widetilde\rho \,\gamma_{1}^{\-}N_{\kappa}& 1-\widetilde\rho \,\gamma_{1}\SL_{\kappa}
\end{bmatrix}^{-1}\\\nonumber
=&\begin{bmatrix}1&0\\
(1-M_{\kappa})^{-1}W_{\kappa}&(1-M_{\kappa})^{-1}
\end{bmatrix},
\end{align}
where
$$
M_{\kappa}=\begin{bmatrix}({v^{2}}-1)\Delta^{\max}_{\Omega }N_{\kappa}
&({v^{2}}-1)\Delta^{\max}_{\Omega }1_{\Omega}\SL_{\kappa}\\
\widetilde\rho \,\gamma_{1}^{\-}N_{\kappa}& \widetilde\rho \,\gamma_{1}\SL_{\kappa}
\end{bmatrix}
$$
and
$$
W_{\kappa}=\begin{bmatrix}
({v^{2}}-1)\Delta^{\max}_{\Omega }1_{\Omega}R_{\kappa}1_{\Omega_{\+}}^{*}\\
\widetilde\rho \,\gamma_{1}R_{\kappa}1_{\Omega_{\+}}^{*}\end{bmatrix}.
$$
Hence,
$$
Z_{\sb}\supseteq\{\kappa\in\CO_{+}: \big(1-M_{\kappa}\,\big)^{-1}\in\B(L^{2}(\Omega)\oplus H^{-1/2}(\Gamma))\}\,.
$$
Notice that, by the relations
$$
-\Delta^{\max}_{\Omega }N_{\kappa}=1+\kappa^{2}N_{\kappa}\,,\qquad
-\Delta^{\max}_{\Omega }1_{\Omega}\SL_{\kappa}=\kappa^{2}1_{\Omega}\SL_{\kappa}\,.
$$
$M_{\kappa}$ rewrites as
\be\label{Mk}
M_{\kappa}=\begin{bmatrix}({1-v^{2}})(1+\kappa^{2}N_{\kappa})
&(1-{v^{2}})\kappa^{2}1_{\Omega}\SL_{\kappa}\\
\widetilde\rho \,\gamma_{1}^{\-}N_{\kappa}& \widetilde\rho \,\gamma_{1}\SL_{\kappa}
\end{bmatrix}
\ee
By the analyticity of the operator-valued map $\CO\ni\kappa\mapsto M_{\kappa}$, $Z_{\sb}$ is not void whenever $\big(1-M_{0}\,\big)^{-1}\in\B(L^{2}(\Omega)\oplus H^{-1/2}(\Gamma))$. One has, denoting by $K_{0}\in\B(H^{1/2}(\Gamma))$ the Neumann-Poincar\'e operator $K_{0}:=\gamma_{0}\DL_{0}=(\gamma_{1}\SL_{0})^{*}$,
$$
(1-M_{0})^{-1}=\begin{bmatrix}v^{2}
&0\\
-\widetilde\rho \,\gamma_{1}^{\-}N_{0}& 1-\widetilde\rho \,K^{*}_{0}
\end{bmatrix}^{-1}=\begin{bmatrix}v^{-2}
&0\\
v^{-2}\,\widetilde\rho (1-\widetilde\rho \,K^{*}_{0})^{-1}\gamma_{1}^{\-}N_{0}& (1-\widetilde\rho \,K^{*}_{0})^{-1}
\end{bmatrix}.
$$
If $\rho_{1}=\dots=\rho_{n}=1$, then $\widetilde\rho =0$ and 
$$
(1-M_{0})^{-1}=\begin{bmatrix}v^{-2}
&0\\
0& 1
\end{bmatrix}\in \B(L^{2}(\Omega)\oplus H^{-1/2}(\Gamma))\,.
$$ 
Otherwise, without loss of generality, we can suppose that there exists $1\le n_{0}<n$ such that
$\rho_{\ell}\in(0,+\infty)\backslash\{1\}$ for any $\ell\in[n_{0},n]$ and $\rho_{\ell}=1$ otherwise.
Defining  $\Gamma_{0}:=\Gamma_{n_{0}}\cup\dots\cup\Gamma_{n}$, we use the orthogonal decomposition $H^{-1/2}(\Gamma)=H^{-1/2}(\Gamma\backslash\Gamma_{0})\oplus H^{-1/2}(\Gamma_{0})$. Therefore,
$$
(1-M_{0})^{-1}=\begin{bmatrix}v^{-2}
&0\\
v^{-2}\,(0\oplus\widetilde\rho_{0} (1-\widetilde\rho_{0} \,K^{*}_{00})^{-1})\gamma_{1}^{\-}N_{0}& 1\oplus(1-\widetilde\rho_{0} \,K^{*}_{00})^{-1}
\end{bmatrix},
$$
where $K_{00}$ denotes the Neumann-Poincar\'e operator for $\Gamma_{0}$ and 
$$\widetilde\rho_{0}(\phi_{n_{0}}\oplus\dots\oplus\phi_{n}):=(\widetilde\rho_{n_{0}}\phi_{n_{0}})\oplus\dots\oplus(\widetilde\rho_{n}\phi_{n})\,.
$$
Since ${\widetilde \rho_{\ell}\,}^{-1}\notin [-1/2,1/2)$ for any $\ell\in[n_{0},n]$ and 
$\sigma(K_{00}^{*})=\sigma(K_{00})\subseteq [-1/2,1/2)$ (see, e.g., \cite[Section 2]{Kang}, \cite[Section 5]{AKMP}), one has   $$(1-M_{0})^{-1}\in \B(L^{2}(\Omega)\oplus H^{-1/2}(\Gamma))\,.$$ 
Hence, $Z_{\sb}$ is not empty and so, by the same arguments as in \cite[Theorem 2.19 and Remark 2.20]{CFP}, one gets 
$$
Z_{\sb}=\{\kappa\in\CO_{+}:-\kappa^{2}\in\varrho(\Delta)\cap\varrho(A_{v,\rho })\}
\,.
$$
Since $A_{v,\rho}\le 0$, one has $\sigma(A_{v,\rho})\subseteq (-\infty,0]$ and so 
$$\varrho(A_{v,\rho })\supseteq\CO\backslash(-\infty,0]=\varrho(\Delta)\,.
$$
Therefore, 
$$
Z_{\sb}=\{\kappa\in\CO_{+}:-\kappa^{2}\in\varrho(\Delta)\}=\CO_{+}\,.
$$
\end{proof}
Combining Theorem \ref{equal} with Lemma \ref{ZB}, one obtains
\begin{theorem}\label{RD} For any $\kappa\in\CO_{+}$, there holds 
\begin{align*}
&(-A_{v,\rho }-\kappa^{2})^{-1}-(-\Delta-\kappa^{2})^{-1}\\
=&\begin{bmatrix}R_{\kappa}1_{\Omega}^{*}& \SL_{\kappa}
\end{bmatrix}
\begin{bmatrix}v^{2}+({v^{2}}-1)\kappa^{2}N_{\kappa}&({v^{2}}-1)\kappa^{2}1_{\Omega}\SL_{\kappa}\\-\widetilde\rho\,\gamma^{\-}_{1}N_{\kappa}&1-\widetilde\rho\,\gamma_{1}\SL_{\kappa}
\end{bmatrix}^{\!-1}
\begin{bmatrix}({v^{2}}-1)\Delta^{\max}_{\Omega }1_{\Omega}R_{\kappa}\\ \widetilde\rho\gamma_{1}R_{\kappa}
\end{bmatrix}.
\end{align*}
\end{theorem}
\begin{proof}
By \eqref{Qf} and \eqref{1-M},
\begin{align*}
&G_{\kappa}
\big(1-(B_{1}\oplus B_{2})\tau G_{\kappa}\big)^{-1}(B_{1}\oplus B_{2})G_{\bar\kappa}^{*}
\\
=&
\begin{bmatrix}R_{\kappa}& \SL_{\kappa}& \DL_{\kappa}
\end{bmatrix}\begin{bmatrix}\big(1-\widetilde M_{\kappa}\,\big)^{-1}&\big(1-M_{\kappa}\,\big)^{-1}V_{\kappa}\\
0&1
\end{bmatrix}\begin{bmatrix}\begin{bmatrix}({v^{2}}-1)1_{\Omega}^{*}\Delta^{\max}_{\Omega }1_{\Omega}R_{\kappa}\\ \widetilde\rho \,\gamma_{1}R_{\kappa}\end{bmatrix}\\0
\end{bmatrix}\\
=&
\begin{bmatrix}R_{\kappa}& \SL_{\kappa}
\end{bmatrix}\big(1-\widetilde M_{\kappa}\,\big)^{-1}\begin{bmatrix}({v^{2}}-1)1_{\Omega}^{*}\Delta^{\max}_{\Omega }1_{\Omega}R_{\kappa}\\ \widetilde\rho \,\gamma_{1}R_{\kappa}
\end{bmatrix}\\
=&
\begin{bmatrix}1_{\Omega_{\+}}R_{\kappa}1^{*}_{\Omega_{\+}}&
1_{\Omega}R_{\kappa}1^{*}_{\Omega}& \SL_{\kappa}
\end{bmatrix}
\begin{bmatrix}1&0\\
(1-M_{\kappa})^{-1}W_{\kappa}&(1-M_{\kappa})^{-1}
\end{bmatrix}\begin{bmatrix}0\\\begin{bmatrix}({v^{2}}-1)\Delta^{\max}_{\Omega }1_{\Omega}R_{\kappa}1_{\Omega}^{*}\\ \widetilde\rho \,\gamma_{1}R_{\kappa}
\end{bmatrix}\end{bmatrix}\\
=&
\begin{bmatrix}
R_{\kappa}1^{*}_{\Omega}& \SL_{\kappa}
\end{bmatrix}
(1-M_{\kappa})^{-1}\begin{bmatrix}({v^{2}}-1)\Delta^{\max}_{\Omega }1_{\Omega}R_{\kappa}\\ \widetilde\rho \,\gamma_{1}R_{\kappa}
\end{bmatrix}.
\end{align*}
The proof is then concluded by the definitions of $\widetilde\rho$ and $M_{k}$ in \eqref{B2} and \eqref{Mk}.
\end{proof}
\begin{remark}\label{not1} In the case $v_{\ell}\not=1$ and $p_{\ell}\not=1$ for any index $\ell$, Theorem \ref{RD} rewrites as 
\begin{align*}
&(-A_{v,\rho }-\kappa^{2})^{-1}-(-\Delta-\kappa^{2})^{-1}\\
=&\begin{bmatrix}R_{\kappa}1_{\Omega}^{*}& \SL_{\kappa}
\end{bmatrix}
\begin{bmatrix}v^{2}({v^{2}}-1)^{-1}+\kappa^{2}N_{\kappa}&\kappa^{2}1_{\Omega}\SL_{\kappa}\\-\gamma^{\-}_{1}N_{\kappa}&{\widetilde\rho\,}^{-1}-\gamma_{1}\SL_{\kappa}
\end{bmatrix}^{\!-1}
\begin{bmatrix}\Delta^{\max}_{\Omega }1_{\Omega}R_{\kappa}\\ \gamma_{1}R_{\kappa}
\end{bmatrix}
\end{align*}
\end{remark}
\subsection{Example: volume perturbations.}\label{vol-pert} Taking $\rho_{1}=\dots=\rho_{n}=1$  in Theorem \ref{RD}, one gets an only regular perturbation of $\Delta$ supported on $\Omega$:
\begin{align*}
&(-A_{v,1}-\kappa^{2})^{-1}\\=&R_{\kappa}
+\!\!\begin{bmatrix}R_{\kappa}1_{\Omega}^{*}& \SL_{\kappa}
\end{bmatrix}\!\!
\begin{bmatrix}v^{2}+(v^{2}-1)\kappa^{2}N_{\kappa}&(v^{2}-1)\kappa^{2}1_{\Omega}\SL_{\kappa}\\
0&1
\end{bmatrix}^{\!-1}\!\!
\begin{bmatrix}(v^{2}-1)\Delta^{\max}_{\Omega }1_{\Omega}R_{\kappa}\\ 0
\end{bmatrix}\\
=&R_{\kappa}\!\!
+\!\!\begin{bmatrix}R_{\kappa}1_{\Omega}^{*}& \SL_{\kappa}
\end{bmatrix}\!\!
\begin{bmatrix}(v^{2}+(v^{2}-1)\kappa^{2}N_{\kappa})^{-1}&
(v^{2}+(v^{2}-1)\kappa^{2}N_{\kappa})^{-1}(1-v^{2})\kappa^{2}1_{\Omega}\SL_{\kappa}\\
0&1
\end{bmatrix}\times\\
&\ \,\times\begin{bmatrix}(v^{2}-1)\Delta^{\max}_{\Omega }1_{\Omega}R_{\kappa}\\ 0
\end{bmatrix}
\\
=&R_{\kappa}+(v^{2}-1)R_{\kappa}1_{\Omega}^{*}(v^{2}+(v^{2}-1)\kappa^{2}N_{\kappa})^{-1}\Delta^{\max}_{\Omega }1_{\Omega}R_{\kappa}\,.
\end{align*}
\subsection{Example: surface perturbations.}\label{surf-pert} Taking $v_{1}=\dots=v_{n}=1$  in Theorem \ref{RD}, one gets an only singular perturbation of $\Delta$ supported on $\Gamma$:
\begin{align*}
&(-A_{1,\rho}-\kappa^{2})^{-1}=R_{\kappa}
+\!\!\begin{bmatrix}R_{\kappa}1_{\Omega}^{*}& \SL_{\kappa}
\end{bmatrix}\!\!
\begin{bmatrix}1&0\\
-\widetilde\rho\,\gamma^{\-}_{1}N_{\kappa}&1-\widetilde\rho\,\gamma_{1}\SL_{\kappa}
\end{bmatrix}^{\!-1}\!\!
\begin{bmatrix}0\\ \widetilde\rho\,\gamma_{1}R_{\kappa}
\end{bmatrix}\\
=&R_{\kappa}
+\!\!\begin{bmatrix}R_{\kappa}1_{\Omega}^{*}& \SL_{\kappa}
\end{bmatrix}\!\!
\begin{bmatrix}1&0\\
-\left(1-\widetilde\rho\,\gamma_{1}\SL_{\kappa}\right)^{-1}\widetilde\rho\,\gamma^{\-}_{1}N_{\kappa}&\left(1-\widetilde\rho\,\gamma_{1}\SL_{\kappa}\right)^{-1}
\end{bmatrix}\!\!
\begin{bmatrix}0\\ \widetilde\rho\,\gamma_{1}R_{\kappa}
\end{bmatrix}\\
=&R_{\kappa}
+ \SL_{\kappa}
\left(1-\widetilde\rho\,\gamma_{1}\SL_{\kappa}\right)^{-1}
\widetilde\rho\,\gamma_{1}R_{\kappa}\,.
\end{align*}
\begin{remark}\label{deco} By the decomposition
$$
L^{2}(\Omega)\oplus H^{-1/2}(\Gamma)=
L^{2}(\Omega_{1})\oplus \dots \oplus L^{2}(\Omega_{n}) \oplus H^{-1/2}(\Gamma_{1})\oplus \dots \oplus H^{-1/2}(\Gamma_{n})\,,
$$
denoting by $\SL_{k}^{(\ell)}$ the single layer operator for $\Gamma_{\ell}$ and introducing the notation
$$
\text{vect}(L_{\ell}):=\begin{bmatrix}L_{1}\\L_{2}\\ \vdots \\L_{n}\end{bmatrix},\qquad
\text{vect}(L_{\ell})^{\top}:=\begin{bmatrix}L_{1}&L_{2}&\dots &L_{n}\end{bmatrix},
$$
$$
\text{mat}(L_{\ell m}):=
 \begin{bmatrix}
    L_{11} & L_{12} & \dots & L_{1n} \\
    L_{21} & L_{22} & \dots & L_{2n} \\
    \vdots & \vdots & \ddots & \vdots \\
    L_{n1} & L_{n2} & \dots & L_{nn} \,
  \end{bmatrix},
$$
the resolvent difference formula in Theorem \ref{RD} rewrites as
\begin{align*}
&(-A_{v,\rho }-\kappa^{2})^{-1}-R_{\kappa}=\begin{bmatrix}
\text{vect}(R_{\kappa}1_{\Omega_{\ell}}^{*})^{\top}&\text{vect}(\SL_{\kappa}^{(\ell)})^{\top}
\end{bmatrix}\times\\
&\times
\begin{bmatrix}\text{mat}(v_{\ell}^{2}+({v_{\ell}^{2}}-1)\kappa^{2}1_{\Omega_{\ell}}R_{\kappa}1_{\Omega_{m}}^{*})&
\text{mat}(({v_{\ell}^{2}}-1)\kappa^{2}1_{\Omega_{\ell}}\SL^{(m)}_{\kappa})\\
-\text{mat}(\widetilde\rho_{\ell}
\,\gamma_{1,\ell}R_{\kappa}1_{\Omega_{m}}^{*})&\text{mat}(1-\widetilde\rho_{\ell}\,\gamma_{1,\ell}\SL^{(m)}_{\kappa})
\end{bmatrix}^{\!-1}
\begin{bmatrix}\text{vect}(({v_{\ell}^{2}}-1)\Delta^{\max}_{\Omega_{\ell} }1_{\Omega_{\ell}}R_{\kappa})\\ 
\text{vect}(\widetilde\rho_{\ell}\,\gamma_{1,\ell}R_{\kappa})
\end{bmatrix}.
\end{align*} 
By Remark \ref{not1}, whenever $v_{\ell}\not=1$ and $p_{\ell}\not=1$ for any index $\ell$, one gets
\begin{align*}
&(-A_{v,\rho }-\kappa^{2})^{-1}-R_{\kappa}\\=&\begin{bmatrix}
\text{vect}(R_{\kappa}1_{\Omega_{\ell}}^{*})^{\top}&\text{vect}(\SL_{\kappa}^{(\ell)})^{\top}
\end{bmatrix}\times\\
&\times
\begin{bmatrix}\text{mat}(v_{\ell}^{2}({v_{\ell}^{2}}-1)^{-1}+\kappa^{2}1_{\Omega_{\ell}}R_{\kappa}1_{\Omega_{m}}^{*})&
\text{mat}(\kappa^{2}1_{\Omega_{\ell}}\SL^{(m)}_{\kappa})\\
-\text{mat}(\gamma_{1,\ell}R_{\kappa}1_{\Omega_{m}}^{*})&\text{mat}({\widetilde\rho_{\ell}}^{\,\,-1}-\gamma_{1,\ell}\SL^{(m)}_{\kappa})
\end{bmatrix}^{\!-1}
\begin{bmatrix}\text{vect}(\Delta^{\max}_{\Omega_{\ell} }1_{\Omega_{\ell}}R_{\kappa})\\ 
\text{vect}(\gamma_{1,\ell}R_{\kappa})
\end{bmatrix}.
\end{align*} 
\end{remark}
\subsection{Example: the case of many equal connected components}\label{cup} Let us suppose that 
$$
\rho_{1}=\dots=\rho_{n}=\rho\,,\qquad v_{1}=\dots=v_{n}=v\,,
$$
\be\label{On}
\Omega=\Omega_{1}\cup\dots\cup\Omega_{n}\,,\qquad \Omega_{\ell}=y_{\ell}+\Omega_{\circ}\,,
\ee 
where the $y_{\ell}$'s are such that $\overline \Omega_{\ell}\cap\overline\Omega_{m}=\varnothing$ whenever $\ell\not=m$.\par
Introducing the maps $\varphi_{\ell}(x):=y_{\ell}+x$, such that $\varphi_{\ell}(\Omega_{\circ})=\Omega_{\ell}$, we define
$$
\Psi_{\ell}:L^{2}(\RE^{3})\to L^{2}(\RE^{3})\,,\qquad 
\Psi_{\ell}u:=u\circ\varphi_{\ell}
$$
$$
\Phi_{\ell}:L^{2}(\Omega_{\ell})\to L^{2}(\Omega_{\circ})\,,\qquad 
\Phi_{\ell}u:=u\circ(\varphi_{\ell}|\Omega_{\circ})
$$ 
and 
$$
\Phi^{\partial}_{\ell}:L^{2}(\Gamma_{\ell})\to L^{2}(\Gamma_{\circ})\,,\qquad 
\Phi^{\partial}_{\ell}\phi:=\phi\circ(\varphi_{\ell}|\Gamma_{\circ})\,,\qquad \Gamma_{\circ}:=\partial\Omega_{\circ}\,.
$$
One has the identities
$$
1_{\Omega_{\ell}}=\Phi_{\ell}^{-1}1_{\Omega_{\circ}}\Psi_{\ell}\,,
\qquad
1^{*}_{\Omega_{\ell}}=\Psi_{\ell}^{-1}1^{*}_{\Omega_{\circ}}\Phi_{\ell}\,,
$$
$$
\Delta^{\max}_{\Omega_{\ell}}=\Phi_{\ell}^{-1}\Delta^{\max}_{\Omega_{\circ}}\Phi_{\ell}\,,
\qquad
R_{\kappa}=\Psi_{\ell}^{-1}R_{\kappa}\Psi_{\ell}\,,
$$
$$
\SL^{(\ell)}_{\kappa}=\Psi_{\ell}^{-1}\SL^{\circ}_{\kappa}\Phi^{\partial}_{\ell}\,,
\qquad
\gamma_{1}^{(\ell)}=(\Phi^{\partial}_{\ell})^{-1}\gamma^{\circ}_{1}\Psi_{\ell}\,.
$$
Here, $ \SL^{\circ}_{\kappa}$ and $\gamma^{\circ}_{1}$ denote the operators corresponding to the pivot domain $\Omega_{\circ}$. 
Introducing the notation
$$
\text{diag}(L_{\ell}):=\begin{bmatrix}
    L_{1}& 0& \dots & 0 \\
    0& L_{2}& \dots & 0 \\
    \vdots & \vdots & \ddots & \vdots \\
   0 & 0 & \dots & L_{n}\,
  \end{bmatrix},
$$
the above relations give
\begin{align*}
&\begin{bmatrix}\text{vect}(R_{\kappa}1^{*}_{\Omega_{\ell}})^{\top}&
\text{vect}(\SL^{(\ell)}_{\kappa})^{\top}\end{bmatrix}
=\begin{bmatrix}\text{vect}(\Psi_{\ell}^{-1}R_{\kappa}1^{*}_{\Omega_{\circ}}\Phi_{\ell})^{\top}&
\text{vect}(\Psi_{\ell}^{-1}\SL^{\circ}_{\kappa}\Phi^{\partial}_{\ell})^{\top}\,\end{bmatrix}\\
=&
\begin{bmatrix}\text{diag}(\Psi_{\ell})&0\\
0&\text{diag}(\Psi_{\ell})
\end{bmatrix}^{-1}
\begin{bmatrix}\text{vect}(R_{\kappa}1^{*}_{\Omega_{\circ}})^{\top}&
\text{vect}(\SL^{\circ}_{\kappa})^{\top}\,\end{bmatrix}
\begin{bmatrix}\text{diag}(\Phi_{\ell})&0\\
0&\text{diag}(\Phi^{\partial}_{\ell})
\end{bmatrix},
\end{align*}
\begin{align*}
&\begin{bmatrix}(v^{2}-1)\text{vect}(\Delta^{\max}_{\Omega_{\ell}}1_{\Omega_{\ell}}R_{\kappa})\\ 2\,\frac{\rho-1}{\rho+1}\,\text{vect}(\gamma_{1}^{(\ell)}R_{\kappa})
\end{bmatrix}=
\begin{bmatrix}({v^{2}}-1)\text{vect}(\Phi_{\ell}^{-1}\Delta^{\max}_{\Omega_{\circ}}1_{\Omega_{\circ}}R_{\kappa}\Psi_{\ell})\\
 2\,\frac{\rho-1}{\rho+1}\,\text{vect}((\Phi^{\partial}_{\ell})^{-1}\gamma^{\circ}_{1}R_{\kappa}\Psi_{\ell})\end{bmatrix}\\
=&\begin{bmatrix}\text{diag}(\Phi_{\ell})&0\\
0&\text{diag}(\Phi^{\partial}_{\ell})
\end{bmatrix}^{-1}
\begin{bmatrix}(v^{2}-1)\,\text{vect}(\Delta^{\max}_{\Omega_{\circ}}1_{\Omega_{\circ}}R_{\kappa})\\
 2\,\frac{\rho-1}{\rho+1}\,\text{vect}(\gamma^{\circ}_{1}R_{\kappa})\end{bmatrix}
\begin{bmatrix}\text{diag}(\Psi_{\ell})&0\\
0&\text{diag}(\Psi_{\ell})
\end{bmatrix},
\end{align*}
\begin{align*}
&\begin{bmatrix}\text{mat}(v^{2}+({v^{2}}-1)\kappa^{2}1_{\Omega_{\ell}}R_{\kappa}1^{*}_{\Omega_{m}})&
\text{mat}(({v^{2}}-1)\kappa^{2}1_{\Omega_{\ell}}\SL_{\kappa}^{(m)})
\\
2\,\frac{1-\rho}{1+\rho}\,\text{mat}(\gamma_{1}^{(\ell)\-}1_{\Omega_{\ell}}R_{\kappa}1_{\Omega_{m}}^{*})&\text{mat}(1+2\,\frac{1-\rho}{1+\rho}\,\gamma_{1}^{(\ell)}\SL_{\kappa}^{(m)})
\end{bmatrix}\\
=&\begin{bmatrix}\text{mat}(v^{2}+({v^{2}}-1)\kappa^{2}\Phi_{\ell}^{-1}1_{\Omega_{\circ}}\Psi_{\ell}\Psi^{-1}_{m}R_{\kappa}1_{\Omega_{\circ}}^{*}\Phi_{m})&
(v^{2}-1)\kappa^{2}
\text{mat}(\Phi_{\ell}^{-1}1_{\Omega_{\circ}}\Psi_{\ell}\Psi_{m}^{-1}\SL^{\circ}_{\kappa}\Phi^{\partial}_{m})
\\
2\,\frac{1-\rho}{1+\rho}\,\text{mat}((\Phi^{\partial}_{\ell})^{-1}\gamma^{\circ\-}_{1}1_{\Omega_{\circ}}\Psi_{\ell}\Psi^{-1}_{m}R_{\kappa}1_{\Omega_{\circ}}^{*}\Phi_{m})&\text{mat}(1+2\,\frac{1-\rho}{1+\rho}\,(\Phi^{\partial}_{\ell})^{-1}\gamma^{\circ}_{1}\Psi_{\ell}\Psi_{m}^{-1}\SL^{\circ}_{\kappa}\Phi^{\partial}_{m})
\end{bmatrix}
\\
=&\begin{bmatrix}\text{diag}(\Phi_{\ell})&0\\
0&\text{diag}(\Phi^{\partial}_{\ell})
\end{bmatrix}^{-1}
\times\\&\times
\begin{bmatrix}\text{mat}(v^{2}+({v^{2}}-1)\kappa^{2}1_{\Omega_{\circ}}\Psi_{\ell m}R_{\kappa}1_{\Omega_{\circ}}^{*})&({v^{2}}-1)\kappa^{2}
\text{mat}(1_{\Omega_{\circ}}\Psi_{\ell m}\SL^{\circ}_{\kappa})
\\
2\,\frac{1-\rho}{1+\rho}\,\text{mat}(\gamma^{\circ\-}_{1}1_{\Omega_{\circ}}\Psi_{\ell m}R_{\kappa}1_{\Omega_{\circ}}^{*})&\text{mat}(1+2\,\frac{1-\rho}{1+\rho}\,
\gamma^{\circ}_{1}\Psi_{\ell m}\SL^{\circ}_{\kappa})
\end{bmatrix}
\times\\&\times
\begin{bmatrix}\text{diag}(\Phi_{m})&0\\
0&\text{diag}(\Phi^{\partial}_{m})
\end{bmatrix}
\end{align*}
where $\Psi_{\ell m}:=\Psi_{\ell}\Psi_{m}^{-1}$ acts as $\Psi_{\ell}$, replacing $\varphi_{\ell}$ with $\varphi_{\ell m}(x):=x+(y_{\ell}-y_{m})$.\par
 Therefore, by Theorem \ref{RD} and  Remark \ref{deco},
 \begin{align*}
&(-A_{v,\rho }-\kappa^{2})^{-1}-R_{\kappa}
=\begin{bmatrix}\text{diag}(\Psi_{\ell})&0\\
0&\text{diag}(\Psi_{\ell})
\end{bmatrix}^{-1}\begin{bmatrix}\text{vect}( R_{\kappa}1^{*}_{\Omega_{\circ}})^{\top}&\text{vect}(\SL^{\circ}_{\kappa})^{\top} \,\end{bmatrix}\times\\
&\times
\begin{bmatrix}\text{mat}(v^{2}+({v^{2}}-1)\kappa^{2}1_{\Omega_{\circ}}\Psi_{\ell m}R_{\kappa}1_{\Omega_{\circ}}^{*})&({v^{2}}-1)\kappa^{2}
\text{mat}(1_{\Omega_{\circ}}\Psi_{\ell m}\SL^{\circ}_{\kappa})
\\
2\,\frac{1-\rho}{1+\rho}\,\text{mat}(\gamma^{\circ\-}_{1}1_{\Omega_{\circ}}\Psi_{\ell m}R_{\kappa}1_{\Omega_{\circ}}^{*})&\text{mat}(1+2\,\frac{1-\rho}{1+\rho}\,\gamma^{\circ}_{1}\Psi_{\ell m}\SL^{\circ}_{\kappa})
\end{bmatrix}
^{-1}\times\\
&\times\begin{bmatrix}(v^{2}-1)\,\text{vect}(\Delta^{\max}_{\Omega_{\circ}}1_{\Omega_{\circ}}R_{\kappa})\\
 2\,\frac{\rho-1}{\rho+1}\,\text{vect}(\gamma^{\circ}_{1}R_{\kappa})\end{bmatrix}
 \begin{bmatrix}\text{diag}(\Psi_{\ell})&0\\
0&\text{diag}(\Psi_{\ell})
\end{bmatrix}.
\end{align*}
By Remark \ref{not1}, whenever $v_{\ell}\not=1$ and $\rho_{\ell}\not=1$ for any index $\ell$, one gets
 \begin{align*}
&(-A_{v,\rho }-\kappa^{2})^{-1}-R_{\kappa}
=\begin{bmatrix}\text{diag}(\Psi_{\ell})&0\\
0&\text{diag}(\Psi_{\ell})
\end{bmatrix}^{-1}\begin{bmatrix}\text{vect}( R_{\kappa}1^{*}_{\Omega_{\circ}})^{\top}&\text{vect}(\SL^{\circ}_{\kappa})^{\top} \,\end{bmatrix}\times\\
&\times
\begin{bmatrix}\text{mat}(v^{2}({v^{2}}-1)^{-1}+\kappa^{2}1_{\Omega_{\circ}}\Psi_{\ell m}R_{\kappa}1_{\Omega_{\circ}}^{*})&\kappa^{2}
\text{mat}(1_{\Omega_{\circ}}\Psi_{\ell m}\SL^{\circ}_{\kappa})
\\
\text{mat}(\gamma^{\circ\-}_{1}1_{\Omega_{\circ}}\Psi_{\ell m}R_{\kappa}1_{\Omega_{\circ}}^{*})&\text{mat}(\frac12\,\frac{1+\rho}{1-\rho}-\gamma^{\circ}_{1}\Psi_{\ell m}\SL^{\circ}_{\kappa})
\end{bmatrix}
^{-1}\times\\
&\times\begin{bmatrix}\text{vect}(\Delta^{\max}_{\Omega_{\circ}}1_{\Omega_{\circ}}R_{\kappa})\\
 \text{vect}(\gamma^{\circ}_{1}R_{\kappa})\end{bmatrix}
 \begin{bmatrix}\text{diag}(\Psi_{\ell})&0\\
0&\text{diag}(\Psi_{\ell})
\end{bmatrix}.
\end{align*}
\section{The Spectrum}
Here, following the same kind of arguments as in \cite[Section 2]{ZAMP}, we study the spectrum of $A_{v,\rho }$.
\begin{lemma}\label{sp}
$\sigma_{p}(  A_{v,\rho })    =\varnothing$.
\end{lemma}
\begin{proof}
Since $A_{v,\rho }\le 0$, it suffices to show that $\sigma_{p}( A_{v,\rho } 
) \cap (-\infty,0]=\varnothing$. Given $\lambda< 0$, let $u_{\lambda}\in {\mathscr D}_{\rho}$
be such that $A_{v,\rho }u_{\lambda}=\lambda  u_{\lambda}$. By \eqref{op-incl1}, 
$$
(\Delta^{\max}_{\Omega_{\+}}-\lambda)(u_{\lambda}|\Omega_{\+})=0\,.
$$
By Rellich's estimate (see, e.g., \cite[Corollary 4.8]{Leis}), $u_{\lambda}|B^{c}=0$ for any open ball $B\supset\overline \Omega$; hence, by the unique continuation property (see, e.g., \cite[Section 4.3]{Leis}),  $u_{\lambda}|\Omega_{\+}=0$ and so $\gamma_{0}^{\+}u_{\lambda}=\gamma_{1}^{\+}u_{\lambda}=0$. Hence, by \eqref{bc}, 
\[
\begin{cases}
(    v^{2}\Delta^{\max}_{\Omega}-\lambda)    (u_{\lambda}|\Omega)=0\\
\gamma_{0}^{\-}(u_{\lambda}|\Omega)=\gamma_{1}^{\-}(u_{\lambda}|\Omega)=0\,.
\end{cases}
\]
The unique solution of such a boundary value problem is $u_{\lambda}|\Omega=0$. Therefore, $u_{\lambda}=0$.\par
As regards the case $\lambda=0$, let $u_{0}\in {\mathscr D}_{\rho}$
such that $A_{v,\rho}\,u_{0}=0$. Then, by \eqref{sesqui},
$$
0=\langle A_{v,\rho}\,u_{0},u_{0}\rangle_{L^{2}(\RE^{3},b^{-1})}={\mathscr Q}(u_{0},u_{0})\gtrsim
\|\nabla
u_{0}\|_{L^{2}(  \mathbb{R}^{3})  }^{2}\,.
$$ 
This entails $u_{0}=0$.  
\end{proof}
\begin{lemma}
\label{Lemma_ess}$\sigma_{ess}( A_{v,\rho })    =(-\infty, 0]    $.
\end{lemma}
\begin{proof} 
By\footnote{The Besov-like space $B^{s}_{2,2}(\Gamma)$, $s>0$, denotes the Dirichlet trace space relative to $H^{s+1/2}(\Omega)$; its dual space is denoted by $B^{-s}_{2,2}(\Gamma)$. Whenever $0<s\le1$, $B^{\pm s}_{2,2}(\Gamma)$ coincides with $H^{\pm s}(\Gamma)$. See, e.g., \cite{Trib} for more details. } $\SL_{\kappa}\in\B(B^{-3/2}_{2,2}(\Gamma), L^{2}(\RE^{3}))$ and by the compact embedding $H^{-1/2}(\Gamma)\hookrightarrow B^{-3/2}_{2,2}(\Gamma)$, one has $\SL_{\kappa}\in\mathfrak{S}_{\infty}(H^{-1/2}(\Gamma), L^{2}(\RE^{3}))$. By $R_{\kappa}\in\B(L^{2}(\RE^{3}), H^{2}(\RE^{3}))$ and by the compact embedding $H^{2}(\Omega)\hookrightarrow L^{2}(\Omega) $, one has $1_{\Omega}R_{\kappa}\in \mathfrak{S}_{\infty}(L^{2}(\RE^{3}), L^{2}(\Omega))$; then, by duality, 
$R_{\kappa}1_{\Omega}^{*}\in \mathfrak{S}_{\infty}(L^{2}(\Omega), L^{2}(\RE^{3})) $. 
By the resolvent difference formula in Theorem \ref{RD}, $(-A_{v,\rho }-\kappa^{2})^{-1}-R_{\kappa}$ is compact as well.  
Therefore, the Weyl theorem applies (see, e.g., \cite[Th.
XIII.14]{ReSi IV}) and $\sigma_{ess}( A_{v,\rho})    =\sigma_{ess}(  \Delta)=(-\infty,0]$.
\end{proof}
\subsection{A limiting absorption principle} \hfill
\vskip5pt\par\noindent
Here, we prove that there is no singular continuous spectrum, so that  the spectrum of  $A_{v,\rho}$ is purely absolutely continuous. To this end, arguing as in \cite[Section 2.1]{ZAMP}, we at first show that a limiting absorption principle holds for   $A_{v,\rho}$. As it is well known, a limiting absorption principle holds for the free Laplacian (see, e.g., \cite[Section 4]{Agm}):\par
For any $\lambda\in\RE\backslash\{0\}$ and for any $\alpha>{\frac12}$, the  limits 
\begin{equation}
\lim_{\delta\searrow0}(-\Delta+(\lambda\pm i\delta))^{-1}
\label{lim1}%
\end{equation}
exist in $\B(L^{2}_{\alpha}(\RE^{3}),L^{2}_{-\alpha}(
\mathbb{R}^{3}) )$; the same hold true in the case $\lambda=0$ whenever $\alpha>3/2$.
Here, $L^{2}_{\alpha}(\mathbb{R}^{3})$ denotes the weighted $L^{2}$-spaces with weight $w(x)=(1+|x|^{2})^{\alpha}$ (see \cite[Section 2]{Agm} for more details).
As regards $A_{v,\rho}$, the result is of the same kind:
\begin{theorem}
\label{Theorem_LAP} The limits%
\begin{equation}
\widetilde{\mathcal R}^{\pm}_{\lambda}:=
\lim_{\delta\searrow0}(-A_{v,\rho}+\lambda\pm i\delta)^{-1}
\,, \label{LAP_eps}%
\end{equation}
exist in ${{\B }}(    L^{2}_{\alpha}(  \mathbb{R}^{3})
,L^{2}_{-\alpha}(  \mathbb{R}^{3})  )    $ for any $\alpha>1/2$ and any $\lambda\in\mathbb{R}\backslash\{0\}$; the same hold true in the case $\lambda=0$ whenever $\alpha>3/2$. 
Furthermore, the ${{\B }}(    L^{2}_{\alpha}(  \mathbb{R}^{3})
,L^{2}_{-\alpha}(  \mathbb{R}^{3})  )    $-valued extended resolvents 
$$
z\mapsto
\widetilde{\mathcal R}^{\pm}_{z}:=
\begin{cases}(-A_{v,\rho}+z)^{-1}&z\in\CO\backslash(-\infty,0]\\
\widetilde{\mathcal R}^{\pm}_{\lambda}&z=\lambda\in(-\infty,0]
\end{cases}
$$
are continuous   on $\CO_{\pm}\cup\RE$.
\end{theorem}
\begin{proof} For any $z\in\CO\backslash(-\infty,0]$, we use the abbreviated notation
$$
\widetilde{\mathcal R}_{z}\equiv(-A_{v,\rho}+z)^{-1}\,,\qquad
{\mathcal R}_{z}\equiv(-\Delta+z)^{-1}
$$
and define $\mathcal{S\!L}_{z}:=(\gamma_{0}{\mathcal R}_{\bar z})^{*}$. Notice that $\mathcal{S\!L}_{\kappa^{2}}=\SL_{\kappa}$.\par
By \cite[eqn. (4.8)]{JST18}, one has $\mathcal{R}_{z}\in \mathscr B(L_{\alpha }^{2}(\mathbb{R}^{3}),H_{\alpha }^{2}(\mathbb{R}^{3}))$ for any real $\alpha$. Then, by duality and interpolation  
\be\label{int}
\mathcal{R}_{z}\in \mathscr B(H_{\alpha }^{-s}(\mathbb{R}^{3}),H_{\alpha }^{2-s}(\mathbb{R}^{3}))\,,\qquad 0\le s\le 2\,.
\ee
In particular,
\be
\mathcal{R}_{z}\in \mathscr B(L_{\alpha }^{2}(\mathbb{R}^{3}))
\,.
\label{H1.1}
\ee
By \eqref{int},  $\mathcal{S\!L}_{z}\in \B(H^{-1/2}(\Gamma),L_{\alpha }^{2}(\mathbb{R}^{3}))$, and, by $1^{*}_{\Omega }\in \B(L^{2}(\Omega),L_{\alpha }^{2}(\mathbb{R}^{3}))$,  one obtains $
\mathcal{R}_{z}1^{*}_{\Omega }\in \mathscr B(L^{2}(\Omega),L_{\alpha }^{2}(\mathbb{R}^{3}))$. Thus, by the resolvent difference formula in Theorem \ref{RD},
\begin{equation}
\widetilde{\mathcal{R}}_{z}\in \mathscr B(L_{\alpha }^{2}(\mathbb{R}%
^{3}))\,,
\qquad \alpha \in 
\mathbb{R}\,.
\label{H1.2}
\end{equation}
Since $1_{\Omega}{\mathcal{R}}_{z}\in \B(L^{2}_{-\beta}(\RE^{3}), H^{2}(\Omega))$, by the compact embedding $H^{2}(\Omega)\hookrightarrow L^{2}(\Omega)$, one gets $1_{\Omega}{\mathcal{R}}_{z}\in {\mathfrak{S}}_{\infty }(L_{-\beta }^{2}(\mathbb{R}^{3}),L^{2}(\Omega))$ for any real $\beta $. Thus, by duality,  ${\mathcal{R}}_{z}1^{*}_{\Omega}\in {\mathfrak{S}}_{\infty }(L^{2}(\Omega),L_{\beta}^{2}(\mathbb{R}^{3}))$. By \eqref{int}, 
$\mathcal{S\!L}_{z}\in \B(B^{-3/2}_{2,2}(\Gamma), L^{2}_{\beta}(\RE^{3}))$ and by the compact embedding 
$H^{-1/2}(\Gamma)\hookrightarrow B^{-3/2}_{2,2}(\Gamma) $, one gets 
$\mathcal{S\!L}_{z}\in \mathfrak{S}_{\infty}(H^{-1/2}(\Gamma), L^{2}_{\beta}(\RE^{3}))$.
Hence, by the resolvent difference formula in Theorem \ref{RD},
\begin{equation}
\widetilde{\mathcal{R}}_{z}-\mathcal{R}_{z}\in {\mathfrak{S}}_{\infty
}(L^{2}(\mathbb{R}^{n}),L_{\beta }^{2}(\mathbb{R}^{3}))\,,\quad \beta
>2\alpha \,.
\label{H2}
\end{equation}
Furthermore, according to \cite[Corollary 5.7(b)]{BAD}, for all compact
subset $K\subset (0,+\infty )$ there exists a constant $c_{\nu}>0$ such that,
for $\lambda \in K$ one has 
\begin{equation}
\forall u\in L_{2\alpha }^{2}(\mathbb{R}^{n})\cap \ker (\mathcal{R}_{\lambda
}^{+}-\mathcal{R}_{\lambda }^{-}),\quad \Vert \mathcal{R}_{\lambda }^{\pm
}u\Vert _{L^{2}(\mathbb{R}^{3})}\leq c_{\nu}\Vert u\Vert _{L_{2\alpha }^{2}(%
\mathbb{R}^{3})}\,.  \label{H3}
\end{equation}
The relations \eqref{H1.1}-\eqref{H3} permit us to apply the abstract
results provided in \cite{Ren1}, where, with respect to the notations there, $\mathcal{H}_{1}=L^{2}(\mathbb{R}^{3})$, $\mathcal{H}_{2}=L^{2}(\RE^{3}\!,b^{-1} )$, $X=L_{\alpha }^{2}(\mathbb{R}^{3})$, $H_{1}=-\Delta $,  $H_{2}=-A_{v,\rho}$, $J_{1}$ the identity in $\mathcal{H}_{1}$, $J_{2}:\mathcal{H}%
_{1}\rightarrow \mathcal{H}_{2}$ the multiplication by $\sqrt{b}$. Hypothesis (T1) and (E1) in 
\cite[page 175]{Ren1} corresponds to our \eqref{lim1}, \eqref{H3} and %
\eqref{H2} respectively; then, by \cite[Proposition 4.2]{Ren1}, the latter
imply hypotheses (LAP) and (E) in \cite[page 166]{Ren1}, i.e. \eqref{lim1}
again and 
\begin{equation*}
J_{2}^{\ast }\widetilde{\mathcal{R}}_{z}J_{2}-\mathcal{R}_{z }\in {\mathfrak{S}}%
_{\infty }(L_{-\alpha }^{2}(\mathbb{R}^{3}),L_{\alpha }^{2}(\mathbb{R}%
^{3}))\,,
\end{equation*}
and hypothesis (T) in \cite[page 168]{Ren1}, a technical variant of %
\eqref{H3}. By \cite[Theorem 3.5]{Ren1}, these last hypotheses, together
with $-A_{v,\rho}\geq 0$ and \eqref{H1.1}-\eqref{H1.2} (i.e. hypothesis
(OP) in \cite[page 165]{Ren1}), give the existence of the limits \eqref{LAP_eps} outside $\sigma_{p}(-A_{v,\rho})$; the latter is empty by Lemma \ref{sp}.
\end{proof}
By the obvious relation
$$
\CO_{+}\ni\kappa\mapsto(-A_{v,\rho}-\kappa^{2})^{-1}=
\big(-A_{v,\rho}-\text{Re}(\kappa)^{2}+\text{Im}(\kappa)^{2}-i2\text{Re}(\kappa)\text{Im}(\kappa)\big)^{-1}\,,
$$
Theorem \ref{Theorem_LAP} implies
\begin{corollary}\label{pseudo-LAP} The  ${{\B }}(    L^{2}_{\alpha}(  \mathbb{R}^{3})
,L^{2}_{-\alpha}(  \mathbb{R}^{3})  )    $-valued map
$$
\kappa\mapsto \widetilde R_{\kappa}:=
\begin{cases}(-A_{v,\rho}-\kappa^{2})^{-1}&\text{\rm Im}(\kappa)>0\\
\widetilde {\mathcal R}^{\pm}_{-\kappa^{2}}&\text{\rm Im}(\kappa)=0\,,\ \mp\text{\rm Re}(\kappa)>0
\end{cases}
$$
is continuous  on $\CO_{+}\cup\RE\backslash\{0\}$.
\end{corollary}
\begin{remark} By \cite[Theorem 18.3 ii)]{Kom} and by the continuity of $\kappa\mapsto Q_{\kappa}^{-1}$ in a neighborhood of $0\in \CO$, it could be shown that $\kappa\mapsto \widetilde R_{\kappa}$  is in fact continuous on the whole $\CO_{+}\cup\RE$. However, for our purposes (see the proof of Lemma \ref{res} below) the statement in Corollary \ref{pseudo-LAP} suffices.
\end{remark}
By the self-adjointness of $A_{v,\rho}$ in $L^{2}(\RE^{3}\!,b^{-1})$, we have the orthogonal decomposition%
\[
L^{2}(\RE^{3}\!,b^{-1})  =   L^{2}(    \RE^{3}\!,b^{-1})        _{c}\oplus  L^{2}(\RE^{3}\!,b^{-1})  _{pp}\,,
\]
where $    L^{2}(\RE^{3}\!,b^{-1})      _{c}$ and $  
L^{2}(\RE^{3}\!,b^{-1})      _{pp}$ respectively denote the
continuous and pure point subspaces of $A_{v,\rho}$. Let $P_{\lambda
}    $ be the spectral resolution of the identity
associated with $A_{v,\rho}$. Recall that if $f\in   L^{2}(\RE^{3}\!,b^{-1})       _{c}=    (  L^{2}(\RE^{3}\!,b^{-1})  _{pp})    ^{\bot}$, $\lambda\mapsto\left\langle
f,P_{\lambda}    f\right\rangle_{L^{2}(\RE^{3})} $ is a continuous
function on $\sigma_{c}(A_{v,\rho})    $. The continuous
subspace $    L^{2}(\RE^{3}\!,b^{-1})      _{c}$ further
decomposes as an orthogonal sum of the absolutely continuous and singular
subspaces of $A_{v,\rho}$%
\[
   L^{2}(\RE^{3}\!,b^{-1})      _{c}=    L^{2}(\RE^{3}\!,b^{-1})      _{ac}\oplus   L^{2}(    \mathbb{R}%
^{3}\!,b^{-1})        _{sc}\,.
\]
The absolute continuous part of the spectrum ${\sigma}_{{ac}%
}(A_{v,\rho} )    $ is defined by the condition that, if
$f\in   L^{2}(\RE^{3}\!,b^{-1})     _{ac}$, then
$\lambda\mapsto\left\langle f,P_{\lambda}
f\right\rangle_{L^{2}(\RE^{3})} $ is absolutely continuous on $\sigma_{{ac}%
}( A_{v,\rho})    $. Hence, the absence of singular continuous
spectrum corresponds to the identity%
\[
    L^{2}(\RE^{3}\!,b^{-1})      _{ac}=    (  
L^{2}(\RE^{3}\!,b^{-1})     _{pp})    ^{\bot}\,.
\]
By Theorem \ref{Theorem_LAP}, we get 

\begin{corollary}
\label{Prop_empty_sc}%
\[%
\begin{array}
[c]{ccc}%
\sigma_{ac}(  A_{v,\rho})    =(-\infty,0]
    \,, &  & \sigma_{sc}(  A_{v,\rho}  )    =\varnothing\,.
\end{array}
\]
\end{corollary}
\begin{proof}
Arguing as in \cite[Proof of Theorem 6.1]{Agm}, in order to show that
$\sigma_{c}(A_{v,\rho})=\sigma_{ac}(A_{v,\rho})$, it suffices to show that
$\lambda\mapsto\left\langle f,P_{\lambda}
f\right\rangle $ is absolutely continuous in $\RE\backslash\{0\}$ for any 
$f\in((L^{2}(\RE^{3}\!,b^{-1}))^{pp})^{\bot}$.
Given the compact interval $[\lambda_{1},\lambda_{2}]$ and $f\in L^{2}_{\alpha}
(\mathbb{R}^{3})$, by Theorem \ref{Theorem_LAP}
and by Stone's formula one has
\begin{equation*}
\big\langle f,(P_{\lambda_{2}}-P_{\lambda_{1}})f\big\rangle_{L^{2}(\RE^{3})} =%
\frac{1}{2\pi i}\int_{\lambda_{1}}^{\lambda_{2}}\big\langle f,(\widetilde{\mathcal{R}}_{\lambda }^{+} -\widetilde{\mathcal{R}}_{\lambda }^{-}
)f\big\rangle_{\alpha,-\alpha}\, d\lambda \,,
\end{equation*}
where $\langle\cdot,\cdot\rangle_{\alpha,-\alpha}$ denotes the $L^{2}_{\alpha}
(\mathbb{R}^{3})$-$L^{2}_{-\alpha}
(\mathbb{R}^{3})$ duality.
Hence, $\lambda\mapsto\langle f,P_{\lambda} f\rangle_{L^{2}(\RE^{3})} $ is continuously differentiable for any $f\in L^{2}_{\alpha}(  \mathbb{R}^{3})
$. By the same argument as in \cite[Proof of Theorem 6.1]{Agm}, the property
extends to the whole $(    (  L^{2}(  \mathbb{R}^{3}))  ^{pp})    ^{\bot}$.
\end{proof}
Summing up, we get the following
\begin{theorem} $$\sigma(  A_{v,\rho})    =\sigma_{ess}(  A_{v,\rho})=\sigma_{ac}(  A_{v,\rho}) =(-\infty,0]$$ and   the resolvent formula in Theorem \ref{RD} holds on the full resolvent set of $A_{v,\rho}$. Furthermore, $\CO_{+}\ni \kappa\mapsto (-A_{v,\rho}-\kappa^{2})^{-1}$ is an analytic ${\B }(L^{2}(\RE^{3}))$-valued map. 
\end{theorem}
\begin{proof} By Lemma \ref{sp} and Corollary \ref{Prop_empty_sc}, 
$$
{\sigma}(    A_{v,\rho})=\sigma_{p}(    A_{v,\rho})\cup \sigma_{c}(    A_{v,\rho})=\sigma_{c}(   A_{v,\rho} )=\sigma_{ac}(    A_{v,\rho})\,.
$$ 
By Lemma \ref{sp} and by $\sigma_{{disc}}(   A_{v,\rho}  )\subseteq \sigma_{p}(    A_{v,\rho} )$, 
$$
{\sigma}(     A_{v,\rho})=\sigma_{{disc}}(    A_{v,\rho})\cup \sigma_{ess}(     A_{v,\rho})=(-\infty,0]\,.
$$  
Since $ A_{v,\rho}$ is a closed operator, the map $\CO\backslash(-\infty,0]\ni z\mapsto (- A_{v,\rho}+z)^{-1}$ is analytic. The proof is then concluded by $\kappa\in \CO_{+}\Rightarrow-\kappa^{2}\in \CO\backslash(-\infty,0]$.
\end{proof}
\section{Resolvent convergence} 
By the resolvents difference formula in Theorem \ref{RD},  one has
\begin{align*}
&(-A_{v',\rho'}-\kappa^{2})^{-1}-(-A_{v,\rho}-\kappa^{2})^{-1}=
\begin{bmatrix}R_{\kappa}1_{\Omega}^{*}& \SL_{\kappa}\end{bmatrix}\times\\
&\times\!\!
\left(\begin{bmatrix} v'^{2}+({ v'^{2}}-1)\kappa^{2}N_{\kappa}&({ v'^{2}}-1)\kappa^{2}1_{\Omega}\SL_{\kappa}\\-\widetilde{\rho'}\,\gamma^{\-}_{1}N_{\kappa}&1-\widetilde{\rho'}\,\gamma_{1}\SL_{\kappa}
\end{bmatrix}^{\!-1}\!\!\!\!\!-\!\begin{bmatrix}v^{2}+({v^{2}}-1)\kappa^{2}N_{\kappa}&({v^{2}}-1)\kappa^{2}1_{\Omega}\SL_{\kappa}\\-{\widetilde\rho}\,\gamma^{\-}_{1}N_{\kappa}&1-{\widetilde\rho}\,\gamma_{1}\SL_{\kappa}
\end{bmatrix}^{\!-1}
\right)\!\!\times\\
&\times
\begin{bmatrix}(v'^{2}-1)\Delta^{\max}_{\Omega }1_{\Omega}R_{\kappa}\\ {\widetilde{\rho'}}\,\gamma_{1}R_{\kappa}
\end{bmatrix}+
\begin{bmatrix}R_{\kappa}1_{\Omega}^{*}& \SL_{\kappa}\end{bmatrix}\begin{bmatrix}v^{2}+({v^{2}}-1)\kappa^{2}N_{\kappa}&({v^{2}}-1)\kappa^{2}1_{\Omega}\SL_{\kappa}\\-{\widetilde\rho}\,\gamma^{\-}_{1}N_{\kappa}&1-{\widetilde\rho}\,\gamma_{1}\SL_{\kappa}
\end{bmatrix}^{\!-1}
\!\!\times\\
&\times
\begin{bmatrix}( v'^{2}-v^{2})\Delta^{\max}_{\Omega }1_{\Omega}R_{\kappa}\\ (\widetilde{\rho'}-{\widetilde\rho}\,)\,\gamma_{1}R_{\kappa}
\end{bmatrix}\\
=&
\begin{bmatrix}R_{\kappa}1_{\Omega}^{*}& \SL_{\kappa}\end{bmatrix}
\begin{bmatrix} v'^{2}+({ v'^{2}}-1)\kappa^{2}N_{\kappa}&({ v'^{2}}-1)\kappa^{2}1_{\Omega}\SL_{\kappa}\\-\widetilde{\rho'}\,\gamma^{\-}_{1}N_{\kappa}&1-\widetilde{\rho'}\,\gamma_{1}\SL_{\kappa}
\end{bmatrix}^{\!-1}\times\\
&\times\begin{bmatrix} ( v^{2}-v'^{2})(1+\kappa^{2}N_{\kappa}) &( v^{2}-v'^{2})\kappa^{2}1_{\Omega}\SL_{\kappa}
\\(\widetilde{\rho'}-{\widetilde\rho}\,)\,\gamma^{\-}_{1}N_{\kappa}&(\widetilde{\rho'}-{\widetilde\rho})\gamma_{1}\SL_{\kappa}
\end{bmatrix}
\begin{bmatrix}v^{2}+({v^{2}}-1)\kappa^{2}N_{\kappa}&({v^{2}}-1)\kappa^{2}1_{\Omega}\SL_{\kappa}\\-{\widetilde\rho}\,\gamma^{\-}_{1}N_{\kappa}&1-{\widetilde\rho}\,\gamma_{1}\SL_{\kappa}
\end{bmatrix}^{\!-1}\times\\
&\times
\begin{bmatrix}(v'^{2}-1)\Delta^{\max}_{\Omega }1_{\Omega}R_{\kappa}\\ {\widetilde{\rho'}}\,\gamma_{1}R_{\kappa}
\end{bmatrix}+
\begin{bmatrix}R_{\kappa}1_{\Omega}^{*}& \SL_{\kappa}\end{bmatrix}\begin{bmatrix}v^{2}+({v^{2}}-1)\kappa^{2}N_{\kappa}&({v^{2}}-1)\kappa^{2}1_{\Omega}\SL_{\kappa}\\-{\widetilde\rho}\,\gamma^{\-}_{1}N_{\kappa}&1-{\widetilde\rho}\,\gamma_{1}\SL_{\kappa}
\end{bmatrix}^{\!-1}
\!\!\times\\
&\times
\begin{bmatrix}( v'^{2}-v^{2})\Delta^{\max}_{\Omega }1_{\Omega}R_{\kappa}\\ (\widetilde{\rho'}-{\widetilde\rho}\,)\,\gamma_{1}R_{\kappa}
\end{bmatrix}.
\end{align*}
By Remark \ref{not1}, in the case $v_{\ell}\not=1$, $v'_{\ell}\not=1$ and $\rho_{\ell}\not=1$, $\rho'_{\ell}\not=1$ for any index $\ell$, the above resolvents difference simplifies as 
\begin{align*}
&(-A_{v',\rho'}-\kappa^{2})^{-1}-(-A_{v,\rho}-\kappa^{2})^{-1}
=\begin{bmatrix}R_{\kappa}1_{\Omega}^{*}& \SL_{\kappa}
\end{bmatrix}
\begin{bmatrix} v'^{2}({ v'^{2}}-1)^{-1}+\kappa^{2}N_{\kappa}&\kappa^{2}1_{\Omega}\SL_{\kappa}\\-\gamma^{\-}_{1}N_{\kappa}&\widetilde{\rho'}^{-1}-\gamma_{1}\SL_{\kappa}
\end{bmatrix}^{\!-1}\times\\
&\times\!\begin{bmatrix}((1-{v^{2}})(1-{v'^{2}}))^{-1} ( v^{2}-v'^{2})&0
\\0&(\widetilde{\rho'}\widetilde{\rho\,})^{-1}(\widetilde{\rho'}-\widetilde{\rho}\,)
\end{bmatrix}
\begin{bmatrix}v^{2}({v^{2}}-1)^{-1}+\kappa^{2}N_{\kappa}&\kappa^{2}1_{\Omega}\SL_{\kappa}\\-\gamma^{\-}_{1}N_{\kappa}&\widetilde{\rho\,}^{-1}-\gamma_{1}\SL_{\kappa}
\end{bmatrix}^{\!-1}\!\!\!\times\\
&\times
\begin{bmatrix}\Delta^{\max}_{\Omega }1_{\Omega}R_{\kappa}\\ \gamma_{1}R_{\kappa}
\end{bmatrix}.
\end{align*}
By the relations above, one gets the continuity of $A_{v,\rho}$ with respect to the parameters: 
\begin{theorem}\label{conv-1} If, for any $1\le \ell\le n$, the sequence $\{v_{\ell}(j),\rho_{\ell}(j)\}_{j=1}^{+\infty} \subset (0,+\infty)\times(0,+\infty)$ converges to $(v_{\ell},\rho_{\ell})\in(0,+\infty)\times(0,+\infty)$ as $j\to+\infty$, then $\{A_{v(j),\rho(j)}\}_{j=1}^{+\infty}$ converges to $A_{v,\rho}$ in norm resolvent sense.
\end{theorem}
Now, we consider the extreme cases where the $v_{\ell}(j)$'s and/or the $\rho_{\ell}(j)$'s converge to 
$+\infty$.  Obviously, 
$$
v_{\ell}(j)\to+\infty\quad\Rightarrow\quad\frac{{v_{\ell}(j)}^{2}}{{v_{\ell}(j)}^{2}-1}\to 1\,,
\qquad
\rho_{\ell}(j)\to+\infty\quad\Rightarrow\quad\frac{\rho_{\ell}(j)+1}{\rho_{\ell}(j)-1}=1\,.
$$
For brevity, we consider only the cases in which all the $v_{\ell}$'s and/or all the $\rho_{\ell}$'s  diverge; the case in which the material parameters only diverge on  some connected components of $\Omega$ can be treated in a similar way and is left to the reader.  
Hence, we introduce the block operator matrices in $L^{2}(\Omega)\oplus H^{-1/2}(\Gamma)$
$$
Q^{\flat,\natural}_{\kappa}:=\begin{bmatrix}w_{\flat}+\kappa^{2}N_{\kappa}&\kappa^{2}1_{\Omega}\SL_{\kappa}\\-\gamma^{\-}_{1}N_{\kappa}&z_{\natural}-\gamma_{1}\SL_{\kappa}
\end{bmatrix},\qquad (\flat,\natural)\in\{v,+\infty\}\backslash\{1\}\times \{\rho,+\infty\}\backslash\{1\}
$$
where
$$
w_{v}:=v^{2}(v^{2}-1)^{-1}\,,\qquad w_{+\infty}:=1\,,\qquad
z_{\rho}:=\widetilde\rho^{\,-1}\qquad z_{+\infty}:=\frac12\,.
$$
Since 
$$
(Q^{\flat,\natural}_{0})^{-1}=\begin{bmatrix}w_{\flat}^{-1}&0\\
w_{\flat}^{-1}\left(z_{\natural}-\gamma_{1}\SL_{0}\right)^{-1}\!\!\gamma^{\-}_{1}N_{0}&
\left(z_{\natural}-\gamma_{1}\SL_{0}\right)^{-1}
\end{bmatrix},
$$
by the continuity of $\kappa\mapsto Q^{\flat,\natural}_{\kappa}$, the inverse $(Q^{\flat,\natural}_{\kappa})^{-1}$ exists and is a bounded operator for  any $\kappa$ in a sufficiently small complex neighborhood of the origin $U_{0}$. Hence, 
\be\label{ps}
R^{\flat,\natural}_{\kappa}:=R_{\kappa}+\begin{bmatrix}R_{\kappa}1_{\Omega}^{*}& \SL_{\kappa}
\end{bmatrix}(Q_{k}^{\flat,\natural})^{-1}\begin{bmatrix}\Delta^{\max}_{\Omega }1_{\Omega}R_{\kappa}\\ \gamma_{1}R_{\kappa}
\end{bmatrix}
\ee
is a well defined bounded operator  for any $\kappa\in U_{0}\cap\CO_{+}$. 
Since   
$\|(-A_{v(j),\rho(j)}-\kappa^{2})^{-1}-R^{\flat,\natural}_{\kappa}\|\to 0$ as $(0,+\infty)\times (0,+\infty)\ni(v_{\ell}(j),\rho_{\ell}(j))\to (\flat,\natural)$, the map $-\kappa^{2}\mapsto R^{\flat,\natural}_{\kappa}$ is a pseudo-resolvent. It is the resolvent of a closed operator $A_{\flat,\natural}$ whenever $R^{\flat,\natural}_{\kappa}$ is injective for some $\kappa\in U_{0}\cap\CO_{+}$.
\begin{lemma}\label{inj} If $\flat\not=+\infty$, then $R^{\flat,\natural}_{\kappa}$ is injective for any $\kappa\in\CO_{+}$ such that $Q^{\flat,\natural}_{\kappa}$ has a bounded inverse.
\end{lemma}
\begin{proof} 
Let  $u_{0}\in L^{2}(\RE^{3})$ be such that $R^{\flat,\natural}_{\kappa}u_{0}=0$, i.e.,
$$
R_{k}u_{0}=-\begin{bmatrix}R_{\kappa}1_{\Omega}^{*}& \SL_{\kappa}
\end{bmatrix}(Q_{k}^{\flat,\natural})^{-1}\begin{bmatrix}\Delta^{\max}_{\Omega }1_{\Omega}R_{\kappa}u_{0}\\ \gamma_{1}R_{\kappa}u_{0}
\end{bmatrix}.
$$
By
\begin{align*}
\begin{bmatrix}
\Delta_{\Omega}^{\max}1_{\Omega}\\
\gamma_{1}
\end{bmatrix}
\begin{bmatrix}R_{\kappa}1_{\Omega}^{*}& \SL_{\kappa}
\end{bmatrix}=&
\begin{bmatrix}
\Delta_{\Omega}^{\max}N_{k}&\Delta_{\Omega}^{\max}1_{\Omega}\SL_{\kappa}
\\
\gamma_{1}N_{k}&\gamma_{1}\SL_{\kappa}
\end{bmatrix}
=
\begin{bmatrix}
-1-\kappa^{2}N_{k}&-\kappa^{2}1_{\Omega}\SL_{\kappa}
\\
\gamma_{1}N_{k}&\gamma_{1}\SL_{\kappa}
\end{bmatrix}\\
=&\begin{bmatrix}
w_{\flat}-1&0
\\
0&z_{\natural}
\end{bmatrix}-Q_{k}^{\flat,\natural}\,,
\end{align*}
there follows
$$
\begin{bmatrix}
\Delta_{\Omega}^{\max}1_{\Omega}R_{\kappa}u_{0}\\
\gamma_{1}R_{\kappa}u_{0}
\end{bmatrix}=\left(1-\begin{bmatrix}w_{\flat}-1&0\\0&z_{\natural}\end{bmatrix}(Q_{k}^{\flat,\natural})^{-1}\right)\begin{bmatrix}
\Delta_{\Omega}^{\max}1_{\Omega}R_{\kappa}u_{0}\\
\gamma_{1}R_{\kappa}u_{0}
\end{bmatrix},
$$
equivalently,
\be\label{traces}
\begin{bmatrix}w_{\flat}-1&0\\0&z_{\natural}\end{bmatrix}(Q_{k}^{\flat,\natural})^{-1}\begin{bmatrix}
\Delta_{\Omega}^{\max}1_{\Omega}R_{\kappa}u_{0}\\
\gamma_{1}R_{\kappa}u_{0}
\end{bmatrix}=0\,.
\ee
If $w_{\flat}\not=1$, then, by \eqref{traces}, $\begin{bmatrix}
\Delta_{\Omega}^{\max}1_{\Omega}R_{\kappa}u\\
\gamma_{1}R_{\kappa}u
\end{bmatrix}=0$; this entails $R_{\kappa}u=0$ and so $u=0$. 
\end{proof}
Consequently, one has the following
\begin{theorem}\label{teo-conv} If, for any $1\le \ell\le n$, the sequence $\{v_{\ell}(j),\rho_{\ell}(j)\}_{j=1}^{+\infty} \subset (0,+\infty)\times(0,+\infty)$ converges to $(v_{\ell},+\infty)$, $v_{\ell}\not=1$, as $j\to+\infty$, then $\{A_{v(j),\rho(j)}\}_{j=1}^{+\infty}$ converges in norm resolvent sense to the closed operator $A_{v,+\infty}$ with resolvent $R^{v,+\infty}_{\kappa}$ given by \eqref{ps}.
\end{theorem}
\begin{remark} The analogous of Theorem \ref{teo-conv} for the case  $v_{1}=\dots=v_{n}=1$ follows by using the resolvent in Example \ref{surf-pert}. One has that  $A_{1,\rho(j)}$ converges in norm resolvent sense to $A_{1,+\infty}$, with
$$
(A_{1,+\infty}-\kappa^{2})^{-1}=
R_{\kappa}
+ \SL_{\kappa}
\left(\frac12-\gamma_{1}\SL_{\kappa}\right)^{\!\!-1}\!\!
\gamma_{1}R_{\kappa}\,.
$$
The latter is injective by the same kind of reasonings as in the proof of Lemma \ref{inj}.
\end{remark}
\begin{remark} Looking at the resolvent in Example \ref{vol-pert}, one gets that $v_{\ell}(j)\to+\infty$, implies the norm convergence of $(-A_{v(j),1}-\kappa^{2})^{-1}$ to the pseudo-resolvent 
$$
R_{\kappa}+R_{\kappa}1_{\Omega}^{*}(1+\kappa^{2}N_{\kappa})^{-1}\Delta^{\max}_{\Omega }1_{\Omega}R_{\kappa}\,.
$$
However, the latter is not injective and hence is not the resolvent of a closed operator. Indeed, one can easily check that any function of the kind $1_{\Omega}^{*}u_{\Omega}$, $u_{\Omega}\in L^{2}(\Omega)$, 
belongs to its kernel.
\end{remark}
\begin{remark}
The closed operator $A_{v,+\infty}$ is relevant in acoustic modeling only if it is the generator of a operator cosine function, so that the Cauchy problem for the wave equation $\frac{\partial^{2} u}{\partial{t}^{2}}=A_{v,+\infty}u$ is solvable. In the not self-adjoint case (as for $A_{v,+\infty}$), this can be a not trivial problem, even if one has at disposal the resolvent (see, e.g., \cite[Theorem 2.1, Chapter II]{Fatt}). 
\end{remark}
\section{Resonances}
From now on, we suppose that the boundary $\Gamma$ is of class $\C^{1,\alpha}$, i.e., it can be represented locally as the graph of a 
function having H\"older-continuous, with exponent $\alpha$, first-order derivatives (see, e.g., \cite[page 90]{McLe}).
\begin{lemma}
\label{Fredh}The map 
$$\CO\ni\kappa\mapsto 
Q_{\kappa}:=\begin{bmatrix}v^{2}+({v^{2}}-1)\kappa^{2}N_{\kappa}&({v^{2}}-1)\kappa^{2}1_{\Omega}\SL_{\kappa}\\-\widetilde\rho
\,\gamma^{\-}_{1}N_{\kappa}&1-\widetilde\rho\,\gamma_{1}\SL_{\kappa}
\end{bmatrix}
$$
is a $\mathscr{B}(  L^{2}(  \Omega)\oplus H^{-1/2}(\Gamma))$-valued analytic map and $\CO\ni\kappa\mapsto Q_{\kappa}^{-1}$ is a $\mathscr{B}(L^{2}(\Omega)\oplus H^{-1/2}(\Gamma))$-valued meromorphic map with poles of finite rank. Such poles are the points $\kappa_{\circ}\in\CO\backslash\CO_{+}$ such that $\ker(Q_{\kappa_{\circ}})\not=\{0\}$.
\end{lemma}
\begin{proof} Rewriting $Q_{\kappa}$ as
 $$
Q_{\kappa}=
\begin{bmatrix}v^{2}&0\\
0&1\end{bmatrix}\left(\,\begin{bmatrix}1&0\\
0&1\end{bmatrix}+\begin{bmatrix}v^{-2}({v^{2}}-1)\kappa^{2}&0\\
0&-\widetilde\rho\,\end{bmatrix}\begin{bmatrix}N_{\kappa}&1_{\Omega}\SL_{\kappa}\\
\gamma^{\-}_{1}N_{\kappa}&\gamma_{1}\SL_{\kappa}
\end{bmatrix}\,\right),
$$
by meromorphic Fredholm theory (see, e.g., \cite[Theorem XIII.13]{ReSi IV}) it is enough to show that 
\be\label{red-map}
\CO\ni\kappa\mapsto \begin{bmatrix}N_{\kappa}&1_{\Omega}\SL_{\kappa}\\\gamma^{\-}_{1}N_{\kappa}&\gamma_{1}\SL_{\kappa}
\end{bmatrix}
\ee
is a $\mathfrak{S}_{\infty}( L^{2}(\Omega)\oplus H^{-1/2}(\Gamma))$-valued analytic map.\par
By, e.g., \cite[Lemma 3.1]{ZAMP}, $N_{\kappa}$ is Hilbert-Schmidt and hence compact. One gets $1_{\Omega}\SL_{\kappa}\in \mathfrak{S}_{\infty}(H^{-1/2}(\Gamma), L^{2}(\Omega))$
by $1_{\Omega}\SL_{\kappa}\in \B(H^{-1/2}(\Gamma), H^{1}(\Omega))$ and by the compact embedding $H^{1}(\Omega) \hookrightarrow L^{2}(\Omega)$. 
It is known that the Neumann-Poincar\'e operator $K_{0}$ belongs to $\mathfrak{S}_{\infty}(H^{1/2}(\Gamma))$ (here we need the $\mathcal{C}^{1,\alpha}$ hypothesis on $\Gamma$, see, e.g., \cite{Kang}), $\gamma_{1}\SL_{0}=K_{0}^{*}\in\mathfrak{S}_{\infty}(H^{-1/2}(\Gamma))$. The same argument (see, e.g., \cite[equation (15)]{Kang}) gives 
$\gamma_{1}\SL_{\kappa}\in\mathfrak{S}_{\infty}(H^{-1/2}(\Gamma))$. 
\par
Finally, $\gamma_{1}^{\-}N_{\kappa}\in \mathfrak{S}_{\infty}(L^{2}(\Omega),H^{-1/2}(\Gamma))$ by $\gamma_{1}^{\-}N_{\kappa}\in \B(L^{2}(\Omega),H^{1/2}(\Gamma))$ and by the compact embedding $H^{1/2}(\Gamma) \hookrightarrow H^{-1/2}(\Gamma)$. 
\end{proof}
The analogy with the case of Schr\"{o}dinger operators
suggests to define the resonances of $A_{v,\rho}$ as the poles of the
meromorphic extension of the map $z\mapsto 1_{B_{2}}(-A_{v,\rho}-z)^{-1}
1^{*}_{B_{1}}$ in the
non-physical sheet where $\text{Im}(\sqrt{z})<0$; here, $B_{1}$ and $B_{2}$ are any two  open bounded sets. As is well known, the map $\CO\ni\kappa\mapsto 1_{B_{2}}R_{\kappa}1^{*}_{B_{1}}$ 
is analytic. This entails the analyticity of the maps 
$\CO\ni\kappa\mapsto 1_{B_{2}}\SL_{\kappa}$, $\CO\ni\kappa\mapsto \Delta_{\Omega}^{\mathrm{max}}1^{*}_{\Omega}R_{\kappa}1_{B_{1}}$  and $\CO\ni\kappa\mapsto \gamma_{1}R_{\kappa}1_{B_{1}}$.  Therefore, taking into account the resolvent formula in Theorem \ref{RD} and Lemma \ref{Fredh}, the analytic $\B(L^{2}(B_{1}),L^{2}(B_{2}))$-valued function $\CO_{+}\ni\kappa \mapsto 1_{B_{2}}(-A_{v,\rho}-\kappa^{2})1^{*}_{B_{1}}$ has a 
meromorphic $\B(L^{2}(B_{1}),L^{2}(B_{2}))$-valued extension
\be\label{mero}
\CO\ni\kappa \mapsto 1_{B_{1}}R_{\kappa}1^{*}_{B_{2}}+1_{B_{1}}
\begin{bmatrix}
R_{\kappa}1_{\Omega}^{\ast} & S\!L_{\kappa}%
\end{bmatrix}Q_{\kappa}^{-1}%
\begin{bmatrix}
({v^{2}}-1)\Delta_{\Omega}^{\mathrm{max}}1_{\Omega}R_{\kappa}\\
\widetilde{\rho}\,\gamma_{1}R_{\kappa}%
\end{bmatrix}1^{*}_{B_{2}}\,.
\ee
Since the operator  (for brevity we denote it by $L$) on the right of $Q_{\kappa}^{-1}$ in \eqref{mero} has range which is dense in $L^{2}(\Omega)\oplus H^{-1/2}(\Gamma)$ and the one on the left is injective, the poles of the map in \eqref{mero} coincides with the points $\kappa_{\circ}$ such that $\ker(Q_{\kappa_{\circ}})\not=\{0\}$. Otherwise, denoting by  $\ran(L)_{\|}$ the orthogonal projection onto the closed subspace $\ker(Q_{\kappa_{\circ}})$, we would have $\ran(L)_{\|}\cap \ker(Q_{\kappa_{\circ}})=\{0\}$ which would imply $\ran(L)\subseteq \ker(Q_{\kappa_{\circ}})^{\perp}$, in contradiction with the density of $\ran(L)$. \par
Therefore, we define the set of resonances of $A_{v,\rho}$ as  the discrete set 
\begin{equation}
\mathsf{e}(A_{v,\rho}):=\{\kappa_{\circ}^{2}\in\CO:\kappa_{\circ
}\in{\mathbb{C}}_{-}\ \text{and}\ \ker
(Q_{\kappa_{\circ}})\not =\{0\},\ \}\,. \label{e_id}%
\end{equation}
This definition of resonance set is in agreement with the one given in \cite[Appendix B]{AGHKH}. 
To avoid any misunderstanding, let us remark that in some literature the term ''resonances'' refers to $\kappa_{\circ}$ instead of $\kappa_{\circ}^{2}$, see, e.g., \cite{DZ}.\par
The definition \eqref{e_id} can be somewhat simplified because the condition $\kappa\in\CO_{-}$ is automatically satisfied by the points $\kappa_{\circ}$ for which $Q_{\kappa_{\circ}}$ is not injective:
\begin{lemma}\label{res}
$$
\mathsf{e}(A_{v,\rho})=\big\{\kappa_{\circ}^{2}\in\CO:\ker
(Q_{\kappa_{\circ}})\not =\{0\}\big\}\,.
$$
\end{lemma}
\begin{proof} 
By Lemma \ref{pseudo-LAP}, $\CO_{+}\ni\kappa\mapsto 1_{B_{2}}(-A_{v,\rho}-\kappa^{2})^{-1}1^{*}_{B_{1}}$  has a continuous $\B(L^{2}(B_{1}),L^{2}(B_{2}))$-valued  extension to $\CO_{+}\cup\RE\backslash\{0\}$. Since the meromorphic $\B(L^{2}(B_{1}),L^{2}(B_{2}))$-valued function \eqref{mero}
coincides with $\mapsto 1_{B_{2}}(-A_{v,\rho}-\kappa^{2})^{-1}1^{*}_{B_{1}}$ when restricted to $\CO_{+}$, it is continuous on $\CO_{+}\cup\RE\backslash\{0\}$. Hence, the only possible real pole of $\kappa\mapsto Q_{\kappa}^{-1}$ is $\kappa=0$. However, $Q_{0}$ has a bounded inverse (see the proof of Lemma \ref{ZB}) and the proof is done. 
\end{proof}
\section{The $\ve$-dependent case} 
Given the $(0,+\infty)$-valued functions $\ve\mapsto v_{\ell}^{2}(\ve)=v^{2}(\ve)$, $\ve\mapsto \rho_{\ell}(\ve)=\rho(\ve)$, we consider the family of acoustic operators 
$$
A(\varepsilon):={v^{2}(\ve)\rho(\ve)}\nabla\!\cdot\!{\rho(\ve)}^{-1}\nabla\,,\qquad 0<\epsilon\le 1\,,
$$ 
with domain
$$
{\dom}(  A(\ve))  =\big\{  u\in H^{1}(\RE^{3})\cap
L_{\Delta}^{2}(  \mathbb{R}^{3}\backslash\Gamma(\ve))  \,:\, \rho(\ve)\,\gamma_{1}^{\mathrm{ex}}(\ve)u=\gamma_{1}^{\mathrm{in}}(\ve)u\big\}\,.
$$
Here $\gamma^{\-/\+}_{1}(\ve)$ denote the inner/exterior Neumann traces at $\Gamma(\ve):=\partial\Omega(\ve)$, where
$$
\Omega(\ve)=\Omega_{1}(\ve)\cup\dots\cup\Omega_{n}(\ve)\,,\qquad \Omega_{\ell}(\ve)=y_{\ell}+\Omega_{\circ}(\ve)\,,\qquad y_{1},\dots, y_{n}\in\RE^{3}\,,
$$
$$
\Omega_{\circ}(\ve):=\{\ve x,\  x\in\Omega_{\circ}\}
$$
and the parameter $\ve$ is sufficiently small so that  $\Omega_{i}(\ve)\cap\Omega_{j}(\ve)=\varnothing$ whenever $i\not=j$. \par
By Example \ref{cup} and Lemma \ref{res}, the resonance set of $A(\ve)$ is  
$$
\mathsf e(A(\ve))=\big\{\kappa^{2}(\ve)\in\CO:\ker(\widetilde Q_{\kappa(\ve)}(\ve))\not=\{0\}\,\big\}\,, 
$$ 
where
$$
\widetilde Q_{\kappa}(\ve):\big(\!\oplus_{\ell=1}^{n}L^{2}(\Omega_{\circ}(\ve)\big)
\!\oplus \!\big(\!\oplus_{\ell=1}^{n}H^{-1/2}(\Gamma_{\circ}(\ve)\big)
\to
\big(\!\oplus_{\ell=1}^{n}L^{2}(\Omega_{\circ}(\ve)\big)
\!\oplus \!\big(\!\oplus_{\ell=1}^{n}H^{-1/2}(\Gamma_{\circ}(\ve)\big),
$$
\begin{align*}
&
\widetilde Q_{\kappa}(\ve)
\\
=
&
\begin{bmatrix}\text{mat}(v^{2}(\ve)+(v^{2}(\ve)-1)\kappa^{2}1_{\Omega_{\circ}(\ve)}\Psi_{\ell m}R_{\kappa}1_{\Omega_{\circ}(\ve)}^{*})&({v^{2}(\ve)}-1)\kappa^{2}
\text{mat}(1_{\Omega_{\circ}(\ve)}\Psi_{\ell m}\SL^{\circ}_{\kappa}(\ve))
\\
\frac{1-\rho(\ve)}{1+\rho(\ve)}\,\text{mat}(\gamma^{\circ\-}_{1}(\ve)1_{\Omega_{\circ}(\ve)}\Psi_{\ell m}R_{\kappa}1_{\Omega_{\circ}(\ve)}^{*})&\text{mat}(\frac12+\frac{1-\rho(\ve)}{1+\rho(\ve)}\,\gamma^{\circ}_{1}(\ve)\Psi_{\ell m}\SL^{\circ}_{\kappa}(\ve))
\end{bmatrix}
\end{align*}
with $\SL^{\circ}_{\kappa}(\ve)$ and $\gamma^{\circ}_{1}(\ve)$ denoting  the operators corresponding to the domain $\Omega_{\circ}(\ve)$. 
\par
Given the map $\varphi_{\ve}(x):=\ve x$ such that $\varphi_{\ve}({\Omega_{\circ}})= {\Omega_{\circ}(\ve)}$, we define
$$
\Psi_{\ve}:L^{2}(\RE^{3})\to L^{2}(\RE^{3})\,,\qquad 
\Psi_{\ve}u:=u\circ\varphi_{\ve}
$$
$$
\Phi_{\ve}:L^{2}(\Omega_{\circ}(\ve))\to L^{2}(\Omega_{\circ})\,,\qquad 
\Phi_{\ell}u:=u\circ(\varphi_{\ve}|\Omega_{\circ})
$$ 
and 
$$
\Phi^{\partial}_{\ve}:L^{2}(\Gamma_{\circ}(\ve))\to L^{2}(\Gamma_{\circ})\,,\qquad 
\Phi^{\partial}_{\ve}\phi:=\phi\circ(\varphi_{\ve}|\Gamma_{\circ})\,.
$$
One has the identities
$$
1_{\Omega_{\circ}(\ve)}=\Phi_{\ve}^{-1}1_{\Omega_{\circ}}\Psi_{\ve}\,,
\qquad
1^{*}_{\Omega_{\circ}(\ve)}=\Psi_{\ve}^{-1}1^{*}_{\Omega_{\circ}}\Phi_{\ve}\,,
$$
$$
\Delta^{\max}_{\Omega_{\circ}(\ve)}=\ve^{-2}\Phi_{\ve}^{-1}\Delta^{\max}_{\Omega_{\circ}}\Phi_{\ve}\,,\qquad
R_{\kappa}=\ve^{2}\Psi_{\ve}^{-1}R_{\ve\kappa}\Psi_{\ve}\,,
$$
$$
\SL^{\circ}_{\kappa}(\ve)=\ve \Psi_{\ve}^{-1}\SL^{\circ}_{\ve\kappa}\Phi^{\partial}_{\ve}\,,
\qquad
\gamma_{1}^{\circ}(\ve)=\ve^{-1}(\Phi^{\partial}_{\ve})^{-1}\gamma^{\circ}_{1}\Psi_{\ve}\,.
$$
These relations give, by $\Psi_{\ve}\Psi_{\ell m}\Psi_{\ve}^{-1}=\Psi_{\ell m}$,
\begin{align*}
&\widetilde Q_{\kappa}(\ve)=\begin{bmatrix}\Phi_{\ve}&0\\
0&\Phi^{\partial}_{\ve}
\end{bmatrix}^{-1}
\times\\&\times
\begin{bmatrix}\text{mat}(v^{2}(\ve)+(v^{2}(\ve)-1)\ve^{2}\kappa^{2}1_{\Omega_{\circ}}\Psi_{\ell m}R_{\ve\kappa}1_{\Omega_{\circ}}^{*})&({v^{2}(\ve)}-1)\ve\kappa^{2}
\text{mat}(1_{\Omega_{\circ}}\Psi_{\ell m}\SL^{\circ}_{\ve\kappa})
\\
\frac{1-\rho(\ve)}{1+\rho(\ve)}\,\ve\,\text{mat}(\gamma^{\circ\-}_{1}1_{\Omega_{\circ}}\Psi_{\ell m}R_{\ve\kappa}1_{\Omega_{\circ}}^{*})&\text{mat}(\frac12+\frac{1-\rho(\ve)}{1+\rho(\ve)}\,\gamma^{\circ}_{1}\Psi_{\ell m}\SL^{\circ}_{\ve\kappa})
\end{bmatrix}
\times\\&\times
\begin{bmatrix}\Phi_{\ve}&0\\
0&\Phi^{\partial}_{\ve}
\end{bmatrix}.
\end{align*}
Noticing that for a block operator matrix $\begin{bmatrix}A&B\\C&D\end{bmatrix}$ there holds  
$$
\ker
\begin{bmatrix}A&B\\C&D\end{bmatrix}
\not=\{0\}\quad\Leftrightarrow
\quad\ker
\begin{bmatrix}\ve^{-a}A&\ve^{-b}B\\ \ve^{-c}\,C&\ve^{-d} D\end{bmatrix}
\not=\{0\}
$$
for any $\ve>0$ and $a,b,c,d$ such that $a+d=b+c$, 
we get, by Lemma \ref{res},
\begin{align*}
\mathsf e(A(\ve))=&\{\kappa^{2}(\ve)\in\CO:\ker\big(M^{(a,b,c)}_{\kappa(\ve)}(\ve)\big)\not=\{0\}\,\}\,,
\end{align*}
where
$$
M^{(a,b,c)}_{\kappa}(\ve): \big(\oplus_{\ell=1}^{n}L^{2}(\Omega_{\circ})\big)
\oplus \big(\oplus_{\ell=1}^{n} H^{-1/2}(\Gamma_{\circ})\big)\to 
\big(\oplus_{\ell=1}^{n}L^{2}(\Omega_{\circ})\big)
\oplus \big(\oplus_{\ell=1}^{n} H^{-1/2}(\Gamma_{\circ})\big)\,,
$$
\begin{align*}
&M_{\kappa}^{(a,b,c)}(\ve)\\
:=&
\begin{bmatrix}\text{mat}(v^{2}(\ve)\ve^{-a}+(v^{2}(\ve)-1)\ve^{-a+2}\,\kappa^{2}1_{\Omega_{\circ}}\Psi_{\ell m}R_{\ve\kappa}1_{\Omega_{\circ}}^{*})&({v^{2}(\ve)}-1)\ve^{-b+1}\,\kappa^{2}
\text{mat}(1_{\Omega_{\circ}}\Psi_{\ell m}\SL^{\circ}_{\ve\kappa})
\\
\frac{1-\rho(\ve)}{1+\rho(\ve)}\,\ve^{-c+1}\,\text{mat}(\gamma^{\circ\-}_{1}1_{\Omega_{\circ}}\Psi_{\ell m}R_{\ve\kappa}1_{\Omega_{\circ}}^{*})&\ve^{-(b+c-a)}\,\text{mat}(\frac12+\frac{1-\rho(\ve)}{1+\rho(\ve)}\,
\gamma^{\circ}_{1}\Psi_{\ell m}\SL^{\circ}_{\ve\kappa})
\end{bmatrix}.
\end{align*}
Therefore, if there exists a map 
$$
\ve\mapsto(\kappa(\ve), u(\ve)\oplus\phi(\ve))\in \CO\times \big(\oplus_{\ell=1}^{n}L^{2}(\Omega_{\circ})\big)\oplus\big(
\oplus_{\ell=1}^{n} H^{-1/2}(\Gamma_{\circ})\big)
$$ 
such that
\be\label{LF}
u(\ve)\not=0\,,\quad\phi(\ve)\not=0\,,\qquad M^{(a,b,c)}_{\kappa(\ve)}(\ve)
\begin{bmatrix}u(\ve)\\\phi(\ve)\end{bmatrix}=\begin{bmatrix} 0\\
 0\end{bmatrix},
\ee
then
$$
\kappa^{2}(\ve)\in\mathsf e(A(\ve))\,.
$$
\section{Resonances $\ve$-expansions}
In this Section, we look for the resonance of $A(\ve)$ in the case $n=1$; hence, we set $\Omega_{\circ}=\Omega$ and avoid the use of the apex ${\,}^{\circ}$ in denoting operators. We study the four cases, differing by the behavior for $\ve\ll 1$ of the analytic functions $\ve\mapsto v^{2}(\ve)\chi_{\Omega}$ and $\ve\mapsto \rho(\ve)\chi_{\Omega}$, indicated in \eqref{caso1}-\eqref{caso4} in the Introduction. We denote by $A_{j}(\ve)$, $j=1,2,3,4$, the corresponding acoustic operators. 
\par
Inspired by  \cite{Rose}, in order to solve \eqref{LF}, we use the following implicit function argument: let $$F:U\times\CO\times  \Xi
\to \Xi$$ be an analytic map, where $U\subset\CO$ is an open neighborhood of $0$ and $\Xi$ is a complex Hilbert space with scalar product $\langle\cdot,\cdot\rangle$ and corresponding norm $\|\cdot\|$. We look for an analytic map 
$$
f:U_{0}\to\CO\times  \left(\Xi\backslash\{0\}\right)\,,
$$
where $U_{0}\subset\CO$ is a small open neighborhood of $0$, which uniquely solve the equation 
$$
{F}(z,f(z))=0\,.
$$ Suppose there exists $(\zeta_{0}, \xi_{0})\in \CO\times \left(\Xi\backslash\{0\}\right)$ such that ${F}(0,\zeta_{0},\xi_{0})=0$. Then, given such a $(\zeta_{0}, \xi_{0})$, we define the map 
$$\Phi:U\times\CO\times  \Xi
\to \CO\times\Xi,\qquad\Phi(z,\zeta,\xi):=(\langle\xi_{0},\xi-\xi_{0}\rangle, {F}(z,\zeta,\xi))\,.
$$
Obviously, $\Phi(0,\zeta_{0}, \xi_{0})=(0,0)$. Let us denote by  $D_{0}F:\CO\times\Xi \to \Xi$ 
the Frechet  differential at $(0,\zeta_{0}, \xi_{0})$ of the map $(\zeta,\xi)\mapsto F(0,\zeta,\xi)$. By the implicit function theorem for Banach space-valued analytic functions (see, e.g., \cite{Whitt}), if the bounded linear map 
$$
D_{0}\Phi:\CO\times \Xi\to \CO\times \Xi\,,\qquad D_{0}\Phi(\zeta,\xi):=(\langle\xi_{0},\xi\rangle, D_{0}F(\zeta,\xi))
$$
is a bijection (and hence has a bounded inverse by the open mapping theorem), then there exists $U_{0}\subseteq U$, an open neighborhood of $0$,  and an unique analytic function 
$$
f: U_{0}\to \CO\times \Xi\,,\qquad f(z)\equiv(\zeta(z),\xi(z))\,,
$$ 
such that $\Phi(z,f(z))=(0,0)$. Therefore, by the definition of $\Phi$, 
$$
\begin{cases}
\langle\xi_{0},\xi(z)\rangle=\|\xi_{0}\|^{2}
\\
F(z,f(z))=0\,.
\end{cases}
$$ 
Hence,  by $\xi_{0}\in\Xi\backslash\{0\}$, one gets $\xi(z)\in\Xi\backslash\{0\}$. Furthermore, 
$$
f(z)=f(0)-(D_{0}\Phi)^{-1}(0, \partial_{0}F(\zeta_{0}))\, z+O(|z|^{2})\,,
$$
where $\partial_{0}F:\CO 
\to \Xi$ denotes the partial derivative with respect to the first variable $z$,  calculated at the point $(0,\zeta_{0}, \xi_{0})$. Let us remark that all the coefficients of the expansion of $f$ beyond the first one can be computed recursively, as in the finite dimensional case.
\subsection{Case 1. 
} 
Let us define the map
$$
{F}_{1}:U\times\CO\times  X
\to X
$$
\begin{align*}
&{F}_{1}(\ve,\kappa, u\oplus\phi):=M^{(2,1,1)}_{\kappa}(\ve)
\begin{bmatrix}u\\\phi\end{bmatrix}\\
=&\begin{bmatrix}v^{2}(\ve)\ve^{-2}+(v^{2}(\ve)-1)\kappa^{2}N_{\ve\kappa}&({v^{2}(\ve)}-1)\kappa^{2}1_{\Omega}\SL_{\ve\kappa})
\\
\frac{1-\rho(\ve)}{1+\rho(\ve)}\,\gamma^{\-}_{1}N_{\ve\kappa}&\frac12+\frac{1-\rho(\ve)}{1+\rho(\ve)}\,
\gamma_{1}\SL_{\ve\kappa}
\end{bmatrix}\begin{bmatrix}u\\\phi\end{bmatrix}
,
\end{align*}
where $X=L^{2}(\Omega)\oplus H^{-1/2}(\Gamma)$, 
$U\subset\CO$ is a small open neighborhood of the origin and $\ve\mapsto v^{2}(\ve)$ and $\ve\mapsto \rho(\ve)$ are analytic function as in \eqref{caso1}.\par We look for a solution $\ve\mapsto (\kappa(\ve),u(\ve)\oplus\phi(\ve))\in \CO_{-}\times (X\backslash\{0\})$ of the equation
\be\label{Fe1}
{F}_{1}(\ve,\kappa(\ve),u(\ve)\oplus\phi(\ve))=0\,.
\ee
By the analyticity of the operator-valued maps $z\mapsto N_{z}$ and $z\mapsto \SL_{z}$ (see, e.g.,  \cite{Minn}), writing
$$
N_{z}=N_{0}+N_{(1)}z+O(|z|^{2})\,,\qquad
\SL_{z}=\SL_{0}+\SL_{(1)}z+O(|z|^{2})\,,
$$
and by 
$$
\frac{1-\rho(\ve)}{1+\rho(\ve)}=-\frac{\rho_{1}}2\,\ve+O(\ve^{2})\,,
$$
one gets
\begin{align*}
&{F}_{1}(\ve,\kappa, u\oplus\phi)=\begin{bmatrix}
v^{2}-\kappa^{2}N_{0}&-\kappa^{2}\,
1_{\Omega}\SL_{0}
\\
0&\frac12
\end{bmatrix}
\begin{bmatrix}u\\ \phi\end{bmatrix}\\
&+
\begin{bmatrix}
v_{1}^{2}-\kappa^{3}N_{(1)}&-\kappa^{3}1_{\Omega}\SL_{(1)}\\
 -\frac12\,{\rho_{1}}\gamma^{\-}_{1}N_{0}&-\frac12\,{\rho_{1}}\gamma_{1}\SL_{0}
\end{bmatrix}
\begin{bmatrix}u\\ \phi\end{bmatrix}\ve+O(\ve^{2})
\end{align*}
Therefore, 
\begin{align*}
{F}_{1}(0,\kappa, u\oplus\phi)=0
\Leftrightarrow
\begin{cases}(v^{2}-\kappa^{2}N_{0})u-\kappa^{2}1_{\Omega}\SL_{0}\phi=0\\
\frac12\,\phi=0\end{cases}\!\!\!\!\!
\Leftrightarrow
\begin{cases}
(v^{2}-\kappa^{2}N_{0})u=0\\ 
\phi=0\,.
\end{cases}
\end{align*}
Hence,
$$
{F}_{1}(0,\pm{v}{\lambda ^{-1/2}}, e_{\lambda }\oplus 0)=0\,,
$$
whenever $e_{\lambda }$ is a normalized eigenfunction, with corresponding eigenvalue $\lambda>0 $, of the Newton operator $N_{0}$ corresponding to the connected domain $\Omega$.\par 
Now, we introduce the map 
\be\label{Phi1}
\Phi_{1}:U\times\CO\times  X
\to \CO\times X\,,\qquad \Phi_{1}(\ve,\kappa,u\oplus\phi):=(\langle e_{\lambda },u\rangle_{L^{2}(\Omega)}-1,{F}_{1}(\ve,\kappa,u\oplus\phi))\,.
\ee
Obviously, 
$$
\Phi_{1}(0,\pm{v}{\lambda ^{-1/2}}, e_{\lambda }\oplus 0)=(0,0)\,.
$$
Let us
denote by 
$$
D_{\pm}\Phi_{1}: \CO\times  X
\to \CO\times X
$$
the Frechet  differential at $(0,\pm{v}{\lambda ^{-1/2}}, e_{\lambda }\oplus 0)$ of the map $(\kappa,u\oplus\phi)\mapsto\Phi_{1}(0,\kappa,u\oplus\phi)$, i.e.,
$$
D_{\pm}\Phi_{1}(\kappa,u\oplus\phi)=\left(\langle e_{\lambda },u\rangle_{L^{2}(\Omega)},\big(\mp2v\lambda ^{1/2}\kappa e_{\lambda}+v^{2}(1-\lambda ^{-1}N_{0})u-v^{2}\lambda ^{-1}1_{\Omega}\SL_{0}\phi\big)\oplus
\frac12\,\phi\right)\,.
$$
Now, for any given $( z_{*},u_{*}\oplus\phi_{*})\in\CO\times X$,   we prove 
the existence of a unique solution $(\kappa, u\oplus\phi)\in\CO\times X$ to the equation 
\be\label{uni1}
D_{\pm}\Phi_{1}(\kappa,u\oplus\phi)=( z_{*},u_{*}\oplus\phi_{*})\,.
\ee 
Setting
$$
w:=u_{*}+2v^{2}\lambda ^{-1}1_{\Omega}\SL_{0}\phi_{*}\,, 
$$
equation \eqref{uni1} recasts as 
\be\label{uni1.1}
\begin{cases}
\langle e_{\lambda },u\rangle_{L^{2}(\Omega)}=z_{*}\\
\mp2v\lambda ^{1/2}\kappa e_{\lambda }+v^{2}\lambda ^{-1}(\lambda -N_{0})u =w\\
\frac12\,\phi=\phi_{*}\,.
\end{cases}
\ee
Introducing the orthogonal decompositions 
$$
u=\widehat u_{\lambda }\, e_{\lambda }+u_{\perp}\,,\quad w=\widehat w_{\lambda } e_{\lambda }+w_{\perp}\,,
$$
where
$$ 
\widehat u_{\lambda }:=\langle e_{\lambda },u\rangle_{L^{2}(\Omega)}\,,
\quad \widehat w_{\lambda }:=\langle e_{\lambda },w\rangle_{L^{2}(\Omega)}\,,\qquad\langle e_{\lambda },u_{\perp}\rangle_{L^{2}(\Omega)}=\langle e_{\lambda },w_{\perp}\rangle_{L^{2}(\Omega)}=0\,,
$$
\eqref{uni1} rewrites as 
$$
\begin{cases}
\widehat  u_{\lambda }= z_{*}\\
\mp2v\lambda ^{1/2}\kappa=\widehat  w_{\lambda }\\
v^{2}\lambda ^{-1}(\lambda -N_{0})u_{\perp}=w_{\perp}\\
\frac12\,\phi=\phi_{*}\,.
\end{cases}
$$
We define the linear operator
$$
(\lambda -N_{0})_{\perp}:\ker(\lambda -N_{0})^{\perp}\to \ker(\lambda -N_{0})^{\perp}\,,
\quad
(\lambda -N_{0})_{\perp}:=(\lambda -N_{0})|\ker(\lambda -N_{0})^{\perp}\,;
$$
since $N_{0}$ is compact and self-adjoint (hence, $\ker(\lambda -N_{0})^{\perp}=\ran(\lambda -N_{0})$), it is a bounded bijection.\par If $\lambda $ is a simple eigenvalue then $w_{\perp}\in \ker(\lambda -N_{0})^{\perp}$ and the unique solution of   \eqref{uni1} is given by 
\be\label{sol1}
\begin{cases}
\kappa=\mp(2v\lambda ^{1/2})^{-1}
\langle e_{\lambda }, u_{*}+2v^{2}\lambda ^{-1}1_{\Omega}\SL_{0}\phi_{*}\rangle_{L^{2}(\Omega)}\\
u= z_{*}+{\lambda }{v^{-2}}\,(\lambda -N_{0})_{\perp}^{-1}\big(u_{*}+2v^{2}\lambda ^{-1}1_{\Omega}\SL_{0}\phi_{*}\big)_{\perp}
\\
\phi=2\phi_{*}\,.
\end{cases}
\ee
Therefore, whenever $\lambda$ is a simple eigenvalue, $D_{\pm}\Phi_{1}$ is a bounded bijection and,  by the implicit function argument recalled at the beginning of the Section, in a sufficiently small neighborhood of $0\in\RE$, there is a $\CO\times (X\backslash\{0\})$-valued, analytic  function $\ve\mapsto (\kappa_{\pm}(\ve), u_{\pm} (\ve)\oplus\phi_{\pm} (\ve))$,
$$
\kappa_{\pm}(\ve)=\pm\frac{v}{\lambda ^{1/2}}+\sum_{n=1}^{+\infty}\kappa_{\pm}^{(n)}\ve^{n}
\,,\qquad
u_{\pm} (\ve)=e_{\lambda }+\sum_{n=1}^{+\infty} u_{\pm} ^{(n)}\ve^{n}\,,\qquad
\phi_{\pm} (\ve)=\sum_{n=1}^{+\infty}\phi_{\pm} ^{(n)}\ve^{n}
$$
which uniquely solves \eqref{Fe1}. One can explicitly compute all the coefficients appearing in the expansions. In particular,
 $$
\big(\kappa_{\pm}^{(1)},u_{\pm} ^{(1)}\oplus\phi_{\pm} ^{(1)}\big)=-\left(D_{\pm}\Phi_{1}\right)^{-1}
\left(
\left.\frac{d\ }{d\ve}\right|_{\ve=0}\Phi_{1}(\ve,\pm v\lambda ^{-1/2},e_{\lambda },0)
\right).
$$
By 
$$
N_{(1)}u:=\left.\frac{d\ }{d\kappa}\right|_{\kappa=0} N_{\kappa}u=\frac{i}{4\pi}\,\langle 1,u\rangle_{L^{2}(\Omega)}\,1\,,
$$ 
one obtains
\begin{align*}
&D_{\pm}\Phi_{1}(\kappa_{\pm}^{(1)}, u_{\pm} ^{(1)}\oplus\phi_{\pm} ^{(1)})\\
=&\left(0,
\big(-
{v_{1}^{2}e_{\lambda}\pm i\,(4\pi)^{-1}v^{3}}{\lambda ^{-3/2}} \langle 1,e_{\lambda }\rangle_{L^{2}(\Omega)} 1\big)\oplus\frac12\,\rho_{1} \lambda \gamma_{1}^{\-}e_{\lambda }\right)\,.
\end{align*}
Hence, by \eqref{sol1} and by
$$
\langle e_{\lambda }, 1_{\Omega}\SL_{0}\gamma_{1}^{\-}e_{\lambda }\rangle_{L^{2}(\Omega)}=
\langle \gamma_{0}^{\-}N_{0}e_{\lambda }, \gamma_{1}^{\-}e_{\lambda }\rangle_{L^{2}(\Gamma)}=
\lambda \langle \gamma_{0}^{\-}e_{\lambda }, \gamma_{1}^{\-}e_{\lambda }\rangle_{L^{2}(\Gamma)}\,,
$$
there follows
\begin{align*}
\kappa_{\pm}^{(1)}=\mp\,\frac{v_{1}^{2}}{2v\lambda ^{1/2}}
- i\, \frac{v^{2}}{8\pi\lambda ^{2}}\,\langle 1,e_{\lambda }\rangle_{L^{2}(\Omega)} ^{2}\mp \frac12\,\rho_{1} v{\lambda ^{1/2}}\,\langle\gamma_{0}^{\-}e_{\lambda },\gamma_{1}^{\-}e_{\lambda }\rangle_{L^{2}(\Gamma)}\,.
\end{align*}
This gives 
$$
\kappa_{\pm}(\ve)=\pm \left(\frac{v}{\lambda ^{1/2}}
+\!\left(\frac{v_{1}^{2}}{2v\lambda ^{1/2}}-
\frac12\,\rho_{1} v{\lambda ^{1/2}}\,\langle\gamma_{0}^{\-}e_{\lambda },\gamma_{1}^{\-}e_{\lambda }\rangle_{L^{2}(\Gamma)}\right)\!\ve\!\right)
- i\, \frac{v^{2}}{8\pi\lambda ^{2}}\,\langle 1,e_{\lambda }\rangle^{2}_{L^{2}(\Omega)} \ve+O(\ve^{2})
$$
and so $k^{2}_{\pm}(\ve)\in\mathsf e(A(\ve))$, where 
$$
k^{2}_{\pm}(\ve)=\frac{v^{2}}{\lambda }-\left(\frac{v_{1}^{2}}{\lambda }+
\rho_{1} v^{2}\langle\gamma_{0}^{\-}e_{\lambda },\gamma_{1}^{\-}e_{\lambda }\rangle_{L^{2}(\Gamma)}\right) \ve
\mp i\, \frac{v^{3}}{8\pi\lambda ^{5/2}}\,\langle 1,e_{\lambda }\rangle^{2}_{L^{2}(\Omega)} \ve
+O(\ve^{2})\,.
$$
Let us notice that, in the case where  $v=1$ and $\rho_{1} =v_{1}=0$, the above expansion coincides with the one provided in \cite{ZAMP}.
\begin{remark} Since the eigenfunction $e_{\lambda }$ is real-valued, the scalar products $\langle\gamma_{0}^{\-}e_{\lambda },\gamma_{1}^{\-}e_{\lambda }\rangle_{L^{2}(\Gamma)}$ and $\langle 1, e_{\lambda }\rangle_{L^{2}(\Omega)}$ are real-valued as well. 
\end{remark}
\begin{remark} Regarding our hypothesis on the simplicity of the eigenvalue $\lambda$, since $N_{0}$ is the inverse of a Laplacian in $L^{2}(\Omega)$ with some kind of (not local) boundary conditions (see \cite{KalSur}), we conjecture, by the results in \cite[Chapter 6]{Henry}, that such an assumption holds true for a generic domain $\Omega$.
\end{remark}
The study of Case 2 and of the successive ones requires some preliminaries. 
Setting ${K}_{z}:=\gamma_{0}\DL_{z}$, the map $z\mapsto {K}_{z}$ is  $\B(H^{1/2}(\Gamma_{\circ}))$-valued analytic and (see \cite[Lemma 5.3]{Minn})
$$
{K}_{z}={K}_{0}+{K}_{(2)}\,z^{2}+{K}_{(3)}\,z^{3}+O(|z|^{4})\,.
$$
Defining  the scalar product $[\cdot,\cdot]_{\frac12}$ in $H^{1/2}(\Gamma)$ by 
$$[\phi,\varphi]_{\frac12}:=\langle {S_{0}}^{-1}\phi,\varphi\rangle_{-\frac12,\frac12} \,,
$$
where
$$
S_{0}:H^{-1/2}(\Gamma_{\circ})\to H^{1/2}(\Gamma_{\circ})\,,\qquad
S_{0}:=\gamma_{0}\SL_{0}\,,
$$
we introduce the orthogonal (with respect $[\cdot,\cdot]_{\frac12}\,$) projector
$$
P_{0}\phi:=c_{\Omega}^{-1}\,[1,\phi]_{\frac12}\,1 \,,\qquad c_{\Omega}:=[1,1]_{\frac12}\,.
$$
It is known that  $-\frac12$ is a simple eigenvalues of $K_{0}$ with the constant function $1$ as the corresponding eigenfunction, so that  
$$
P_{0}\left(\frac12+{K}_{0}\right)P_{0}=0\,.
$$
Furthermore (see, \cite[Lemmata 2.3 and 2.4]{Minn}),
$$
P_{0}{K}_{(2)}P_{0}=-\frac{1}{\omega^{2}_{M}}\,P_{0}
\,,\qquad 
\omega_{M}^{2}:=\frac{c_{\Omega}}{|\Omega|} \,,
$$
and 
$$
P_{0}{K}_{(3)}P_{0}=-i\,\frac{|\Omega|}{4\pi}\,P_{0}\,.
$$
By the identity, 
\be\label{id}
{K}_{0}S_{0}=S_{0}{K}_{0}^{*}\equiv S_{0}\gamma_{1}\SL_{0}\,,
\ee
$-\frac12$ is a simple eigenvalue of $\gamma_{1}\SL_{0}$ with eigenfunction   ${S_{0}}^{-1}1$ and 
$$
P_{0}^{*}\left(\frac12+\gamma_{1}\SL_{0}\right)P_{0}^{*}=0\,. 
$$
Here ${\, }^{*}$ denotes the adjoint with respect to the $H^{-1/2}(\Gamma)$-$H^{1/2}(\Gamma)$ duality; therefore,
$$
P_{0}^{*}\phi=c_{\Omega}^{-1}\,\langle1,\phi\rangle_{\frac12,-\frac12} \,{S_{0}}^{-1}1\,.
$$
Furthermore,
\be\label{k2k3}
P_{0}^{*}K^{*}_{(2)}P_{0}^{*}=-\frac{1}{\omega^{2}_{M}}\,P_{0}^{*}
\,,\qquad 
P_{0}^{*}K^{*}_{(3)}P_{0}^{*}=i\,\frac{|\Omega|}{4\pi}\,P_{0}^{*}\,.
\ee
Defining the scalar product $[\cdot,\cdot]_{-\frac12}$ in $H^{-1/2}(\Gamma)$ by 
$$[\phi,\varphi]_{-\frac12}:=\langle S_{0}\phi,\varphi\rangle_{\frac12,-\frac12}\,,
$$ 
we introduce the orthogonal (with respect $[\cdot,\cdot]_{-\frac12}\,$) projector
$$
P\phi:=\widehat c_{\Omega}^{\ -1}\,[{S_{0}}^{-1}1,\phi]_{-\frac12}\,{S_{0}}^{-1}1 \,,\qquad \widehat c_{\Omega}:=[{S_{0}}^{-1}1,{S_{0}}^{-1}1]_{-\frac12}\,.
$$
One has 
$$
\widehat c_{\Omega}
=\langle 1,{S_{0}}^{-1}1\rangle_{\frac12,-\frac12}=
\langle {S_{0}}^{-1}1,1\rangle_{-\frac12,\frac12}=
c_{\Omega}
$$
and hence
$$
P\phi=c_{\Omega}^{\ -1}\,[{S_{0}}^{-1}1,\phi]_{-\frac12}\,{S_{0}}^{-1}1=
c_{\Omega}^{\ -1}\,
\langle 1,\phi\rangle_{\frac12,-\frac12}
\,{S_{0}}^{-1}1=P_{0}^{*}\phi\,.
$$ 
By \eqref{id},
\begin{align*}
[\gamma_{1}\SL_{0}\phi,\varphi]_{-\frac12}=&\langle S_{0}\gamma_{1}\SL_{0}\phi,\varphi\rangle_{\frac12,-\frac12}=
\langle \phi, K_{0}S_{0}\varphi\rangle_{-\frac12,\frac12}\\
=&
\langle S_{0}^{-1}S_{0}\phi,K_{0}S_{0}\varphi\rangle_{\frac12,-\frac12}=
[\phi,S_{0}^{-1}K_{0}S_{0}\varphi]_{-\frac12}\\
=&[\phi,\gamma_{1}\SL_{0}\varphi]_{-\frac12}
\end{align*}
and thus $\gamma_{1}\SL_{0}$ is self-adjoint in the Hilbert space given by $H^{-1/2}(\Gamma)$ equipped with the scalar product $[\cdot,\cdot]_{-\frac12}$; such a Hilbert space will be denoted by 
$\widehat H^{-1/2}(\Gamma)$. Since $\gamma_{1}\SL_{0}$ is self-adjoint and compact (hence, 
$\ran\left(\frac12+\gamma_{1}\SL_{0}\right)=\ker\left(\frac12+\gamma_{1}\SL_{0}\right)^{\perp}$),  
$$
\left(\frac12+\gamma_{1}\SL_{0}
\right)_{\!\!\perp}:\ker\!\left(\frac12+\gamma_{1}\SL_{0}\right)
^{\!\!\perp}\to\ker\!\left(\frac12+\gamma_{1}\SL_{0}\right)
^{\!\!\perp}
$$
is a bounded bijection, where we introduced the shorthand notation 
$$\left(\frac12+\gamma_{1}\SL_{0}
\right)_{\!\!\perp}:=P_{\perp}\left(\frac12+\gamma_{1}\SL_{0}
\right)P_{\perp}\,,\qquad P_{\perp}:=(1-P)\,.
$$ 
Finally, we introduce the Hilbert space 
$$\widehat X:=L^{2}(\Omega)\oplus \widehat H^{-1/2}(\Gamma)$$
with scalar product 
$$
\langle u_{1}\oplus\phi_{1}, u_{2}\oplus\phi_{2}\rangle_{\widehat X}:=\langle u_{1},u_{2}\rangle_{L^{2}(\Omega)}+[\phi_{1},\phi_{2}]_{-\frac12}\,.
$$
\subsection{Case 2. 
} 
Let us define the map
$$
{F}_{2}:U\times \CO\times  \widehat X
\to \CO\times  \widehat X
$$
\begin{align*}
&{F}_{2}(\ve,\kappa, u\oplus\phi):=M^{(0,1,1)}_{\kappa}(\ve)
\begin{bmatrix}u\\P\phi+\ve^{2}P_{\perp}\phi\end{bmatrix}\\
=&\begin{bmatrix}v^{2}(\ve)+(v^{2}(\ve)-1)\kappa^{2}N_{\ve\kappa}&
(v^{2}(\ve)-1)\kappa^{2}1_{\Omega}\SL_{\ve\kappa}
\\
\frac{1-\rho(\ve)}{1+\rho(\ve)}\,\gamma^{\-}_{1}N_{\ve\kappa}&\ve^{-2}
\left(\frac12+\frac{1-\rho(\ve)}{1+\rho(\ve)}\,
\gamma_{1}\SL_{\ve\kappa}\right)\end{bmatrix}
\begin{bmatrix}u\\P\phi+\ve^{2}P_{\perp}\phi\end{bmatrix},
\end{align*}
where $U\subset\CO$ is a small open neighborhood of the origin and $\ve\mapsto v^{2}(\ve)$ and $\ve\mapsto \rho(\ve)$ are analytic function as in \eqref{caso2}.\par
We look for a solution $\ve\mapsto (\kappa(\ve), u(\ve)\oplus\phi(\ve))\in\CO_{-}\times (\widehat X\backslash\{0\})$ of the equation
\be\label{Fe2} 
{F}_{2}(\ve,\kappa(\ve), u(\ve)\oplus\phi(\ve))=0\,.
\ee
By $\gamma_{1}\SL_{z}={K}_{-\bar z}^{*}$, one gets 
$$
\gamma_{1}\SL_{z}=\gamma_{1}\SL_{0}+K^{*}_{(2)}\,{z}^{2}-K^{*}_{(3)}\,{z}^{3}+O(|z|^{4})\,.
$$
Hence, by
$$
\frac{1-\rho(\ve)}{1+\rho(\ve)}=1-2\rho\ve^{2}+O(\ve^{3})\,,
$$
one obtains
\begin{align*}
&\ve^{-2}\left(\frac12+\frac{1-\rho(\ve)}{1+\rho(\ve)}\ \gamma_{1}\SL_{\ve\kappa}\right)
\\=&
\ve^{-2}\left(\frac12+\gamma_{1}\SL_{0}
\right)_{\!\!\perp}+
\big(\kappa^{2}K_{(2)}^{*}-2\rho\, \gamma_{1}\SL_{0}\big) -
\kappa^{3}K_{(3)}^{*}\ve+O(\ve^{2})\,.
\end{align*}
This gives (notice that $\gamma_{1}^{\-}N_{(1)}=0$)
\begin{align*}
{F}_{2}(\ve,\kappa, u\oplus\phi)=&
\begin{bmatrix}1&0\\\gamma_{1}^{\-}N_{0}&
\left(\frac12+\gamma_{1}\SL_{0}
\right)_{\!\!\perp}+
\big(\kappa^{2}K_{(2)}^{*}-2\rho\,\gamma_{1}\SL_{0}\big)P\end{bmatrix}
\begin{bmatrix}u\\ \phi\end{bmatrix}
\\+&
\begin{bmatrix}v^{2}&\kappa^{2}v_{1}^{2}1_{\Omega}\SL_{0}P\\
0&-\kappa^{3}K_{(3)}^{*}P\end{bmatrix}\begin{bmatrix}u\\ \phi\end{bmatrix}\ve+
O(\ve^{2})
\,.
\end{align*}
Therefore, ${F}_{2}(0,\kappa, u\oplus\phi)=0$ if and only if $(\kappa, u,\phi)$ solves 
$$
\begin{cases}
u=0\\
\left(\frac12+\gamma_{1}\SL_{0}
\right)_{\!\!\perp}\phi
+(\kappa^{2}K_{(2)}^{*}-2\rho\, \gamma_{1}\SL_{0})P\phi
=0\,.\end{cases}
$$
This is equivalent to  
$$
\begin{cases}
u=0\\
P(\kappa^{2}K_{(2)}^{*}-2\rho\, \gamma_{1}\SL_{0})P\phi
=0\\
\left(\frac12+\gamma_{1}\SL_{0}
\right)_{\!\!\perp}\phi
+P_{\perp}(\kappa^{2}K_{(2)}^{*}-2\rho\,\gamma_{1}\SL_{0})P\phi
=0\,,
\end{cases}
$$
i.e.,
$$
\begin{cases}
u=0\\
\left(\rho-\frac{\kappa_{0}^{2}}{\omega^{2}_{M}}\right)P\phi=0\\
\left(\frac12+\gamma_{1}\SL_{0}
\right)_{\!\!\perp}\phi
+\kappa^{2}P_{\perp}K_{(2)}^{*}P\phi
=0\,.\end{cases}
$$
Thus, ${F}_{2}(0,\pm\kappa_{\circ},0\oplus\phi_{\circ})=0$, whenever 
$$
\kappa_{\circ}=\sqrt\rho\ \omega_{M}\,,\qquad \phi_{\circ}={S_{0}}^{-1}1-{\rho\,\omega^{2}_{M}}\varphi_{\circ}^{\perp}\,,
$$
$$\varphi_{\circ}^{\perp}:=\left(\frac12+\gamma_{1}\SL_{0}
\right)_{\!\!\perp}^{-1}
P_{\perp}K_{(2)}^{*}{S_{0}}^{-1}1\,.
$$
Now we introduce the map
\be\label{Phi2}
\Phi_{2}:U\times\CO\times \widehat X
\to \CO\times \widehat X\,,\qquad \Phi_{2}(\ve,\kappa,u\oplus\phi)
:=([\phi_{\circ},\phi-\phi_{\circ}]_{-\frac12},{F}_{2}(\ve,\kappa,u\oplus\phi))\,.
\ee
Obviously, 
$$
\Phi_{2}(0,\pm\kappa_{\circ},0\oplus\phi_{\circ})=(0,0)\,.
$$
Let us denote by 
$$
D_{\pm}\Phi_{2}: \CO\times  \widehat X
\to \CO\times \widehat X
$$
the Frechet  differential at $(0,\pm\kappa_{\circ},0\oplus\phi_{\circ})$ of the map $(\kappa,u\oplus\phi)\mapsto\Phi_{2}(0,\kappa,u\oplus\phi)$, i.e., 
\begin{align*}
D_{\pm}\Phi_{2}(\kappa,u\oplus\phi)=\bigg([\phi_{\circ},\phi]_{-\frac12},u\oplus&\bigg(
  \pm 2\kappa_{\circ}\kappa K_{(2)}^{*}P\phi_{\circ}+\gamma^{\-}_{1}N_{0}u\\
&+\left(\frac12+\gamma_{1}\SL_{0}
\right)_{\!\!\perp}\phi+
\big(\kappa_{\circ}^{2}K_{(2)}^{*}-2\rho\,\gamma_{1}\SL_{0}\big)P\phi\bigg)\bigg)\,.
\end{align*}
Now, for any given $( z_{*},u_{*}\oplus\phi_{*})\in\CO\times\widehat X$,   we prove 
the existence of a unique solution $(\kappa, u\oplus\phi)\in\CO\times\widehat X$ to the equation 
\be\label{uni2}
D_{\pm}\Phi_{2}(\kappa,u\oplus\phi)=( z_{*},u_{*}\oplus\phi_{*})\,.
\ee 
By $P\big(\kappa_{\circ}^{2}K_{(2)}^{*}-2\rho\,\gamma_{1}\SL_{0}\big)P=0$, \eqref{uni2} is equivalent to the system
$$
\begin{cases}
[\phi_{\circ},\phi]_{-\frac12}=z_{*}\\
u=u_{*}
\\
P\big(\pm 
2\kappa_{\circ}\kappa K_{(2)}^{*}P\phi_{\circ}+\gamma^{\-}_{1}N_{0}u \big)=P\phi_{*}
\\
P_{\perp}\big(\pm 
2\kappa_{\circ}\kappa K_{(2)}^{*}P\phi_{\circ}+\gamma^{\-}_{1}N_{0}u
+\left(\frac12+\gamma_{1}\SL_{0}
\right)P_{\perp}\phi+
\kappa_{\circ}^{2}K_{(2)}^{*}P\phi\big)=P_{\perp}\phi_{*}\,.
\end{cases}
$$
The latter re-write as
$$
\begin{cases}
[\phi_{\circ},\phi]_{-\frac12}=z_{*}\\
u =u_{*}
\\
\mp 2\omega_{M}^{-2}\kappa_{\circ}\kappa \,P\phi_{\circ}
=P\big(\phi_{*}-\gamma^{\-}_{1}N_{0}u_{*} \big)
\\
P\phi=c_{\phi}\,{S_{0}}^{-1}1
\\
\left(\frac12+\gamma_{1}\SL_{0}
\right)_{\!\!\perp}\phi=P_{\perp}\big(\phi_{*}-\gamma^{\-}_{1}N_{0}u_{*}-(c_{\phi}\kappa_{\circ}^{2}\pm 2\kappa_{\circ}\kappa )K_{(2)}^{*}{S_{0}}^{-1}1\big)\end{cases}
$$
and so the unique solution of \eqref{uni2} is given by
\be\label{sol2}
\begin{cases}
u=u_{*}
\\
\kappa=\mp\,\omega_{M}^{2}(2\kappa_{\circ})^{-1}\,[\phi_{\circ}, P\phi_{\circ}]_{-\frac12}^{-1}\, [\phi_{\circ}, P\big(\phi_{*}-\gamma^{\-}_{1}N_{0}u_{*} \big)]_{-\frac12}
\\
\,\,\,\,\, =\mp\,({2\sqrt\rho\, \omega_{M}|\Omega|}\,)^{-1}\,\langle 1,\phi_{*}-\gamma^{\-}_{1}N_{0}u_{*}\rangle_{\frac12,-\frac12}\\
\phi=c_{\phi}\,({S_{0}}^{-1}1-\kappa_{\circ}^{2}\varphi_{\circ}^{\perp})+\varphi_{\perp}=
c_{\phi}\phi_{\circ}+\varphi_{\perp}\,,\quad c_{\phi}=c_{\Omega}^{-1}\big(z_{*}-[\phi_{\circ},\varphi_{\perp}]_{-\frac12}\big)
\\
\varphi_{\perp}=\left(\frac12+\gamma_{1}\SL_{0}
\right)_{\!\!\perp}^{-1}
P_{\perp}\big(\phi_{*}-\gamma^{\-}_{1}N_{0}u_{*}\mp 2\kappa_{\circ}\kappa K_{(2)}^{*}{S_{0}}^{-1}1\big)\,.
\end{cases}
\ee
Therefore, by the implicit function argument recalled at the beginning of the Section, in a sufficiently small neighborhood of $0\in\RE$, there is a $\CO\times (\widehat X\backslash\{0\})$-valued, analytic  function $\ve\mapsto (\kappa_{\pm}(\ve), u _{\pm}(\ve)\oplus\phi_{\pm} (\ve))$,
$$
\kappa_{\pm}(\ve)=\pm\sqrt\rho\,\omega_{M}+\sum_{n=1}^{+\infty}\kappa_{\pm}^{(n)}\ve^{n}
\,,\qquad
u_{\pm} (\ve)=\sum_{n=1}^{+\infty} u_{\pm} ^{(n)}\ve^{n}\,,\qquad
\phi_{\pm} (\ve)=\phi_{\circ}+\sum_{n=1}^{+\infty}\phi_{\pm} ^{(n)}\ve^{n}
$$
which uniquely solves \eqref{Fe2}. One can explicitly compute all the coefficients appearing in the expansions. In particular,
 $$
\big(\kappa_{\pm}^{(1)},u_{\pm} ^{(1)}\oplus\phi_{\pm} ^{(1)}\big)=-\left(D_{\pm}\Phi_{2}\right)^{-1}
\left(
\left.\frac{d\ }{d\ve}\right|_{\ve=0}\Phi_{2}(\ve,\pm\sqrt\rho\,\omega_{M},0\oplus\phi_{\circ})
\right).
$$
equivalently,
$$
D_{\pm}\Phi_{2}\big(\kappa_{\pm}^{(1)},u_{\pm} ^{(1)}\oplus\phi_{\pm} ^{(1)}\big)=\left(0,-\kappa_{\circ}^{2}\left(\left(
v_{1}^{2}1_{\Omega}\SL_{0}P\phi_{\circ}\right)\oplus\left(
\mp\kappa_{\circ}K_{(3)}^{*}P\phi_{\circ}\right)\right)\right).
$$
Hence, by \eqref{sol2} and \eqref{k2k3},
\begin{align*}
\kappa_{\pm}^{(1)}=&\pm\,\frac{\kappa_{\circ}}{2|\Omega|}\, 
[\phi_{\circ},P( \mp\kappa_{\circ}K_{(3)}^{*}P\phi_{\circ}
- v_{1}^{2}\gamma^{\-}_{1}N_{0}1_{\Omega}\SL_{0}P\phi_{\circ})
]_{-\frac12}\\
=&\pm\,\frac{\sqrt\rho\,\omega_{M}}{2|\Omega|}\, 
[\phi_{\circ},\mp\sqrt\rho\,\omega_{M}\,i\,\frac{|\Omega|}{4\pi}\,P\phi_{\circ}
- v_{1}^{2}P\gamma^{\-}_{1}N_{0}1_{\Omega}\SL_{0}P\phi_{\circ}]_{-\frac12}
\\
=&-i\,\frac{c_{\Omega}}{8\pi} \,\rho\,\omega_{M}^{2}\mp \,\frac{\sqrt\rho\,\omega_{M}}{2|\Omega|}\,v_{1}^{2}
{\langle 1, 
\gamma^{\-}_{1}N_{0}1_{\Omega}\SL_{0}{S_{0}}^{-1}1\rangle
_{\frac12,-\frac12}}\,.
\end{align*}
Since $u_{\Omega}:=1_{\Omega}\SL_{0}{S_{0}}^{-1}1$ solves 
$$
\begin{cases}\Delta_{\Omega}u_{\Omega}=0\\
\gamma^{\-}_{0}u_{\Omega}=1\,,\end{cases}
$$
one gets $u_{\Omega}=1$ and so, by Green's formula,
\begin{align*}
&\langle 1, 
\gamma^{\-}_{1}N_{0}1_{\Omega}\SL_{0}{S_{0}}^{-1}1\rangle
_{\frac12,-\frac12}=\langle 1,\gamma^{\-}_{1}N_{0}1\rangle
_{\frac12,-\frac12}=
\langle 1, 
\Delta_{\Omega}^{\max} N_{0}1\rangle
_{L^{2}(\Omega)}
=-\langle 1, 
1\rangle
_{L^{2}(\Omega)}=-|\Omega|\,,
\end{align*}
In conclusion, by $c_{\Omega}=\omega_{M}^{2}|\Omega|$, one gets
$$
\kappa_{\pm}(\ve)=\pm\left(1+\frac{v_{1}^{2}}{2}\,\ve\right)\sqrt\rho\,\omega_{M}-i\,\frac{|\Omega|}{8\pi} \,\rho\,\omega_{M}^{4}\ve+O(\ve^{2})
$$
and so $k^{2}_{\pm}(\ve)\in\mathsf e(A(\ve))$, where
$$
\kappa^{2}_{\pm}(\ve)=\left(1+v_{1}^{2}\ve\right)\rho\,\omega^{2}_{M}\mp i\,\frac{|\Omega|}{4\pi} \,\rho^{3/2}\omega_{M}^{5}\ve+O(\ve^{2})\,.
$$
In the Case 3 and Case 4 below we use some notations which are the same as for Case 2. However, be aware that such notations could denote vectors and constants which do not necessarily coincide between the different cases; if so, the definitions are given explicitly.\par
\subsection{Case 3. } 
Let us define the map
$$
{F}_{3}:U\times\CO\times  \widehat X
\to \widehat X
$$
\begin{align*}
&{F}_{3}(\ve,\kappa, u\oplus\phi):=M^{(1,1,1)}_{\kappa}(\ve)
\begin{bmatrix}u\\P\phi+\ve P_{\perp}\phi\end{bmatrix}
\\
=&\begin{bmatrix}v^{2}(\ve)\ve^{-1}+(v^{2}(\ve)-1)\ve\kappa^{2}N_{\ve\kappa}&({v^{2}(\ve)}-1)\kappa^{2}1_{\Omega}\SL_{\ve\kappa})
\\
\frac{1-\rho(\ve)}{1+\rho(\ve)}\,\gamma^{\-}_{1}N_{\ve\kappa}&\ve^{-1}\left(\frac12+\frac{1-\rho(\ve)}{1+\rho(\ve)}\,
\gamma_{1}\SL_{\ve\kappa}\right)\end{bmatrix}
\begin{bmatrix}u\\P\phi+\ve P_{\perp}\phi\end{bmatrix}
,
\end{align*}
where  $U\subset\CO$ is a small open neighborhood of the origin and
$\ve\mapsto v^{2}(\ve)$ and $\ve\mapsto \rho(\ve)$ are analytic function as in \eqref{caso3}.
\par  
We look for a solution $\ve\mapsto (\kappa(\ve), u(\ve)\oplus\phi(\ve))\in\CO_{-}\times(\widehat X\backslash\{0\})$ of the equation
\be\label{Fe3} 
{F}_{3}(\ve,\kappa(\ve), u(\ve)\oplus\phi(\ve))=0\,.
\ee
By the same kind of arguments as in Case 2 and by 
$$
\frac{1-\rho(\ve)}{1+\rho(\ve)}=1-2\rho\,\ve+(2\rho^{2}-\rho_{1})\ve^{2}+O(\ve^{3})
$$ 
one obtains
\begin{align*}
&\ve^{-1}\left(\frac12+\frac{1-\rho(\ve)}{1+\rho(\ve)}\ \gamma_{1}\SL_{\ve\kappa}\right)
\\=&
\ve^{-1}\left(\frac12+\gamma_{1}\SL_{0}
\right)_{\!\!\perp}-2\rho\, \gamma_{1}\SL_{0} +\big((2\rho^{2}-\rho_{1})\gamma_{1}\SL_{0}+
\kappa^{2}K_{(2)}^{*}\big)\ve+O(\ve^{2}).
\end{align*}
This gives, by $\gamma_{1}\SL_{0}P\phi=-\frac12\,P\phi$ and by $\gamma_{1}^{\-}N_{(1)}=0$,
\begin{align*}
&{F}_{3}(\ve,\kappa, u\oplus\phi)
=
\begin{bmatrix}v^{2}&-\kappa^{2}1_{\Omega}\SL_{0}P\\\gamma_{1}^{\-}N_{0}&\left(\frac12+\gamma_{1}\SL_{0}
\right)_{\!\!\perp}+\rho\,P\end{bmatrix}\begin{bmatrix} u\\ \phi\end{bmatrix}\\
+&\begin{bmatrix}v_{1}^{2}-\kappa^{2}N_{0}&\kappa^{2}1_{\Omega}\big((v^{2}\SL_{0}-\kappa\SL_{(1)})P-\SL_{0}P_{\perp}\big)\\
-2\rho \gamma_{1}^{\-}N_{0}&\kappa^{2}K_{(2)}^{*}P-2\rho\,\gamma_{1}\SL_{0}P_{\perp}+
(\frac12\,\rho_{1}-\rho^{2})P\end{bmatrix}\begin{bmatrix} u\\ \phi\end{bmatrix}\ve+O(\ve^{2})\,.
\end{align*}
Therefore, ${F}_{3}(0,\kappa, u\oplus\phi)=0\oplus 0$ if and only if $(\kappa, u,\phi)$ solves 
\be\label{eq3}
\begin{cases}
v^{2}u-\kappa^{2}1_{\Omega}\SL_{0}P\phi=0\\
\gamma_{1}^{\-}N_{0}u+\left(\frac12+\gamma_{1}\SL_{0}
\right)_{\!\!\perp}\phi+\rho\,P\phi
=0\,.\end{cases}
\ee
Taking $\phi$ such that $P\phi= S_{0}^{-1}1$, \eqref{eq3} rewrites as
 $$
\begin{cases}
u= v^{-2}\kappa^{2}\,1\\
v^{-2}\kappa^{2}P\gamma_{1}^{\-}N_{0}1+\rho\,S_{0}^{-1}1=0\\
\left(\frac12+\gamma_{1}\SL_{0}
\right)_{\!\!\perp}\phi=- v^{-2}\kappa^{2}P_{\perp}\gamma_{1}^{\-}N_{0}1\,.\end{cases}
$$
Since
$$
PS_{0}^{-1}1=S_{0}^{-1}1\,,\qquad P\gamma_{1}^{\-}N_{0}1=c_{\Omega}^{-1}\langle 1,\gamma_{1}^{\-}N_{0}1\rangle_{\frac12,-\frac12}S_{0}^{-1}1=-c_{\Omega}^{-1}|\Omega|S_{0}^{-1}1=-\omega_{M}^{-2}S_{0}^{-1}1\,,
$$
one has that ${F}_{3}(0,\pm\kappa_{\circ},u_{\circ},\phi_{\circ})=0$ whenever
$$
\kappa_{\circ}=v\sqrt\rho\,\omega_{M}\,,\qquad u_{\circ}=\rho\,\omega_{M}^{2}1\,,\qquad \phi_{\circ}= S_{0}^{-1}1-\rho\,\omega_{M}^{2}\varphi_{\circ}^{\perp}\,,
$$
$$
\varphi_{\circ}^{\perp}=\left(\frac12+\gamma_{1}\SL_{0}
\right)_{\!\!\perp}^{-1}P_{\perp}\gamma_{1}^{\-}N_{0}1\,.
$$
Now we introduce the map
$$
\Phi_{3}:U\times\CO\times \widehat X
\to \CO\times \widehat X\,,
$$
\be\label{Phi3}
\Phi_{3}(\ve,\kappa,u\oplus\phi)
:=\big(\langle u_{\circ},u-u_{\circ}\rangle_{L^{2}(\Omega)}+[\phi_{\circ},\phi-\phi_{\circ}]_{-\frac12},{F}_{3}(\ve,\kappa,u\oplus\phi)\big)\,.
\ee
Obviously, 
$$
\Phi_{3}(0,\pm\kappa_{\circ},u_{\circ}\oplus\phi_{\circ})=(0,0)\,.
$$
Let us denote by 
$$
D_{\pm}\Phi_{3}: \CO\times  \widehat X
\to \CO\times \widehat X
$$
the Frechet  differential at $(\pm\kappa_{\circ},u_{\circ}\oplus\phi_{\circ})$ of the map $(\kappa,u\oplus\phi)\mapsto\Phi_{3}(0,\kappa,u\oplus\phi)$, i.e., 
\begin{align*}
D_{\pm}\Phi_{3}(\kappa,u\oplus\phi)=\bigg(\langle u_{\circ},u\rangle_{L^{2}(\Omega)}+[\phi_{\circ},\phi]_{-\frac12},&(\mp 2\kappa_{\circ}\kappa 1_{\Omega}\SL_{0}P\phi_{\circ}+v^{2}u-\kappa_{\circ}^{2}1_{\Omega}\SL_{0}P\phi)
\oplus\\
&\left(\gamma_{1}^{\-}N_{0}u+\left(\frac12+\gamma_{1}\SL_{0}
\right)_{\!\!\perp}\phi+\rho\,P\phi\right)\bigg)
\end{align*}
Now, for any given $( z_{*},u_{*}\oplus\phi_{*})\in\CO\times\widehat X$,   we prove 
the existence of a unique solution $(\kappa, u\oplus\phi)\in\CO\times\widehat X$ to the equation 
\be\label{uni3}
D_{\pm}\Phi_{3}(\kappa,u\oplus\phi)=( z_{*},u_{*}\oplus\phi_{*})\,.
\ee 
equivalently, to the system of equations
\be\label{sys}
\begin{cases}
\langle u_{\circ},u\rangle_{L^{2}(\Omega)}+[\phi_{\circ},\phi]_{-\frac12}=z_{*}
\\
v^{2}u \mp 2\kappa_{\circ}\kappa1-c_{\phi}\kappa_{\circ}^{2}1=u_{*}
\\
P\gamma_{1}^{\-}N_{0}u +c_{\phi}\rho\,{S_{0}}^{-1}1=P\phi_{*}
\\
P\phi =c_{\phi}\,{S_{0}}^{-1}1
\\
P_{\perp}\left(\gamma_{1}^{\-}N_{0}u +\left(\frac12+\gamma_{1}\SL_{0}
\right)P_{\perp}\phi \right)=P_{\perp}\phi_{*}
\,.
\end{cases}
\ee
Then, by
$$
[\phi_{\circ},S_{0}^{-1}1]_{-\frac12}=[S_{0}^{-1}1,S_{0}^{-1}1]_{-\frac12}=c_{\Omega}\,,\quad[\phi_{\circ},P\gamma_{1}^{\-}N_{0}1]_{-\frac12}=[S_{0}^{-1}1,\gamma_{1}^{\-}N_{0}1]_{-\frac12}=-|\Omega|\,,
$$
\eqref{sys} gives
\be\label{sol3.1}
\begin{cases}
\kappa=\pm(2\kappa_{\circ})^{-1}(\widehat\kappa-c_{\phi}\kappa_{\circ}^{2})
\\
u =v^{-2}\big(u_{*}+\widehat\kappa 1)
\\
\phi =c_{\phi}S_{0}^{-1}1-v^{-2}\widehat\kappa\varphi_{\circ}^{\perp}+\varphi_{*}^{\perp}\,,
\end{cases}
\ee
where
$$
\varphi_{*}^{\perp}=\left(\frac12+\gamma_{1}\SL_{0}
\right)_{\!\!\perp}^{-1}P_{\perp}\big(\phi_{*}-v^{-2}\gamma^{\-}_{1}N_{0}u_{*}\big)
$$
and $(\widehat\kappa,c_{\phi})$ is the solution of  
\be\label{sol3.2}
\begin{cases}
-v^{-2}|\Omega|\,\widehat\kappa+ c_{\Omega}\rho\,c_{\phi}= [\phi_{\circ},P\big(\phi_{*}-v^{-2}\gamma^{\-}_{1}N_{0}u_{*}\big)]_{-\frac12}
\\
v^{-2}\big(\langle u_{\circ},1\rangle_{L^{2}(\Omega)}
-[\phi_{\circ},\varphi_{\circ}^{\perp}]_{-\frac12}\big)\widehat\kappa+c_{\Omega}c_{\phi}=z_{*}-v^{-2}\langle u_{\circ},u_{*}\rangle_{L^{2}(\Omega)}-
[\phi_{\circ},\varphi_{*}^{\perp}]_{-\frac12}
\,.
\end{cases}
\ee
Since
\begin{align*}
&\det\begin{bmatrix}
-v^{-2}|\Omega|&c_{\Omega}\rho\\
v^{-2}\big(\langle u_{\circ},1\rangle_{L^{2}(\Omega)}
-[\phi_{\circ},\varphi_{\circ}^{\perp}]_{-\frac12}\big)&c_{\Omega}
\end{bmatrix}\\
=&-v^{-2}c_{\Omega}\big(|\Omega|+
\rho\,\langle u_{\circ},1\rangle_{L^{2}(\Omega)}
-\rho\,[\phi_{\circ},\varphi_{\circ}^{\perp}]_{-\frac12}\big)\\
=&-v^{-2}c_{\Omega}\big(|\Omega|+
\rho^{2}\omega_{M}^{2}\big(|\Omega|
+\|\varphi_{\circ}^{\perp}\|^{2}_{\widehat H^{-1/2}(\Gamma)}\big)\big)\not=0\,,
\end{align*}
the system \eqref{sol3.2} has a unique solution. \par
Therefore, by the implicit function argument recalled at the beginning of the Section, the implicit function theorem for analytic, in a sufficiently small neighborhood of $0\in\RE$, there is a $\CO\times (\widehat X\backslash\{0\})$-valued, analytic  function $\ve\mapsto (\kappa_{\pm}(\ve), u_{\pm} (\ve)\oplus\phi_{\pm} (\ve))$,
$$
\kappa_{\pm}(\ve)=\pm v\sqrt\rho\,\omega_{M}+\sum_{n=1}^{+\infty}\kappa_{\pm}^{(n)}\ve^{n}
\,,\qquad
u_{\pm} (\ve)=u_{\circ}+\sum_{n=1}^{+\infty} u_{\pm} ^{(n)}\ve^{n}\,,\qquad
\phi_{\pm} (\ve)=\phi_{\circ}+\sum_{n=1}^{+\infty}\phi_{\pm} ^{(n)}\ve^{n}
$$
which uniquely solves \eqref{Fe3}. One can explicitly compute all the coefficients appearing in the expansions. In particular,
 $$
\big(\kappa_{\pm}^{(1)},u_{\pm} ^{(1)}\oplus\phi_{\pm} ^{(1)}\big)=
-\left(D_{\pm}\Phi_{3}\right)^{-1}
\left(
\left.\frac{d\ }{d\ve}\right|_{\ve=0}\Phi_{3}(\ve,\pm v\sqrt\rho\,\omega_{M},u_{\circ}\oplus\phi_{\circ})
\right).
$$
equivalently,
$$
D_{\pm}\Phi_{3}(\kappa_{\pm}^{(1)},u ^{(1)}\oplus\phi ^{(1)})=\big(0,
\widehat u_{\circ}\oplus\widehat\phi_{\circ}\big)\,,
$$
where
$$
\begin{bmatrix}\widehat u_{\circ}\\\widehat\phi_{\circ}\,\end{bmatrix}:=
-
\begin{bmatrix}v_{1}^{2}-\kappa_{\circ}^{2}N_{0}&\kappa_{\circ}^{2}1_{\Omega}\big((v^{2}\SL_{0}-\kappa_{\circ}\SL_{(1)})P-\SL_{0}P_{\perp}\big)\\
-2\rho \gamma_{1}^{\-}N_{0}&\kappa_{\circ}^{2}K_{(2)}^{*}P-2\rho\,\gamma_{1}\SL_{0}P_{\perp}+
(\frac12\,\rho_{1}-\rho^{2})P\end{bmatrix}
\begin{bmatrix} u_{\circ}\\ \phi_{\circ}\end{bmatrix}.
$$
Hence, $\kappa_{\pm}^{(1)}$ is like $\kappa$ in \eqref{sol3.1}, replacing the couple  
$u_{*}$ and $\phi_{*}$ with  $\widehat u_{\circ}$ and $\widehat\phi_{\circ}$ and taking $z_{*}=0$ in \eqref{sol3.2}. Writing for brevity $\kappa_{\pm}^{(1)}=r_{\pm}-iI_{\pm}$, with both $r_{\pm}$ and $I_{\pm}$ real, one has
$$
\kappa_{\pm}(\ve)=\pm v\sqrt\rho\,\omega_{M}+r_{\pm}\ve -iI_{\pm}\ve+O(\ve^{2})
$$
and then
$$
\kappa^{2}_{\pm}(\ve)=v^{2}\rho\,\omega_{M}^{2}\pm 2v\sqrt\rho\,\omega_{M}r_{\pm}\ve\mp i  2v\sqrt\rho\,\omega_{M}I_{\pm}\ve+O(\ve^{2})
$$
belongs to $\mathsf e(A(\ve))$.
\subsection{Case 4.}\label{sub4}
Let us define the map
$$
{F}_{4}:U\times\CO\times   \widehat X
\to \widehat X
$$
\begin{align*}
&{F}_{4}(\ve,\kappa, u\oplus\phi):=M^{(2,1,1)}_{\kappa}(\ve)
\begin{bmatrix}u\\\phi\end{bmatrix}\\
=&\begin{bmatrix}v^{2}(\ve)\ve^{-2}+(v^{2}(\ve)-1)\kappa^{2}N_{\ve\kappa}&({v^{2}(\ve)}-1)\kappa^{2}1_{\Omega}\SL_{\ve\kappa}
\\
\frac{1-\rho(\ve)}{1+\rho(\ve)}\,\gamma^{\-}_{1}N_{\ve\kappa}&\frac12+\frac{1-\rho(\ve)}{1+\rho(\ve)}\,
\gamma_{1}\SL_{\ve\kappa}
\end{bmatrix}\begin{bmatrix}u\\\phi\end{bmatrix}
,
\end{align*}
where  $U\subset\CO$ is a small open neighborhood of the origin and $\ve\mapsto v^{2}(\ve)$ and $\ve\mapsto \rho(\ve)$ are analytic function as in \eqref{caso4}.  
One has  
$$
\frac{1-\rho(\ve)}{1+\rho(\ve)}=1-2\rho\,\ve+(2\rho^{2}-\rho_{1})\ve^{2}+O(\ve^{3})
$$ 
and\begin{align*}
&\frac{1-\rho(\ve)}{1+\rho(\ve)}\, \gamma_{1}\SL_{\ve\kappa}
=
\gamma_{1}\SL_{0}
-2\rho\, \gamma_{1}\SL_{0}\ve +\big((2\rho^{2}-\rho_{1})\gamma_{1}\SL_{0}+
\kappa^{2}K_{(2)}^{*}\big)\ve^{2}+O(\ve^{3})\,.
\end{align*}
Therefore (notice that $\gamma_{1}^{\-}N_{(1)}=0$), 
\begin{align*}
&{F}_{4}(\ve,\kappa, u\oplus\phi)=
\begin{bmatrix}v^{2}-\kappa^{2}N_{0}&-\kappa^{2}1_{\Omega}\SL_{0}
\\
\gamma^{\-}_{1}N_{0}&\frac12+
\gamma_{1}\SL_{0}
\end{bmatrix}\begin{bmatrix}u\\\phi\end{bmatrix}
+
\begin{bmatrix}v_{1}^{2}-\kappa^{3}N_{(1)}&-\kappa^{3}1_{\Omega}\SL_{(1)}\,
\\
-2\rho\,\gamma^{\-}_{1}N_{0}&-2\rho\, \gamma_{1}\SL_{0}
\end{bmatrix}\begin{bmatrix}u\\\phi\end{bmatrix}\ve\\
+&\begin{bmatrix}v_{2}^{2}+v^{2}\kappa^{2}N_{0}-\kappa^{4}N_{(2)}&v^{2}\kappa^{2}1_{\Omega}\SL_{0}-\kappa^{4}1_{\Omega}\SL_{(2)}\,
\\
((2\rho^{2}-\rho_{1})\gamma_{1}\SL_{0}+
\kappa^{2}K_{(2)}^{*})\gamma^{\-}_{1}N_{0}+\gamma^{\-}_{1}N_{(2)}& 
(2\rho^{2}-\rho_{1})\gamma_{1}\SL_{0}+
\kappa^{2}K_{(2)}^{*}
\end{bmatrix}\begin{bmatrix}u\\\phi\end{bmatrix}\ve^{2}+O(\ve^{3})\,,
\end{align*}
\begin{remark}\label{k=0}
It is immediate to check that $u\oplus\phi\in \widehat X\backslash\{0\}$, solves  ${F}_{4}(0,0, u\oplus\phi)=0$ if and only if $u=0$ and $P_{\perp}\phi=0$, i.e., if and only if 
$$
u\oplus\phi=0\oplus c_{0}S_{0}^{-1 }1\,,\quad c_{0}\in\CO\backslash\{0\}\,.
$$ The case $\kappa\not=0$ is more involved and is described in the following Lemma.
\end{remark}
\begin{lemma} $(\kappa, u\oplus\phi)\in (\CO\backslash\{0\})\times (\widehat X\backslash\{0\})$ solves the equation ${F}_{4}(0,\kappa, u\oplus\phi)=0$ if and only if 
$$
(\kappa,u\oplus\phi)=(\pm v\nu^{1/2},u_{\nu}\oplus\phi_{\nu})\,,
$$
where 
$(\nu,u_{\nu})$ is an eigenpair  of the Neumann Laplacian  $-\Delta_{\Omega}^{N}$ in $L^{2}(\Omega)$ and
\be\label{phinu}
\phi_{\nu}=S_{0}^{-1}\gamma_{0}^{\-}({\nu}^{-1}-N_{0})u_{\nu}
\,.
\ee
\end{lemma}
\begin{proof} Obviously, ${F}_{4}(0,\kappa, u\oplus\phi)=0$ if and only if  
\be\label{s4}
\begin{cases}
\left(\frac{v^{2}}{\kappa^{2}}-N_{0}\right)u-1_{\Omega}\SL_{0}\phi=0\\
\gamma^{\-}_{1}N_{0}u+\left(\frac12+
\gamma_{1}\SL_{0}\right)\phi=0\,.
\end{cases}
\ee
Therefore, if $(\kappa, u\oplus\phi)\not=(0,0\oplus0)$ is a solution, then
$$\begin{cases}
\frac{v^{2}}{\kappa^{2}}\gamma^{\-}_{1}u-\gamma^{\-}_{1}N_{0}u-\gamma^{\-}_{1}\SL_{0}\phi=0\\
\gamma^{\-}_{1}N_{0}u+\frac12\,\phi+
\frac12\,(\gamma_{1}^{\+}+\gamma_{1}^{\-})\SL_{0}\phi=0\,.
\end{cases}
$$
Adding the two lines, one obtains
$$
\frac{v^{2}}{\kappa^{2}}\,\gamma^{\-}_{1}u+\frac12\,\phi+
\frac12\,[\gamma_{1}]\SL_{0}\phi=0\,.
$$
Then, by $[\gamma_{1}]\SL_{0}\phi=-\phi$, 
$$
\gamma^{\-}_{1}u=0\,.
$$
Furthermore, by 
$$
0=-\Delta^{\max}_{\Omega}\left(\left(\frac{v^{2}}{\kappa^{2}}-N_{0}\right)u-1_{\Omega}\SL_{0}\phi\right)=-\frac{v^{2}}{\kappa^{2}}\Delta^{\max}_{\Omega}u-u\,,
$$
one has
$$
\frac{\kappa^{2}}{v^{2}}\in\sigma(-\Delta_{\Omega}^{N})=\sigma_{disc}(-\Delta_{\Omega}^{N})
$$
and so $(v^{-2}\kappa^{2},u)\equiv ({\nu},u_{\nu})$ is an eigenpair  of the Neumann Laplacian $-\Delta_{\Omega}^{N}$.\par
By
$$
P\gamma_{1}^{\-}N_{0}u_{\nu}=c_{\Omega}^{-1}\langle1,\gamma_{1}^{\-}N_{0}u_{\nu}\rangle_{\frac12,-\frac12}S_{0}^{-1}1=-c_{\Omega}^{-1}\langle1,u_{\nu}\rangle_{L^{2}(\Omega)}S_{0}^{-1}1=0\,,
$$
there follows 
$$
P_{\perp}\gamma_{1}^{\-}N_{0}u_{\nu}=\gamma_{1}^{\-}N_{0}u_{\nu}
$$
and so, the second equation in \eqref{s4} gives
\be\label{DN3}
P_{\perp}\phi\equiv \phi_{\nu}^{\perp}=-\left(\frac12+
\gamma_{1}\SL_{0}\right)_{\!\!\perp}^{-1}\!\!\gamma^{\-}_{1}N_{0}u_{\nu}\,.
\ee
By the relation 
\be\label{DN1}
\DN_{0}=\left(\frac12+
\gamma_{1}\SL_{0}\right)S_{0}^{-1}
\ee
and by 
\be\label{DN2}
\DN_{0}\gamma_{0}^{\-}({\nu}^{-1}-N_{0})u_{\nu}=\gamma_{1}^{\-}({\nu}^{-1}-N_{0})u_{\nu}=
-\gamma_{1}^{\-}N_{0}u_{\nu}\,,
\ee
one gets
$$
-\left(\frac12+
\gamma_{1}\SL_{0}\right)_{\!\!\perp}^{-1}\!\!\gamma_{1}^{\-}N_{0}u_{\nu}=P_{\perp}S_{0}^{-1}\gamma_{0}^{\-}({\nu}^{-1}-N_{0})u_{\nu}
$$
and so 
$$
\phi_{\nu}^{\perp}
=P_{\perp}S_{0}^{-1}\gamma_{0}^{\-}({\nu}^{-1}-N_{0})u_{\nu}\,.
$$
Writing $P\phi=c_{\nu}S_{0}^{-1}1$, by the first equation in \eqref{s4} one gets 
$$
h_{\nu}:=\left(\frac1{\nu}-N_{0}\right)u_{\nu}-1_{\Omega}\SL_{0}\phi_{\nu}^{\perp}=c_{\nu}1
$$
and so 
$$
c_{\nu}|\Omega|=\langle 1,h_{\nu}\rangle_{L^{2}(\Omega)}\,.
$$
Therefore we proved that if $(\kappa, u\oplus\phi)\not=(0,0\oplus 0)$ solves ${F}_{4}(0,\kappa, u\oplus\phi)=0\oplus0$, then $(\kappa, u\oplus\phi)=(\pm v{\nu}^{1/2},u_{\nu}\oplus\phi_{\nu})$ where $({\nu},u_{\nu})$ is an eigenpair  of the Neumann Laplacian and 
\be\label{p-nu}
\phi_{\nu}=|\Omega|^{-1}\langle 1,h_{\nu}\rangle_{L^{2}(\Omega)}\,S_{0}^{-1}1+\phi_{\nu}^{\perp}\,.
\ee 
Now, let us show that $h_{\nu}$ is a constant function. Since $h_{\nu}$ is harmonic, $h_{\nu}$  is constant  if and only if $\gamma_{0}^{\-}h_{\nu}$ is. By
\begin{align*}
h_{\nu}=&({\nu}^{-1}-N_{0})u_{\nu}-1_{\Omega}\SL_{0}
P_{\perp}S_{0}^{-1}\gamma_{0}^{\-}({\nu}^{-1}-N_{0})u_{\nu}\\
=&({\nu}^{-1}-N_{0})u_{\nu}-1_{\Omega}\SL_{0}
S_{0}^{-1}\gamma_{0}^{\-}({\nu}^{-1}-N_{0})u_{\nu}\\
&+
1_{\Omega}\SL_{0}
PS_{0}^{-1}\gamma_{0}^{\-}({\nu}^{-1}-N_{0})u_{\nu}\,,
\end{align*}
and by $S_{0}=\gamma_{0}\SL_{0}=\gamma_{0}^{\-}1_{\Omega}\SL_{0}$, there follows
$$
\gamma_{0}^{\-}h_{\nu}=S_{0}
PS_{0}^{-1}\gamma_{0}^{\-}({\nu}^{-1}-N_{0})u_{\nu}=
c_{\Omega}^{-1}\langle 1,S_{0}^{-1}\gamma_{0}^{\-}({\nu}^{-1}-N_{0})u_{\nu}\rangle_{\frac12,-\frac12}1\,.
$$
Therefore, $h_{\nu}$ is a constant and
$$
h_{\nu}=c_{\nu}=c_{\Omega}^{-1}\langle 1,S_{0}^{-1}\gamma_{0}^{\-}({\nu}^{-1}-N_{0})u_{\nu}\rangle_{\frac12,-\frac12}1\,.
$$ 
By \eqref{p-nu}, 
$$
\phi_{\nu}=c_{\nu}S_{0}^{-1}1+\phi_{\nu}^{\perp}=S_{0}^{-1}\gamma_{0}^{\-}({\nu}^{-1}-N_{0})u_{\nu}\,.
$$
Conversely, let us now prove that ${F}_{4}(0,\pm v{\nu}^{1/2},u_{\nu}\oplus\phi_{\nu})=0$, equivalently
$$
\begin{cases}
(\nu^{-1}-N_{0})u_{\nu}-1_{\Omega}\SL_{0}\phi_{\nu}=0\\
\gamma_{1}^{\-}N_{0}u_{\nu}+\left(\frac12+\gamma_{1}\SL_{0}\right)\phi_{\nu}=0\,.
\end{cases}
$$
Since both $1_{\Omega}\SL_{0}\phi_{\nu}$ and $({\nu}^{-1}-N_{0})u_{\nu}$ are harmonic and 
$$\gamma_{0}^{\-}1_{\Omega}\SL_{0}\phi_{\nu}=S_{0}\phi_{\nu}=\gamma_{0}^{\-}({\nu}^{-1}-N_{0})u_{\nu}\,,
$$ 
they coincide. Furthermore, by \eqref{DN1} and \eqref{DN2},
$$
\gamma_{1}^{\-}N_{0}u_{\nu}=-\DN_{0}S_{0}\phi_{\nu}=-\left(\frac12+\gamma_{1}\SL_{0}\right)\phi_{\nu}\,.
$$
\end{proof}
\begin{remark}
By the proof of the previous Lemma, since $h_{\nu}$ is constant and hence $1_{\Omega}\SL_{0}P\phi_{\nu}=h_{\nu}$, one gets
\be\label{un}
N_{0}u_{\nu}+1_{\Omega}\SL_{0}\phi_{\nu}=\nu^{-1} u_{\nu}+h_{\nu}-\big((\nu^{-1} -N_{0})u_{\nu}
-1_{\Omega}\SL_{0}\phi_{\nu}^{\perp}\big)=\nu^{-1} u_{\nu}\,.
\ee
\end{remark}
Now we introduce the map
$$
\Phi_{4}:U\times\CO\times \widehat X
\to \CO\times \widehat X\,,
$$
\be\label{Phi4}
\Phi_{4}(\ve,\kappa,u\oplus\phi)
:=\big(\langle u_{\nu},u-u_{\nu}\rangle_{L^{2}(\Omega)}+[\phi_{\nu},\phi-\phi_{\nu}]_{-\frac12},{F}_{4}(\ve,\kappa,u\oplus\phi)\big)\,.
\ee
Obviously, 
$$
\Phi_{4}(0,\pm v\nu^{1/2},u_{\nu}\oplus\phi_{\nu})=(0,0)\,.
$$
Let us denote by
$$
D_{\pm}\Phi_{4}: \CO\times  \widehat X
\to \CO\times \widehat X
$$
the Frechet  differential at $(0,\pm v\nu^{1/2},u_{\nu}\oplus\phi_{\nu})$ of the map $(\kappa,u\oplus\phi)\mapsto\Phi_{4}(0,\kappa,u\oplus\phi)$, i.e., 
\begin{align*}
D_{\pm}\Phi_{4}(\kappa,u\oplus\phi)=\bigg(&\langle u_{\nu},u\rangle_{L^{2}(\Omega)}+[\phi_{\nu},\phi]_{-\frac12},\big( \mp2v{\nu}^{1/2}\kappa (N_{0}u_{\nu}+ 1_{\Omega}\SL_{0}\phi_{\nu}\big)\\
&+v^{2}(1-{\nu}N_{0})u -v^{2}{\nu}1_{\Omega}\SL_{0}\phi \big)
\oplus\left(\gamma^{\-}_{1}N_{0}u +\left(\frac12+\gamma_{1}\SL_{0}\right)\phi\right)
\bigg)
\end{align*}
Now, we study the invertibility of  $D_{\pm}\Phi_{4}$. Since the case $\nu=0$ is somehow more involved, we begin considering the case $\nu\not=0$.
\begin{lemma} Suppose that the Neumann eigenvalue $\nu\not=0$ is simple. Then,  for any given $( z_{*},u_{*}\oplus\phi_{*})\in\CO\times\widehat X$,   the equation 
\be\label{uni4}
D_{\pm}\Phi_{4}(\kappa,u\oplus\phi)=( z_{*},u_{*}\oplus\phi_{*})
\ee 
has a unique solution  $(\kappa, u\oplus\phi)\in\CO\times\widehat X$ . In particular,
\be\label{kappa}
\kappa
=\mp\,\frac12\,\nu^{1/2}\|u_{\nu}\|^{-2}_{L^{2}(\Omega)}\left(v^{-1}\langle u_{\nu},u_{*}\rangle_{L^{2}(\Omega)}+v\,\langle\gamma_{0}^{\-}u_{\nu},\phi_{*}\rangle_{\frac12,-\frac12}\right)\,
\ee
where $u_{\nu}$ denotes the eigenfunction corresponding to $\nu$.
\end{lemma}
\begin{proof} The equation \eqref{uni4} is equivalent
to the system of equations
$$
\begin{cases}
\langle u_{\nu},u\rangle_{L^{2}(\Omega)}+[\phi_{\nu},\phi]_{-\frac12}=z_{*}
\\
\mp2v{\nu}^{1/2}\kappa (N_{0}u_{\nu}+ 1_{\Omega}\SL_{0}\phi_{\nu})
+v^{2}(1-{\nu}N_{0})u -v^{2}{\nu}1_{\Omega}\SL_{0}\phi =u_{*}
\\
\gamma^{\-}_{1}N_{0}u +(\frac12+\gamma_{1}\SL_{0})\phi =\phi_{*}
\,.
\end{cases}
$$
By \eqref{un}, the latter rewrites as
\be\label{sys4}
\begin{cases}
\langle u_{\nu},u\rangle_{L^{2}(\Omega)}+[\phi_{\nu},\phi]_{-\frac12}=z_{*}
\\
\mp2v{\nu}^{-1/2}\kappa u_{\nu}
+v^{2}(1-{\nu}N_{0})u -v^{2}{\nu}1_{\Omega}\SL_{0}\phi =u_{*}
\\
\gamma^{\-}_{1}N_{0}u +(\frac12+\gamma_{1}\SL_{0})\phi =\phi_{*}
\,.
\end{cases}
\ee
Let us now suppose that \eqref{sys4} has at least one solution $(\kappa,u,\phi)$. By the relation in the second line, $u_{*}-v^{2}u \in H^{1}_{\Delta}(\Omega)$ and so its Neumann trace is well-defined as a distribution in $H^{-1/2}(\Gamma)$. Hence, by taking the Neumann trace in the second line and by adding the result to the third one, one obtains
$$
\frac12\,\phi +
\frac12\,[\gamma_{1}]\SL_{0}\phi =\nu^{-1}\gamma^{\-}_{1}(v^{-2}u_{*}-u )
+\phi_{*}\,.
$$
Thus, by $[\gamma_{1}]\SL_{0}\phi=-\phi$,
$$
\nu^{-1}\gamma^{\-}_{1}(u -v^{-2}u_{*})=\phi_{*}\,.
$$
Then, by 
\begin{align*}
\Delta_{\Omega}^{\max}\big(\mp2v^{-1}{\nu}^{-3/2}\kappa u_{\nu}
-N_{0}u -1_{\Omega}\SL_{0}\phi 
\big)
=&\mp2v^{-1}{\nu}^{-3/2}\kappa\Delta_{\Omega}^{\max} u_{\nu}
+u \\
=\pm2v^{-1}{\nu}^{-1/2}\kappa u_{\nu}
+u =&
-\Delta_{\Omega}^{\max}(\nu^{-1}(u -v^{-2}u_{*}))\,,
\end{align*}
the function 
$$
w_{*}:=\nu^{-1}(u -v^{-2}u_{*})
$$ 
solves the boundary value problem
\be\label{bvp4}
\begin{cases}
(-\Delta^{\max}_{\Omega}-\nu) w_{*}=f_{*}\\
\gamma^{\-}_{1}w_{*}=\phi_{*}\,,
\end{cases}
\ee
where 
$$
f_{*}=v^{-2}u_{*}\pm2v^{-1}{\nu}^{-1/2}\kappa u_{\nu}\,.
$$
By \cite[Theorem 4.10]{McLe}, \eqref{bvp4} is solvable if and only if 
$$
\langle u_{\nu},f_{*}\rangle_{L^{2}(\Omega)}+\langle\gamma_{0}^{\-}u_{\nu},\phi_{*}\rangle_{\frac12,-\frac12}=0\,.
$$
The latter entails \eqref{kappa}. Furthermore, the set of solutions of \eqref{bvp4} is 
$$
\{w_{*}=c_{*}u_{\nu}+\widetilde w_{*}\,,\quad c_{*}\in \CO\}\,
$$
where $\widetilde w_{*}$ is any fixed particular solution. Hence, $u$ is of the kind
\be\label{sol-u}
u =\nu(c_{*}u_{\nu}+\widetilde w_{*})+v^{-2}u_{*}\,,\quad c_{*}\in\CO\,.
\ee 
If $u$ is as in \eqref{sol-u}, then 
\begin{align*}
P\gamma_{1}^{\-}N_{0}u =&c_{\Omega}^{-1}\langle1,\gamma_{1}^{\-}N_{0}u \rangle_{\frac12,-\frac12}S_{0}^{-1}1=-c_{\Omega}^{-1}\langle1,u \rangle_{L^{2}(\Omega)}S_{0}^{-1}1\\
=&-c_{\Omega}^{-1}\big(\langle 1,\nu\widetilde w_{*}+v^{-2}u_{*} \rangle_{L^{2}(\Omega)} \big)S_{0}^{-1}1\\
=&-c_{\Omega}^{-1}\big(-\langle 1,\Delta^{\max}_{\Omega}\widetilde w_{*}\rangle_{L^{2}(\Omega)}+\langle 1,v^{-2}u_{*} -f_{*}\rangle_{L^{2}(\Omega)} \big)S_{0}^{-1}1\\
=&c_{\Omega}^{-1}\langle 1,\Delta^{\max}_{\Omega}\widetilde w_{*}\rangle_{L^{2}(\Omega)}=
c_{\Omega}^{-1}\langle 1,\phi_{*}\rangle_{\frac12,-\frac12}\\
=&P\phi_{*}\,.
\end{align*}
Notice that the latter also gives $P\gamma_{1}^{\-}N_{0}u_{\nu}=0$, i.e., 
$\gamma_{1}^{\-}N_{0}u_{\nu}=P_{\perp}\gamma_{1}^{\-}N_{0}u_{\nu}$.
Therefore, by the third line in \eqref{sys4}, 
$$
P_{\perp}\phi =
\left(\frac12+\gamma_{1}\SL_{0}
\right)_{\!\!\perp}^{-1}P_{\perp}(\phi_{*}-\gamma^{\-}_{1}N_{0}u)=
\psi_{\nu}^{\perp}-\nu c_{*}\varphi_{\nu}^{\perp}\,,
$$
where 
$$
\varphi_{\nu}^{\perp}=\left(\frac12+\gamma_{1}^{\-}\SL_{)}\right)_{\!\!\perp}^{\!-1}\gamma_{1}^{\-}N_{0}u_{\nu}
$$
and
$$
\psi_{\nu}^{\perp}=\left(\frac12+\gamma_{1}\SL_{0}
\right)_{\!\!\perp}^{-1}P_{\perp}(\phi_{*}-\gamma^{\-}_{1}N_{0}(\nu\widetilde w_{*}+v^{-2}u_{*}))\,.
$$
By \eqref{DN3}, $\varphi_{\nu}^{\perp}=-\phi_{\nu}^{\perp}$ and so
\be\label{sol-phi-perp}
P_{\perp}\phi =\psi_{\nu}^{\perp}+\nu c_{*}\phi_{\nu}^{\perp}\,.
\ee
Writing 
\be\label{c-phi}
P\phi=c_{\phi}S_{0}^{-1}1
\ee 
and considering the scalar product with $1\in L^{2}(\Omega)$ in the second line in \eqref{sys4}, one gets, using \eqref{sol-u},
\begin{align*}
\langle1,u_{*}\rangle_{L^{2}(\Omega)}=&v^{2}\langle1,(1-\nu N_{0})(\nu c_{*}u_{\nu}+\nu\widetilde w_{*}+v^{-2}u_{*})\rangle_{L^{2}(\Omega)}-v^{2}\nu c_{\phi}|\Omega|\\
&-v^{2}\nu\langle1,1_{\Omega}\SL_{0}(\psi_{\nu}^{\perp}+\nu c_{*}\phi_{\nu}^{\perp})\rangle_{L^{2}(\Omega)}\,.
\end{align*}
Combining this relation with the one in the first line of \eqref{sys4}, the couple $(c_{*},c_{\phi})$ solves the system of equations
\be\label{sol4.2}
\begin{cases}
\nu \|u_{\nu}\|^{2}_{L^{2}(\Omega)}c_{*}+[\phi_{\nu},S_{0}^{-1}1]_{-\frac12}c_{\phi}=z_{*}-\langle u_{\nu},\nu\widetilde w_{*}+v^{-2}u_{*} \rangle_{L^{2}(\Omega)}-[\phi_{\nu},P_{\perp}\phi]_{-\frac12}
\\
\nu\langle1,N_{0}u_{\nu}+1_{\Omega}\SL_{0}\phi_{\nu}^{\perp}\rangle_{L^{2}(\Omega)}\, c_{*}+|\Omega|\, c_{\phi}=
\langle 1,(1-\nu N_{0})\widetilde w_{*}-v^{-2}N_{0}u_{*}-1_{\Omega}\SL_{0}\psi_{\nu}^{\perp}\rangle_{L^{2}(\Omega)}\,.
\end{cases}
\ee
By \eqref{un}, writing $\phi_{\nu}=c_{\nu}S_{0}^{-1}1+\phi^{\perp}_{\nu}$,
$$
N_{0}u_{\nu}+1_{\Omega}\SL_{0}\phi^{\perp}_{\nu}=\nu^{-1} u_{\nu}-c_{\nu}1_{\Omega}\SL_{0}S_{0}^{-1}1=\nu^{-1} u_{\nu}-c_{\nu}1\,.
$$
Hence,
\begin{align*}
&\det\begin{bmatrix} \nu \|u_{\nu}\|^{2}_{L^{2}(\Omega)} & [\phi_{\nu},S_{0}^{-1}1]_{-\frac12}
\\
\nu\langle1,N_{0}u_{\nu}+1_{\Omega}\SL_{0}\phi_{\nu}^{\perp}\rangle_{L^{2}(\Omega)}&|\Omega|
\end{bmatrix}\\
=&\nu \left(|\Omega|\|u_{\nu}\|^{2}_{L^{2}(\Omega)}-[\phi_{\nu},S_{0}^{-1}1]_{-\frac12}\langle1,\nu^{-1} u_{\nu}-c_{\nu}1\rangle_{L^{2}(\Omega)}\right)\\
=&\nu \left(|\Omega|\|u_{\nu}\|^{2}_{L^{2}(\Omega)}+ c_{\Omega}|c_{\nu}|^{2}\right)>0
\end{align*}
and the system \eqref{sol4.2} is uniquely solvable. In conclusion, if the system \eqref{sys4} is solvable, then there is an unique solution $(\kappa,u,\phi)$ given by \eqref{kappa}, \eqref{sol-u} and by \eqref{sol-phi-perp}, \eqref{c-phi}, with $(c_{*},c_{\phi})$ solving \eqref{sol4.2}.\par
Conversely, given $(z_{*}, u_{*},\phi_{*})$ let us define $(\kappa,u,\phi)$ by \eqref{kappa}, \eqref{sol-u} and by \eqref{sol-phi-perp}, \eqref{c-phi}, with $(c_{*},c_{\phi})$ solving \eqref{sol4.2}.  Then,
$$
D_{\pm}\Phi_{4}(\kappa,u\oplus\phi)=( z_{*},(\mp2v{\nu}^{-1/2}\kappa u_{\nu}
+v^{2}(1-{\nu}N_{0})u -v^{2}{\nu}1_{\Omega}\SL_{0}\phi)\oplus\phi_{*})\,.
$$
Since 
$$
u=\nu(c_{*}u_{\nu}+\widetilde w_{*})+v^{-2}u_{*}\,,
$$
where $\widetilde w_{*}$ solves \eqref{bvp4}, one gets
$$
D_{\pm}\Phi_{4}(\kappa,u\oplus\phi)=( z_{*},(u_{*}+\widetilde u_{*})\oplus\phi_{*})\,,
$$
where 
$$
\widetilde u_{*}:=\mp2v{\nu}^{-1/2}\kappa u_{\nu}+v^{2}\nu(1-\nu N_{0})(c_{*}u_{\nu}+\widetilde w_{*})-
{\nu}N_{0}u_{*}-v^{2}{\nu}1_{\Omega}\SL_{0}\phi
\,.
$$
Thus, to conclude the proof we need to show that $\widetilde u_{*}=0$. At first, let us notice that $\widetilde u_{*}$ is harmonic. Indeed,
\begin{align*}
-\Delta_{\Omega}^{\max}\widetilde u_{*}=&\mp2v{\nu}^{1/2}\kappa u_{\nu}+v^{2}\nu^{2}c_{*}u_{\nu}-v^{2}\nu\Delta_{\Omega}^{\max}\widetilde w_{*}-
v^{2}\nu^{2}(c_{*}u_{\nu}+\widetilde w_{*})-\nu u_{*}\\
=&\mp2v{\nu}^{1/2}\kappa u_{\nu}+v^{2}\nu f_{*}-\nu u_{*}=0\,.
\end{align*}
Furthermore, by $\gamma_{1}^{\-}\widetilde w_{*}=\phi_{*}$ and  by $[\gamma_{1}]\SL_{0}\phi=-\phi$, 
\begin{align*}
\gamma_{1}^{\-}\widetilde u_{*}=&v^{2}\nu\big(\phi_{*}-\gamma_{1}^{\-} N_{0}u-\gamma_{1}^{\-}\SL_{0}\phi\big)\\
=&v^{2}\nu\left(\phi_{*}-\gamma_{1}^{\-} N_{0}u-\frac12\,\phi-\frac12\,[\gamma_{1}]\SL_{0}\phi-\gamma_{1}^{\-}\SL_{0}\phi\right)
\\
=&v^{2}\nu\left(\phi_{*}-\left(\gamma_{1}^{\-} N_{0}u+\frac12\,\phi+\gamma_{1}\SL_{0}\phi\right)\right)\\
=&0\,.
\end{align*}
Therefore, $\widetilde u_{*}$ is a constant; let us set $\widetilde u_{*}=\widetilde c_{*}$ and write $P\phi=c_{\phi}S_{0}^{-1}1$. Then, by the second line in \eqref{sol4.2}
\begin{align*}
\widetilde c_{*}|\Omega|=&\langle 1,\widetilde u_{*}\rangle_{L^{2}(\Omega)}
\\
=&v^{2}\nu\big(\langle 1,\widetilde w_{*}\rangle_{L^{2}(\Omega)}-\langle 1,N_{0}(\nu(c_{*}u_{\nu}+\widetilde w_{*})+v^{-2}u_{*})\rangle_{L^{2}(\Omega)}-\langle 1,1_{\Omega}\SL_{0}\phi\rangle_{L^{2}(\Omega)}\big)\\
=&v^{2}\nu\big(\langle 1,(1-\nu N_{0})\widetilde w_{*}-v^{-2}N_{0}u_{*}-1_{\Omega}\SL_{0}\psi_{\nu}^{\perp}\rangle_{L^{2}(\Omega)}-\nu c_{*}\langle 1,N_{0}u_{\nu}+1_{\Omega}\SL_{0}\phi_{\nu}^{\perp}\rangle_{L^{2}(\Omega)}\big)\\
=&0\,.
\end{align*}
Thus $\widetilde u_{*}=0$ and the proof is done.
\end{proof}
\begin{remark}
Let us notice that our hypothesis on the simplicity of the Neumann eigenvalue $\nu$, holds true for a generic domain $\Omega$, see \cite[Example 6.4]{Henry}.
\end{remark}
By the previous Lemma and by the implicit function argument recalled at the beginning of the Section, in a sufficiently small neighborhood of $0\in\RE$, there is a $\CO\times (\widehat X\backslash\{0\})$-valued, analytic  function $\ve\mapsto (\kappa_{\pm}(\ve), u_{\pm} (\ve)\oplus\phi_{\pm} (\ve))$,
$$
\kappa_{\pm}(\ve)=\pm v\sqrt\nu+\sum_{n=1}^{+\infty}\kappa_{\pm}^{(n)}\ve^{n}
\,,\qquad
u_{\pm} (\ve)=u_{\nu}+\sum_{n=1}^{+\infty} u_{\pm} ^{(n)}\ve^{n}\,,\qquad
\phi_{\pm} (\ve)=\phi_{\nu}+\sum_{n=1}^{+\infty}\phi_{\pm} ^{(n)}\ve^{n}
$$
which uniquely solves ${F}_{4}(\ve,\kappa(\ve), u(\ve)\oplus\phi(\ve))=0$. One can explicitly compute all the coefficients appearing in the expansions. In particular,
 $$
\big(\kappa_{\pm}^{(1)},u_{\pm} ^{(1)}\oplus\phi_{\pm} ^{(1)}\big)=
-\left(D_{\pm}\Phi_{4}\right)^{-1}
\left(
\left.\frac{d\ }{d\ve}\right|_{\ve=0}\Phi_{4}(\ve,\pm v\nu^{1/2},u_{\nu}\oplus\phi_{\nu})
\right)\,,
$$
equivalently,
$$
D_{\pm}\Phi_{4}(\kappa_{\pm}^{(1)},u ^{(1)}\oplus\phi ^{(1)})=\big(0,
\widehat u_{\nu}\oplus\widehat\phi_{\nu}\big)\,,
$$
where
$$
\begin{bmatrix}\widehat u_{\nu}\\\widehat\phi_{\nu}\,\end{bmatrix}:=
-
\begin{bmatrix}v_{1}^{2}\mp v^{3}\nu^{3/2}N_{(1)}&\mp v^{3}\nu^{3/2}1_{\Omega}\SL_{(1)}\,
\\
-2\rho\,\gamma^{\-}_{1}N_{0}&-2\rho\, \gamma_{1}\SL_{0}\end{bmatrix}
\begin{bmatrix} u_{\nu}\\ \phi_{\nu}\end{bmatrix}.
$$
Hence, by \eqref{kappa},
$$
\kappa_{\pm}^{(1)}=\pm\,\frac{1}2\,{\nu}^{1/2}\|u_{\nu}\|^{-2}_{L^{2}(\Omega)}\left(v^{-1}\langle u_{\nu},\widehat u_{\nu}\rangle_{L^{2}(\Omega)}+v\,\langle\gamma_{0}^{\-}u_{\nu},\widehat \phi_{\nu}\rangle_{\frac12,-\frac12}\right).
$$  
By
$$
N_{(1)}u=\frac{i}{4\pi}\,\langle 1,u\rangle_{L^{2}(\Omega)}1\,,\qquad 
\SL_{(1)}\phi=\frac{i}{4\pi}\,\langle 1,\phi\rangle_{\frac12,-\frac12}1\,,
$$
there follows
\begin{align*}
&v^{-1}\langle u_{\nu},\widehat u_{\nu}\rangle_{L^{2}(\Omega)}+v\,\langle\gamma_{0}^{\-}u_{\nu},\widehat \phi_{\nu}\rangle_{\frac12,-\frac12}\\
=&-v_{1}^{2}v^{-1}\|u_{\nu}\|^{2}_{L^{2}(\Omega)}-2\rho v\langle\gamma_{0}^{\-}u_{\nu},\gamma^{\-}_{1}N_{0}u_{\nu}\rangle_{\frac12,-\frac12}
-2\rho v\langle\gamma_{0}^{\-}u_{\nu},\gamma_{1}\SL_{0}\phi_{\nu}
\rangle_{\frac12,-\frac12}\,.
\end{align*}
Since both $(\nu^{-1}-N_{0})u_{\nu}=0$ and $1_{\Omega}\SL_{0}\phi_{\nu}$ are harmonic and 
$$
\gamma_{0}^{\-}\SL_{0}\phi_{\nu}=\gamma^{\-}_{0}\SL_{0}S_{0}^{-1}\gamma^{\-}_{0}(\nu^{-1}-N_{0})u_{\nu}=\gamma^{\-}_{0}(\nu^{-1}-N_{0})u_{\nu}\,,
$$ 
one gets 
$$
1_{\Omega}\SL_{0}\phi_{\nu}=(\nu^{-1}-N_{0})u_{\nu}\,.
$$
Furthermore, 
$$
\gamma_{1}\SL_{0}\phi_{\nu}=\frac12(\gamma_{1}^{\-}+\gamma_{1}^{\+})\SL_{0}\phi_{\nu}=\gamma_{1}^{\-}\SL_{0}\phi_{\nu}+\frac12\,[\gamma_{1}]\SL_{0}\phi_{\nu}=\gamma_{1}^{\-}\SL_{0}\phi_{\nu}-\frac12\,\phi_{\nu}\,.
$$
Green's formula gives
\begin{align*}
\langle\gamma_{0}^{\-}u_{\nu},\gamma_{1}^{\-}N_{0}u_{\nu}\rangle_{\frac12,-\frac12}=&
\langle u_{\nu},\Delta_{\Omega} N_{0}u_{\nu}\rangle_{L^{2}(\Omega)}-\langle \Delta_{\Omega} u_{\nu},N_{0}u_{\nu}\rangle_{L^{2}(\Omega)}\\
=&-\| u_{\nu}\|^{2}_{L^{2}(\Omega)}+\nu\langle  u_{\nu},N_{0}u_{\nu}\rangle_{L^{2}(\Omega)}\,,
\end{align*}
\begin{align*}
&\langle\gamma_{0}^{\-}u_{\nu},\gamma_{1}^{\-}\SL_{0}\phi_{\nu}\rangle_{\frac12,-\frac12}=
\langle u_{\nu},\Delta_{\Omega} 1_{\Omega}\SL_{0}\phi_{\nu}\rangle_{L^{2}(\Omega)}-\langle \Delta_{\Omega} u_{\nu},1_{\Omega}\SL_{0}\phi_{\nu}\rangle_{L^{2}(\Omega)}\\
=&\nu\langle u_{\nu},1_{\Omega}\SL_{0}\phi_{\nu}\rangle_{L^{2}(\Omega)}=
\langle u_{\nu},(1-\nu N_{0})u_{\nu}\rangle_{L^{2}(\Omega)}=
\|u_{\nu}\|^{2}_{L^{2}(\Omega)}-\nu \langle u_{\nu}, N_{0}u_{\nu}\rangle_{L^{2}(\Omega)}\,.
\end{align*}
Thus
$$
\kappa_{\pm}^{(1)}=\pm\,\frac{1}2\,{\nu}^{1/2}\|u_{\nu}\|^{-2}_{L^{2}(\Omega)}\left(\rho v\,\langle\gamma_{0}^{\-}u_{\nu}, \phi_{\nu}\rangle_{\frac12,-\frac12}-v_{1}^{2}v^{-1}\|u_{\nu}\|^{2}_{L^{2}(\Omega)}\right).
$$ 
By \eqref{phinu}, $\kappa_{\pm}^{(1)}$ does not depend on $\|u_{\nu}\|_{L^{2}(\Omega)}$; hence, by taking $\|u_{\nu}\|_{L^{2}(\Omega)}=1$, 
$$
\kappa_{\pm}^{(1)}=\pm\,\frac{1}2\,{\nu}^{1/2}\left(\rho v\,\langle\gamma_{0}^{\-}u_{\nu}, \phi_{\nu}\rangle_{\frac12,-\frac12}-v_{1}^{2}v^{-1}\right).
$$
and, writing $\kappa_{\pm}^{(1)}=\pm\,\kappa_{\nu}$, $\kappa_{\nu}\in\RE$,
$$
\kappa_{\pm}(\ve)=\pm (v\nu^{1/2}+\kappa_{\nu}\ve) +O(\ve^{2})\,;
$$
then
$$
\kappa^{2}_{\pm}(\ve)=v^{2}\nu+ 2v\nu^{1/2} \kappa_{\nu}\ve+O(\ve^{2})
$$
belongs to $\mathsf e(A(\ve))$.\par
We conclude by considering the case $\nu=0$, where the implicit function argument does not apply directly and a Lyapunov-Schmidt reduction (see, e.g., \cite[Sections 3 and 4, Chapter 5]{AP}) has to be used.\par
Let us denote by 
$$
D_{0}\Phi_{4}: \CO\times  \widehat X
\to \CO\times \widehat X
$$
the Frechet  differential at $(0,0,0\oplus c_{0}S_{0}^{-1}1)$ (see Remark \eqref{k=0}) of the map $(\kappa,u\oplus\phi)\mapsto\Phi_{4}(0,\kappa,u\oplus\phi)$, i.e., writing $P\phi=c_{\phi}S_{0}^{-1}1$,
\begin{align*}
D_{0}\Phi_{4}(\kappa,u\oplus\phi)=\left(c_{\Omega}\overline c_{0}c_{\phi},
v^{2}u\oplus\left(\gamma^{\-}_{1}N_{0}u +\left(\frac12+\gamma_{1}\SL_{0}\right)P_{\perp}\phi\right)
\right).
\end{align*}
One has 
$$
\ker(D_{0}\Phi_{4})=\{(\kappa,u\oplus\phi):u=0\,,\ \phi=0\}\simeq\CO\,.
$$
Writing $$
\phi_{*}=c_{\phi_{*}}S_{0}^{-1}1+\phi_{*}^{\perp}\,,\qquad u=c_{u_{*}}1+u_{*}^{\perp}\,,
$$ 
the system
$$
\begin{cases}
c_{\Omega}\overline c_{0}c_{\phi}=z_{*}\\
v^{2}u=u_{*}\\
\gamma^{\-}_{1}N_{0}u +\left(\frac12+\gamma_{1}\SL_{0}\right)P_{\perp}\phi=\phi_{*}
\end{cases}
$$
is solvable if and only if
$$
P(v^{2}\phi_{*}-\gamma^{\-}_{1}N_{0}u_{*})=(v^{2}c_{\phi_{*}}+c_{\Omega}^{-1}|\Omega|c_{u_{*}})S_{0}^{-1}1=0
$$
and in that case the solution is given by
\be\label{sol-LS}
\begin{cases}
c_{\phi}=c_{\Omega}^{-1}{\overline c}_{0}^{\,-1}z_{*}\\
u=v^{-2}u_{*}\\
P_{\perp}\phi=\left(\frac12+\gamma_{1}\SL_{0}
\right)_{\!\!\perp}^{-1}(\phi_{*}-v^{-2}\gamma^{\-}_{1}N_{0}u_{*})\,.
\end{cases}
\ee
This entails
$$
\ran(D_{0}\Phi_{4})=\CO\oplus \widehat X_{0}\,,
$$
where
$$\widehat X_{0}:=\big\{u_{*}\oplus\phi_{*}\in\widehat X: \langle 1,u_{*}\rangle_{L^{2}(\Omega)}+v^{2}[S_{0}^{-1}1,\phi_{*}]_{-\frac12}=0\big\}=\{1\oplus v^{2}S_{0}^{-1}1\}^{\perp}\,.
$$
Therefore, denoting by $\widehat P$ and by $\widehat P_{\perp}$ the orthogonal projectors onto the closed subspaces $\CO\oplus\widehat X_{0}$ and $(\CO\oplus\widehat X_{0})^{\perp}=\{0\}\oplus\text{span}(1\oplus v^{2}S_{0}^{-1}1)$ respectively, one has $\widehat P=1-\widehat P_{\perp}$ with
$$
\widehat P_{\perp}(z_{*},u_{*}\oplus\phi_{*}):=\left(0,(|\Omega|+v^{4}c_{\Omega})^{-1}\left(\langle 1,u_{*}\rangle_{L^{2}(\Omega)}+ v^{2}\langle1,\phi_{*}\rangle_{\frac12,-\frac12}\right)(1\oplus v^{2}S_{0}^{-1}1)\right)\,.
$$
Obviously,
\be\label{iff}
\Phi_{4}(\ve,\kappa,u\oplus\phi)=0\quad\Leftrightarrow\quad\begin{cases}
\widehat P\Phi_{4}(\ve,\kappa,u\oplus\phi)=0\\
\widehat P_{\perp}\Phi_{4}(\ve,\kappa,u\oplus\phi)=0\,.
\end{cases}
\ee
Defining 
$$
\widehat\Phi_{4}:U\times\CO\times  \widehat {X}\to \CO\times\CO\times  \widehat X_{0}\,,\qquad 
\widehat\Phi_{4}:=\widehat P\Phi_{4}\,,
$$
and denoting by $\widehat D_{0}\widehat\Phi_{4}$ the Frechet  differential at $(0,0,0\oplus c_{0}S_{0}^{-1}1)$  of the map $u\oplus\phi\mapsto\Psi(0,0,u\oplus\phi)$, one gets 
$$
\widehat D_{0}\widehat\Phi_{4}=\widehat PD_{0}\Phi_{4}\,.
$$
By \eqref{sol-LS}, $\widehat D_{0}\widehat\Phi_{4}:\widehat X\to \CO\times  \widehat X_{0}$ is a bounded bijection, and 
$$
\widehat D_{0}\widehat\Phi_{4}(u\oplus\phi)=(z_{*},u_{*}\oplus\phi_{*})\quad\Leftrightarrow\quad\begin{cases}
u=v^{-2}u_{*}\\
\phi=(c_{\Omega}\overline c_{0})^{-1}z_{*}S_{0}^{-1}1+\left(\frac12+\gamma_{1}\SL_{0}
\right)_{\!\!\perp}^{-1}(\phi_{*}-v^{-2}\gamma^{\-}_{1}N_{0}u_{*})\,.
\end{cases}
$$
Since 
\begin{align*}
\widehat\Phi_{4}(0,0,0\oplus c_{0}S_{0}^{-1}1)=\widehat P\Phi_{4}(0,0,0\oplus c_{0}S_{0}^{-1}1)=0\,,
\end{align*}
the implicit function theorem applies to $\widehat\Phi_{4}$ and therefore, in a sufficiently small neighborhood of $(0,0)\in\RE\times\CO$, there exists a $\widehat X_{0}$-valued analytic map $(\ve,\kappa)\mapsto u(\ve,\kappa)\oplus\phi(\ve,\kappa)$ such that 
$$
\widehat P\Phi_{4}(\ve,\kappa,u(\ve,\kappa)\oplus\phi(\ve,\kappa))=0\,,
$$
equivalently, 
\be\label{ftl.0}
\Phi_{4}(\ve,\kappa,u(\ve,\kappa)\oplus\phi(\ve,\kappa))=\widehat P_{\perp}\Phi_{4}(\ve,\kappa,u(\ve,\kappa)\oplus\phi(\ve,\kappa))\,.
\ee
Defining 
\be\label{sf4}
{f}(\ve,\kappa):=\langle 1\oplus v^{2}S_{0}^{-1}1, F_{4}(\ve,\kappa,u(\ve,\kappa)\oplus\phi(\ve,\kappa))\rangle_{\widehat X}\,,
\ee
and using the relations
$$
\begin{cases}
\langle u(\ve,\kappa)\oplus\phi(\ve,\kappa), u(\ve,\kappa)\oplus(\phi(\ve,\kappa)-c_{0}S_{0}^{-1}1)\rangle_{\widehat X}=0
\\
\langle u(\ve,\kappa)\oplus\phi(\ve,\kappa), 1\oplus v^{2}S_{0}^{-1}1\rangle_{\widehat X}=0\,,
\end{cases}
$$
\eqref{ftl.0} rewrites as
\be
\begin{cases}\label{ftl}
\langle 1,u(\ve,\kappa)\rangle_{L^{2}(\Omega)}=-v^{2}c_{0}c_{\Omega}\\
[S_{0}^{-1}1,\phi(\ve,\kappa)]_{-\frac12}\equiv\langle 1,\phi(\ve,\kappa)\rangle_{\frac12,-\frac12}=c_{0}c_{\Omega}\\
F_{4}(\ve,\kappa,u(\ve,\kappa)\oplus\phi(\ve,\kappa))=(|\Omega|+v^{2}c_{\Omega})^{-1}
{f}(\ve,\kappa)(1\oplus v^{2}S_{0}^{-1}1)\,.
\end{cases}
\ee
Moreover, by \eqref{iff} and \eqref{sf4},
$$
\Phi_{4}(\ve,\kappa,u(\ve,\kappa)\oplus\phi(\ve,\kappa))=0\ \Leftrightarrow\ 
\widehat P_{\perp}\Phi_{4}(\ve,\kappa,u(\ve,\kappa)\oplus\phi(\ve,\kappa))=0
\ \Leftrightarrow\ {f}(\ve,\kappa)=0
\,.
$$
In the following we use the shorthand notation
$$
\partial^{m}_{\ve}\partial_{\kappa}^{n} {f}_{0}:=\partial^{m}_{\ve}\partial_{\kappa}^{n} {f}(0,0)\,,\qquad 
\partial^{m}_{\ve}\partial_{\kappa}^{n} u_{0}:=\partial^{m}_{\ve}\partial_{\kappa}^{n} u(0,0)\,,\qquad 
\partial^{m}_{\ve}\partial_{\kappa}^{n}\phi_{0}:=\partial^{m}_{\ve}\partial_{\kappa}^{n}\phi(0,0)\,,
$$
Since $F_{4}$ is a function of $\kappa^{2}$, $\partial_{k}\Psi(0,0,0\oplus c_{0}S_{0}^{-1}1)=0$ and so, 
$$
\partial_{\kappa}u_{0}\oplus\partial_{k}\phi_{0}=-(\widehat D_{0}\Psi)^{-1}\partial_{k}\Psi(0,0,0\oplus c_{0}S_{0}^{-1}1)=0\oplus 0\,.
$$
By the first two lines in \eqref{ftl}),
$$
\left\langle1,\partial^{m}_{\ve}\partial_{\kappa}^{n} u_{0}\right\rangle_{L^{2}(\Omega)}=0\,,\qquad
\left[S_{0}^{-1}1,\partial^{m}_{\ve}\partial_{\kappa}^{n}\phi_{0}\right]_{-\frac12}=0\,,
$$
One has $f_{0}:=f(0,0)=0$ and, by 
\begin{align*}
&\langle1,u\rangle_{L^{2}(\Omega)}+
\left[S_{0}^{-1}1,\gamma_{1}^{\-}N_{0}u+
\left(\frac12+\gamma_{1}\SL_{0}\right)\phi\right]_{-\frac12}
\\
=&\langle 1,u\rangle_{L^{2}(\Omega)}+\langle 1, \gamma_{1}^{\-}N_{0}u\rangle_{\frac12,-\frac12}
=\langle 1,u\rangle_{L^{2}(\Omega)}-\langle 1,u\rangle_{L^{2}(\Omega)}
=0\,,
\end{align*}
which holds for any $u\oplus \phi\in \widehat X$, 
 \begin{align*}
\partial_{\kappa} {f}_{0}=&
v^{2}\left\langle1,\partial_{\kappa} u_{0}\right\rangle_{L^{2}(\Omega)}+v^{2}
\left[S_{0}^{-1}1,\left(\gamma_{1}^{\-}N_{0}\,\partial_{\kappa} u_{0}+\left(\frac12+\gamma_{1}\SL_{0}\right)\partial_{\kappa} \phi_{0}\right)\right]_{-\frac12}
=0\,,
\end{align*}
\begin{align*}
\frac12\,\partial^{2}_{\kappa} {f}_{0}=&
\frac{v^{2}}2\,\left\langle1,\partial^{2}_{\kappa} u_{0}\right\rangle_{L^{2}(\Omega)}+\frac{v^{2}}2\,
\left[S_{0}^{-1}1,\left(\gamma_{1}^{\-}N_{0}\partial^{2}_{\kappa} u_{0}+\left(\frac12+\gamma_{1}\SL_{0}\right)\partial^{2}_{\kappa} \phi_{0}\right)\right]_{-\frac12}
\\
&-c_{0}\langle 1,1_{\Omega}\SL_{0}S_{0}^{-1}1\rangle_{L^{2}(\Omega)}
=-c_{0}|\Omega|
\,,
\end{align*}
\begin{align*}
\frac16\,\partial^{3}_{\kappa} {f}_{0}=&
\frac{v^{2}}6\,\left\langle1,\partial^{3}_{\kappa} u_{0}\right\rangle_{L^{2}(\Omega)}+\frac{v^{2}}6\,
\left[S_{0}^{-1}1,\left(\gamma_{1}^{\-}N_{0}\partial^{3}_{\kappa} u_{0}+\left(\frac12+\gamma_{1}\SL_{0}\right)\partial^{3}_{\kappa} \phi_{0}\right)\right]_{-\frac12}
\\
&-\frac13\,\left\langle1,N_{0}\partial_{\kappa} u_{0}+1_{\Omega}\SL_{0}
\partial_{\kappa} \phi_{0}\right\rangle_{L^{2}(\Omega)}=0\,,
\end{align*}
\begin{align*}
\partial_{\ve} {f}_{0}=& 
v^{2}\left\langle1,\partial_{\ve} u_{0}\right\rangle_{L^{2}(\Omega)}+v^{2}
\left[S_{0}^{-1}1,\left(\gamma_{1}^{\-}N_{0}\partial_{\ve} u_{0}+\left(\frac12+\gamma_{1}\SL_{0}\right)\partial_{\ve} \phi_{0}\right)\right]_{-\frac12}
\\
&-2\rho c_{0}[S_{0}^{-1}1,\gamma_{1}\SL_{0}S_{0}^{-1}1]_{-\frac12}
=\rho c_{0}c_{\Omega}\,,
\end{align*}
\begin{align*}
\partial_{\ve}\partial_{\kappa}{f}_{0}=&
v^{2}\left\langle1,\partial_{\ve}\partial_{\kappa}u_{0}\right\rangle_{L^{2}(\Omega)}+v^{2}
\left[S_{0}^{-1}1,\left(\gamma_{1}^{\-}N_{0}\partial_{\ve}\partial_{\kappa}u_{0}+\left(\frac12+\gamma_{1}\SL_{0}\right)\partial_{\ve}\partial_{\kappa}\phi_{0}\right)\right]_{-\frac12}
\\
&+v_{1}^{2}\left\langle1,\partial_{\kappa}u_{0}\right\rangle_{L^{2}(\Omega)}
-2\rho v^{2}\left[S_{0}^{-1}1,\left(\gamma_{1}^{\-}N_{0}\partial_{\kappa}u_{0}+\gamma_{1}\SL_{0}\partial_{\kappa}\phi_{0}\right)\right]_{-\frac12}=0
\,,
\end{align*}
\begin{align*}
\frac12\,\partial_{\ve}\partial^{2}_{\kappa}{f}_{0}=&
\frac{v^{2}}2\left\langle1,\partial_{\ve}\partial^{2}_{\kappa}u_{0}\right\rangle_{L^{2}(\Omega)}+\frac{v^{2}}2
\left[S_{0}^{-1}1,\left(\gamma_{1}^{\-}N_{0}\partial_{\ve}\partial_{\kappa}u_{0}+\left(\frac12+\gamma_{1}\SL_{0}\right)\partial_{\ve}\partial^{2}_{\kappa}\phi_{0}\right)\right]_{-\frac12}
\\
&-\frac12\,\left\langle1,N_{0}\partial_{\ve}u_{0}+1_{\Omega}\SL_{0}\partial_{\ve}\phi_{0}+
v_{1}^{2}\partial^{2}_{\kappa}u_{0}\right\rangle_{L^{2}(\Omega)}\\
&-\rho v^{2}\left[S_{0}^{-1}1,\left(\gamma_{1}^{\-}N_{0}\partial^{2}_{\kappa}u_{0}+\gamma_{1}\SL_{0}\partial^{2}_{\kappa}\phi_{0}\right)\right]_{-\frac12}\\
=&-\frac12\,\left\langle1,N_{0}\partial_{\ve}u_{0}+1_{\Omega}\SL_{0}\partial_{\ve}\phi_{0}+
v_{1}^{2}\partial^{2}_{\kappa}u_{0}\right\rangle_{L^{2}(\Omega)}
+\frac12\,\rho v^{2}\left[S_{0}^{-1}1,\partial^{2}_{k}\phi_{0}\right]_{-\frac12}\\
=&-\frac12\,\left\langle1,N_{0}\partial_{\ve}u_{0}+1_{\Omega}\SL_{0}\partial_{\ve}\phi_{0}\right\rangle_{L^{2}(\Omega)}
\,.
\end{align*}
\begin{remark} To determine $\partial_{\ve}u_{0}\oplus\partial_{\ve}\phi_{0}$ one uses the relation  
$\Psi(\ve,\kappa,u(\ve,\kappa),\phi(\ve,\kappa)=0$ which gives 
\begin{align*}
&\widehat D_{0}\Psi(\partial_{\ve}u_{0}\oplus\partial_{\ve}\phi_{0})
=-\partial_{\ve}\Psi(0,0,0\oplus c_{0}S_{0}^{-1}1)
=-\widehat P(0,\partial_{\ve}F_{4}(0,0,0\oplus c_{0}S_{0}^{-1}1))\\
=&
2\rho\widehat P(0,0\oplus c_{0}\gamma_{1}\SL_{0}S_{0}^{-1}1)=2\rho\big((0,0\oplus c_{0}\gamma_{1}\SL_{0}S_{0}^{-1}1)-\widehat P_{\perp}(0,0\oplus c_{0}\gamma_{1}\SL_{0}S_{0}^{-1}1)\big)
\\
=&
2\rho c_{0}\left(0, 0\oplus \gamma_{1}\SL_{0}S_{0}^{-1}1-
\frac{v^{2}c_{\Omega}}{|\Omega|+v^{4}c_{\Omega}}\,(1\oplus v^{2}S_{0}^{-1}1)
\right).
\end{align*}
Hence,
$$
\partial_{\ve}u_{0}=-
\frac{2\rho c_{0}c_{\Omega}}{|\Omega|+v^{4}c_{\Omega}}\,1=
-
\frac{2\rho c_{0}\omega^{2}_{M}}{1+v^{4}\omega_{M}^{2}}\,1
$$
$$
\partial_{\ve}\phi_{0}=-2\rho c_{0}\left(\frac12+\gamma_{1}\SL_{0}
\right)_{\!\!\perp}^{-1}\left(\gamma_{1}\SL_{0}S_{0}^{-1}1-
\frac{\omega^{2}_{M}}{1+v^{4}\omega_{M}^{2}}
\left(v^{4}S_{0}^{-1}1-\gamma^{\-}_{1}N_{0}1\right)\right).
$$
\end{remark}
Now, we introduce the analytic function 
\begin{align*}
&{g}(t,z):=t^{-2}{f}(t^{2},tz)={f}_{0}\,t^{-2}+\partial_{\ve} {f}_{0}+\partial_{\kappa} {f}_{0}\,t^{-1}z+\frac12 \,\partial_{\ve}^{2} {f}_{0}\,t^{2}\\
&+\partial_{\ve}\partial_{\kappa} {f}_{0}\,tz+\frac12\, \partial^{2}_{\kappa} {f}_{0}\,z^{2}+\frac16\,\partial^{3}_{\ve} {f}_{0}t^{4}
+\frac12\,\partial^{2}_{\ve}\partial_{\kappa} {f}_{0}\,t^{3}z
\\
&+
\frac12\,\partial_{\ve}\partial^{2}_{\kappa} {f}_{0}\,t^{2}z^{2}+
\frac16\,\partial^{3}_{\kappa} {f}_{0}\,tz^{3}+
O(t^{2}(t^{2}+|z|^{2})^{2})\\
=&\rho c_{0}c_{\Omega}-c_{0}|\Omega|z^{2}
+
\frac12\,\partial_{\ve}\partial^{2}_{\kappa} {f}_{0}\,t^{2}z^{2}
+O(t^{4})+O(|tz|^{3})+
O(t^{2}(t^{2}+|z|^{2})^{2})\,.
\end{align*}
Hence, by
$$
{g}(0,z)=\rho c_{0}c_{\Omega}
-c_{0}|\Omega|z^{2}\,,
$$
one gets 
$$
{g}(0,\pm\sqrt\rho\,\omega_{M})=0\,.
$$
Then
$$
\partial_{z} {g}(0,\pm\sqrt\rho\,\omega_{M})=\mp 2c_{0}|\Omega|\sqrt\rho\,\omega_{M}\not=0\,,
$$
$$
\partial_{t} {g}(0,\pm\sqrt\rho\,\omega_{M})=0\,,\qquad
\partial^{2}_{t} {g}(0,\pm\sqrt\rho\,\omega_{M})=\partial_{\ve}\partial^{2}_{\kappa} {f}_{0}\,\rho\,\omega_{M}^{2}
\,.
$$
By the implicit function theorem, there exists an unique analytic function   $t\mapsto z_{\pm}(t)$, in a neighborhood of $0\in\RE$, 
\begin{align*}
z_{\pm}(t)=&\pm\sqrt\rho\,\omega_{M}-\frac12\left(\partial_{z} {g}(0,\pm\sqrt\rho\,\omega_{M})\right)^{-1}\partial^{2}_{t} {g}(0,\pm\sqrt\rho\,\omega_{M})\,t^{2}+O(t^{3})\\
=&\pm\sqrt\rho\,\omega_{M}\left(1+\frac{\partial_{\ve}\partial^{2}_{\kappa} {f}_{0}}{4 c_{0}|\Omega|}\ t^{2}\right)+O(t^{3})
\end{align*}
such that ${g}(t,z_{\pm}(t))=0$. Therefore, for $\ve>0$ sufficiently small,
$$
\kappa_{\pm}(\ve)=\sqrt\ve\,z_{\pm}(\sqrt\ve)=
\pm\sqrt\rho\,\omega_{M}\left(1+\frac{\partial_{\ve}\partial^{2}_{\kappa} {f}_{0}}{4 c_{0}|\Omega|}\ \ve\right)\sqrt\ve+O(\ve^{2})
$$
uniquely solves 
$$
F_{4}(\ve,\kappa(\ve),u(\ve,\kappa(\ve))\oplus\phi(\ve,\kappa(\ve)))=0
$$
and  $u(\ve,\kappa(\ve))\oplus\phi(\ve,\kappa(\ve))\not=0$ by \eqref{ftl}. Hence,
$$
\kappa^{2}_{\pm}(\ve)=
\rho\,\omega_{M}\left(1+\frac{\partial_{\ve}\partial^{2}_{\kappa} {f}_{0}}{2 c_{0}|\Omega|}\ \ve\right)\ve+O(\ve^{5/2})
$$
belongs to $\mathsf e(A(\ve))$. 

\vskip15pt\noindent
\textbf{Acknowledgements.} The authors acknowledge funding from the Next Generation EU-Prin project 2022 ''Singular Interactions and Effective Models in Mathematical Physics''. A.P.  acknowledge the support of the National Group of Mathematical Physics (GNFM-INdAM).

\end{document}